\documentclass[a4paper,12pt,reqno]{amsart}

\usepackage[mathscr]{euscript}
\usepackage{amssymb, amscd,enumerate,enumitem,color}
\usepackage[utf8]{inputenc}
\makeatletter
\def\subsection{\@startsection{subsection}{3}%
  \z@{.5\linespacing\@plus.7\linespacing}{.1\linespacing}%
  {\normalfont\bfseries}}
\makeatother

\font\smallsmc = cmcsc9
\font\smalltt = cmtt8
\font\smallit = cmti8

\font\smc = cmcsc10 at 12 pt

\newcommand{\bss}{\fontfamily{cmr}\fontseries{b}\fontshape{n}\selectfont}

\numberwithin{equation}{section}

\theoremstyle{plain}
\newtheorem{theo}{Theorem}[section]
\newtheorem{lem}[theo]{Lemma}
\newtheorem{prop}[theo]{Proposition}

\newtheorem{cor}[theo]{Corollary}

\newtheorem{lemma}[theo]{Lemma}

\theoremstyle{definition}
\newtheorem{rem}[theo]{Remark}

\newtheorem{definition}[theo]{Definition}

\theoremstyle{plain}
\newtheorem{Athm}{Theorem A\!\!}
\newtheorem{Aprop}[Athm]{Proposition A\!\!}
\newtheorem{Alem}[Athm]{Lemma A\!\!}

\theoremstyle{definition}
\newtheorem{Arem}[Athm]{Remark A\!\!}
\newtheorem{Adef}[Athm]{Definition A\!\!}

\newcommand{\Aref}{A\ref}

\newcommand{\rank}{\operatorname{rank}}
\newcommand{\beq}{\begin{equation}}
\newcommand{\eeq}{\end{equation}}
\renewcommand{\a}{\alpha}
\renewcommand{\b}{\beta}

\renewcommand{\d}{\delta}

\newcommand{\f}{\varphi}
\newcommand{\g}{\gamma}
\newcommand{\h}{\eta}

\renewcommand{\j}{\psi}
\renewcommand{\k}{\kappa}
\renewcommand{\l}{\lambda}
\renewcommand{\o}{\omega}
\newcommand{\q}{\vartheta}
\renewcommand{\r}{\rho}
\newcommand{\s}{\sigma}
\renewcommand{\t}{\tau}

\newcommand{\z}{\zeta}

\newcommand{\D}{\Delta}

\renewcommand{\L}{\Lambda}

\newcommand{\bC}{\mathbb{C}}
\newcommand{\bR}{\mathbb{R}}
\newcommand{\bZ}{\mathbb{Z}}

\newcommand{\bH}{\mathbb{H}}

\newcommand{\bN}{\mathbb{N}}

\newcommand{\bJ}{\mathbb{J}}

\renewcommand{\gg}{\mathfrak{g}}

\newcommand{\gp}{\mathfrak{p}}
\newcommand{\gq}{\mathfrak{q}}

\newcommand{\gF}{\mathfrak{F}}
\newcommand{\gJ}{\mathfrak{J}}
\newcommand{\gX}{\mathfrak{X}}

\newcommand{\gH}{\mathfrak{H}}

\newcommand{\gV}{\mathfrak{V}}

\newcommand{\gsp}{\mathfrak{sp}}

\newcommand{\ggl}{\mathfrak{gl}}

\newcommand\GL{\mathrm{GL}}
\newcommand\SL{\mathrm{SL}}

\newcommand\Spin{\mathrm{Spin}}
\newcommand\Sp{\mathrm{Sp}}

\newcommand{\cA}{\mathscr{A}}
\newcommand{\cB}{B}
\newcommand{\cC}{\mathcal{C}}
\newcommand{\cD}{\mathscr{D}}

\newcommand{\cH}{\mathscr{H}}

\newcommand{\cL}{\mathscr{L}}

\newcommand{\cN}{\mathscr{N}}

\newcommand{\cP}{\mathscr{P}}

\newcommand{\cS}{\mathscr{S}}
\newcommand{\cT}{\mathscr{T}}
\newcommand{\cU}{\mathscr{U}}
\newcommand{\cV}{\mathscr{V}}
\newcommand{\cW}{\mathscr{W}}
\newcommand{\cX}{\mathscr{X}}

\newcommand{\cZ}{\mathscr{Z}}

\newcommand\fr[2]{\tfrac{#1}{#2}}
\newcommand{\p}{\partial}
\newcommand{\ra}{\rightarrow}

\DeclareMathOperator\End{End\;}

\DeclareMathOperator\Ad{Ad}

\DeclareMathOperator\Id{Id}

\DeclareMathOperator{\Span}{span}
\DeclareMathOperator{\ad}{ad}
\DeclareMathOperator{\Diff}{Diff}
\DeclareMathOperator{\Lie}{Lie}

\newcommand\Hom{\operatorname{Hom}}
\renewcommand\Re{\operatorname{Re}}
\renewcommand\Im{\operatorname{Im}}

\newcommand{\wt}{\widetilde}
\newcommand{\wh}{\widehat}
\newcommand{\ol}{\overline}

\newcommand{\bt}{\begin{theo}\ \ }
\newcommand{\et}{\end{theo}}
\newcommand{\bp}{\begin{prop}\ \ }
\newcommand{\ep}{\end{prop}}
\newcommand{\bc}{\begin{cor}\ \ }
\newcommand{\ec}{\end{cor}}
\newcommand{\bl}{\begin{lem}\ \ }
\newcommand{\el}{\end{lem}}
\newcommand{\bd}{\begin{definition}}
\newcommand{\ed}{\end{definition}}
\newcommand{\n}{\nabla}
\newcommand{\op}{\oplus}

\newcommand{\be}{\begin{equation}}
\newcommand{\ee}{\end{equation}}
\newcommand\la[1]{\label{#1}}

\def\<#1,#2>{\langle\,#1,\,#2\,\rangle}
\newcommand{\arr}{\begin{array}{rlll}}
\newcommand{\ea}{\end{array}}
\newcommand{\bea}{\begin{eqnarray}}
\newcommand{\eea}{\end{eqnarray}}
\newcommand{\bean}{\begin{eqnarray*}}
\newcommand{\eean}{\end{eqnarray*}}

\newcommand{\ccU}{\CMcal{U}}

\begin{document}
\rightline{J.\ Math.\ Phys.\ {\bf 57}, 102501 (2016)}
\rightline{AEI-2014-040}

\vskip 16truemm

\title[Pseudo-hyperk\"ahler prepotentials]
{On Pseudo-hyperk\"ahler Prepotentials}
\author[C. Devchand]{Chandrashekar Devchand}
\author[A.  Spiro]{Andrea Spiro}


\thanks{This research was partially supported by the 
{\it Ministero dell'Istruzione, Universit\`a e Ricerca} in the framework of the project 
``Real and Complex Manifolds: Geometry, Topology and  Harmonic Analysis''  
and by GNSAGA of INdAM}

\maketitle
\vskip 6truemm
{\small  
\begin{quotation}
\noindent
{\smc Abstract.} 
An explicit surjection from a set of  (locally defined)  unconstrained holomorphic
functions on  a certain submanifold of  
\linebreak
$\Sp_1(\bC) {\times} \bC^{4n}$ onto the 
set HK$_{p,q}$  of local isometry classes of real analytic pseudo-hyperk\"ahler 
metrics of signature $(4p,4q)$ in dimension $4n$ is constructed. The holomorphic 
functions, called prepotentials, are analogues of  K\"ahler potentials for K\"ahler 
metrics and provide a complete parameterisation of HK$_{p,q}$.  In particular, 
there exists  a bijection between HK$_{p,q}$  and the set of  equivalence classes 
of  prepotentials. This affords the explicit construction of  pseudo-hyperk\"ahler 
metrics from specified prepotentials.  The construction generalises one due to 
Galperin, Ivanov, Ogievetsky and Sokatchev. Their work is given a coordinate-free 
formulation and complete, self-contained proofs are provided.  An appendix provides 
a vital tool for this construction:  a reformulation of real analytic  
$G$-structures  in terms of  holomorphic frame  fields on  complex manifolds.
\end{quotation}
}
\bigskip

\section{Introduction}
\noindent 
This  paper is about a parametrisation of local isometry classes of real analytic 
pseudo-hyperk\"ahler metrics on $4n$-dimensional manifolds. This  parametrisation 
is surjective onto the space of local isometry classes and it allows the explicit construction 
of  metrics.  The parameter space consists of  unconstrained holomorphic
functions on  a certain submanifold of   $\Sp_1(\bC) {\times} \bC^{4n}$.

A pseudo-Riemannian manifold  $(M, g)$ is  determined by  the holonomy subbundle 
$P\subset O_g(M)$  of  its orthonormal frame bundle  $\pi\colon O_g(M)\rightarrow M$.  
In turn, $P$ is determined, up to local equivalence,  by its fundamental vector  fields 
$(E_A, e_a)$,  the infinitesimal transformations $E_A$  of its structure group
and the horizontal vector fields $e_a$  given  by the Levi-Civita connection form on $P$.  
Two pseudo-Riemannian manifolds $(M, g)$ and $(M', g')$ are  {\it locally isometric}  
if and only if  their respective  vector fields $(E_A, e_a)$ and $ (E'_A, e'_a)$  are related  by 
a local diffeomorphism.

In the case of a  pseudo-hyperk\"ahler manifold $(M,g)$,  
the associated holonomy bundle  $P\subset O_g(M)$ is locally identifiable  with 
the trivial bundle $\,\pi\colon P|_{\cV} \simeq  \Sp_{p,q} {\times} \cV \rightarrow \cV$, 
for  some open subset $\cV  \subset \bR^{4n}$. 
We shall regard the holonomy bundle as a  subbundle of a larger  
bundle of orthonormal frames  with  structure group  
$\Sp_1 {\cdot} \Sp_{p,q}\,$. This larger bundle has a double covering identifiable with
$\Sp_1 {\times} \Sp_{p,q} {\times} \cV $, a real submanifold of  the complex
Lie group  $\cP =  (\Sp_1(\bC) {\times} \Sp_n(\bC)) \ltimes \bC^{4n}$. 

Using the above local identifications, 
the  vector fields  $(E_A, e_a)$ associated with $(M,g)$ can be identified
with  corresponding vector  fields on the larger space
$\Sp_1 {\times} \Sp_{p,q} {\times} \cV  \subset \cP$.
If $g$  is  real analytic,  these vector fields admit 
holomorphic extension to an open subset  $\cU \subset \cP$.  
Including  the basis vector fields   $(H_0, H_{++}, H_{--})$ of   
$\gsp_1(\bC) \subset \Lie(\cP) = \gsp_1(\bC) + \gsp_n(\bC) +  \bC^{4n}$,  we obtain a set  
$\cA = (H_0, H_{\pm\pm}, E_A, e_a)$ of holomorphic vector fields on  $\cU$, 
which is naturally associated with the  pseudo-hyperk\"ahler metric $g|_{\cV}$. 
This mapping  from  real analytic pseudo-hyperk\"ahler metrics to sets of 
holomorphic vector fields admits an explicit inversion. Introducing the notion  
of  an {\it hk-pair} $(\cA, M)$,  consisting of a set  $\cA
= (H_0, H_{\pm\pm}, E_A, e_a)$ of holomorphic vector fields on an open subset $\cU \subset \cP$, 
satisfying certain Lie bracket relations, and 
a real submanifold $M \subset \cP$, satisfying appropriate transversality conditions  
with respect to  the  fields $H_0, H_{\pm\pm}$ and  $E_A$,
we shall show that  every hk-pair $(\cA, M)$ determines  a  pseudo-hyperk\"ahler 
metric $g$ on the  manifold $M$. Further, the real submanifold $M \cdot \Sp_{p,q} \subset \cU$  is identifiable with 
the (trivial) holonomy bundle $\pi\colon P = M \times \Sp_{p,q} \to M$ of $(M, g)$.

The correspondence between hk-pairs and pseudo-hyperkähler metrics is 
crucial in order to obtain  a complete parametrisation  of the  local isometry 
classes of  real analytic  pseudo-hyperk\"ahler metrics.  We shall prove that: 

\medskip
\begin{itemize} [itemsep=8pt plus 5pt minus 2pt, leftmargin=18pt]
\item[A)]  
There exists a bijection between the  local isometry classes of  
real analytic pseudo-hyperk\"ahler metrics and local equivalence classes of hk-pairs.

\item[B)]  Every local equivalence class of hk-pairs contains a distinguished 
subclass of {\it  canonical hk-pairs}, each  completely determined by a  
single unconstrained  holomorphic function $\cL$,  the {\it prepotential},
defined   on a  complex submanifold   
of $\cP =  (\Sp_1(\bC) {\times} \Sp_n(\bC)) \ltimes  \bC^{4n}$. 
\end{itemize}
 
\newpage
\noindent
Combining A and B:  

\noindent
{\it For every real analytic pseudo-hyperk\"ahler manifold $(M, g)$,   
the restriction  of the metric $g|_\cV$  to a sufficiently small open subset 
$\cV \subset M$ can be associated with  an unconstrained  prepotential $\cL$. 
Conversely, an arbitrary  holomorphic  function $\cL$ on a certain 
complex submanifold of   $\cP$ 
 determines  a  pseudo-hyperk\"ahler metric $g$, unique up to  isometry. 
We call $\cL$ a  prepotential of the  pseudo-hyperk\"ahler metric $g$.}

One of the striking developments in theoretical physics, which animated 
much  mathematical interest in supersymmetry, was the appearance
of special geometries in various surprising contexts. For instance, requiring 
the harmonic map equations on a four-dimensional Lorentzian manifold 
to be supersymmetric  automatically provided the target manifold with a
K\"ahler structure \cite{z}. Soon, it was found \cite{agf} that extended 
versions of supersymmetry yielded hyperk\"ahler targets. The search
for a supersymmetric action functional gave rise to the harmonic space method
\cite{gikos}, which yielded a construction of a supersymmetric
Lagrangian $\cL$ in an extended space  called {\it harmonic superspace}.
The construction established a correspondence between the functions $\cL$ and  
hyperk\"ahler metrics. In an interesting collateral development \cite{gios1, gios2,gios_book}, 
these authors extracted the latter correspondence from the original context of
supersymmetric field theories, presenting a construction of   hyperk\"ahler
metrics, parametrised by a prepotential  $\cL$, much as the K\"ahler potential
parametrises K\"ahler metrics.  
A streamlined presentation of the correspondence
was given in \cite{do1}. This was amenable to a generalisation to supersymmetric 
hyperkähler spaces \cite{do2,do3}.
Further,  a discussion of the prepotential in the framework of
quaternionic Kähler metrics has been given in \cite{gio}.

The correspondence  between (pseudo-)hyperk\"ahler metrics and free prepotentials  
provides an efficient  parameterisation of  {\it all} local isometry classes of 
real analytic pseudo-hyperk\"ahler metrics and is important from both 
field theoretical  and differential geometric  points of view. 
The  purpose of  this paper is to give an appropriate mathematical presentation,  
in coordinate-free language, with complete and  self-contained proofs;
thus opening the way to further developments and applications. 
An  analogous correspondence between Yang-Mills connections on (generalised) 
hyperk\"ahler manifolds and free prepotentials has been discussed in \cite{acd}.  

 \vskip 2pt plus 10pt minus 6pt
{\it Structure of the paper.} 
Section \ref{basic} contains  our notation  and certain basic facts.  
In Sect.\  \ref{hk_frames} and \ref{classify} we discuss hk-pairs, show how 
the associated  pseudo-hyperk\"ahler metrics may be determined and
we state the two  main  theorems, which establish the  surjective 
correspondence between prepotentials  and   equivalence classes of hk-pairs. 
In Sect.\ \ref{harmonics} we obtain some technical results on differential 
equations  on harmonic spaces required  for the proofs 
of  our main theorems   in  Sect.\  \ref{pf_vpot_thm} and \ref{pf_prepot_thm}.
A summary of our construction, a five-step recipe to obtain  a pseudo-hyperk\"ahler 
metric from a  prepotential, is given in  Sect.\  \ref{reconstruct}. 
The appendix discusses  real and complex $G$-structures, reformulating them
in terms of  holomorphic frame  fields on  complex manifolds.
The latter provide a useful tool for the investigation of local properties of 
manifolds with  real analytic $G$-structures.
In  Sect.\ \ref{appendixHK} we discuss the particular case of   $G$-structures 
corresponding to real analytic pseudo-hyperk\"ahler manifolds.  
Finally, in Sect.\ \ref{bijection}, we prove the bijection  between 
local isometry classes  of real analytic pseudo-hyperk\"ahler metrics and 
local equivalence classes of  hk-pairs.

The discussion in the appendix has been kept general enough with a view to
being directly applicable to other geometries. In particular,  we intend to use  
it  to obtain a new  parametrisation of the local isometry classes of 
quaternionic Kähler metrics.

\noindent
{\bf Acknowledgments.} 
We are grateful to Dmitri Alekseevsky for  very useful discussions on many aspects 
of this paper.  One of us (CD) thanks Hermann Nicolai and the Albert Einstein 
Institute for providing an excellent research environment. 
We thank  the Max-Planck-Instit\"ut f\"ur Mathematik,  Universit\`a degli Studi di Firenze, 
Universit\"at Potsdam and  Universit\`a degli Studi di Camerino 
for supporting visits for the purposes of this collaboration. 
We are happy to thank the anonymous referee for a careful reading of the manuscript.

\section{Basic notions}
\setcounter{equation}{0}
\label{basic}

\subsection{A basis for  $\  \gp =(\gsp_1(\bC) + \gsp_n(\bC)) + \bC^{4n}$}
\noindent 
Consider the Lie algebra $\ \gg= \gsp_1(\bC) {+} \gsp_n(\bC)$.  
Since the  vector space  $V = \bC^2 \otimes \bC^{2n} \simeq \bC^{4n}$
is a  $\gg$-module, we may extend  $\gg$ to the Lie algebra
$\gp = \gg + V$  with  additional  brackets, 
$$
[v, v'] = 0\,,\quad [A, v] = A \cdot v\,, \qquad v, v' \in V,\quad A \in \gg\,,
$$
where $A\cdot \colon V \rightarrow V ,\  v\mapsto A {\cdot} v$,  denotes  
the standard action  of $A$ on  $V$. 
As our standard  basis  for $ \gsp_1(\bC)$, we use the triple 
$$
H^o_0  = \begin{pmatrix} 1 & 0 \\ 0 & -1 \end{pmatrix} ,\qquad 
H^o_{++}  = \begin{pmatrix} 0 & 1 \\ 0 & 0 \end{pmatrix} ,\qquad   
H^o_{--} = \begin{pmatrix} 0 & 0 \\ 1 & 0 \end{pmatrix}\ .
$$
 The  standard bases of  $\bC^2$ and $\bC^{2n}$ are  denoted respectively by 
 $(h^o_+, h^o_-)$ and $(e^o_a), a=1,\dots,2n$. 
In terms of these, the basis elements of $V = \bC^{4n}$ are given by  
$e^o_{\pm a} := h^o_{\pm} \otimes e^o_a$.  In this basis 
$V$ manifestly decomposes as $V = V_{+} + V_{-}\,$, where
$V_{\pm} :=   \Span_\bC \{e^o_{\pm 1}, \dots, e^o_{\pm 2n}\}$ are eigenspaces 
of $H^o_0$, since  $h^o_\pm$ are its eigenvectors with eigenvalues  $\pm 1$.
As a basis of  $ \gsp_n(\bC)$, we consider  the $N$-tuple  of  $2n{\times} 2n$ matrices,
$(E^o_1, \ldots, E^o_{N})$, $N = (2 n+1) n$,  corresponding, by  the classical 
identification  $\gsp_n(\bC)\simeq \vee^2 \bC^{2n}$,   
to   the    tensors  $E^o_{ab} := e^o_a \vee e^o_b$. 
The nonzero Lie brackets of the basis elements of $\gp$ are given by
\begin{align} 
  [H^o_{++}, H^o_{--} ] &= H^o_0\,, 
&  [H^o_0, H^o_{\pm\pm} ] &= \pm 2 H^o_{\pm \pm}\,, 
 & [E^o_A, E^o_B] &=  c_{AB}^C E^o_C\,, 
\la{flat_g}
\\
 [H^o_{\pm\pm}, e^o_{\mp a}]  &= e^o_{\pm a}\,, 
 &  [E^o_A, e^o_{\pm a}] &= (E^o_A)_a^b e^o_{\pm b}\,,   
 &  [H^o_0, e^o_{\pm a}] &= \pm e^o_{\pm a}\,, 
\nonumber
\end{align}
where  $(E^o_A)_a^b$ denote  the entries of the  matrix $E^o_A$ and  $c^C_{AB}$ 
the structure constants of $\gsp_n(\bC)$ with respect to the  basis $(E^o_A)$.
  
The above  basis manifestly displays $\gp=\gg + V$ as a  Lie algebra with a five-fold
gradation, 
$$
\gp = \gp_{-2} \op \gp_{-1} \op \gp_{0} \op  \gp_{+1} \op  \gp_{+2}\ ,
$$
with $[ \gp_i , \gp_j ] \subset \gp_{i+j}$ and $\gp_{i+j} = 0$ if $ |i+j| > 2$.
Here the one-dimensional submodules $\gp_{\pm2}$ are generated by $H^o_{\pm\pm}$, 
and $ \gp_{\pm 1} = V_{\pm}$.  The element $H^o_0\in \gp_0$ is the grading element, 
which defines the gradation of $\gp$ by virtue of the space $\gp_j$ being the eigenspace
of $\ad_{H^o_0}$  with eigenvalue $j$. We say that an element  $x\in \gp_j$ has
{\it charge} $j$ and we write $x$, as in the basis above, with an appropriate
number of $+$ or $-$ signs in the subscript.

\subsection{Coordinate systems and left-invariant vector fields
}
\label{coords}
\noindent 
The  connected subgroups of  $\GL_{4n}(\bC) \ltimes \bC^{4n}$  with Lie algebras  
$\gg$ and $\gp$,   respectively, are  denoted by 
$$G = \Sp_1(\bC) \times \Sp_n(\bC)\quad\mbox{and}\quad 
\cP =  G \ltimes \bC^{4n}.
$$
We parametrise $\GL_{2n}(\bC)$ and $ \GL_2(\bC)$  using the entries 
of their respective elements, $B = (B^a_b) \in \GL_{2n}(\bC)$  and
$$ 
U  = \begin{pmatrix} u^1_+ & u^1_-\\ u^2_+ & u^2_- \end{pmatrix} \in \GL_2(\bC)\,.
$$ 
The elements of the subgroup $\Sp_n(\bC) \subset \GL_{2n}(\bC)$ 
are characterised  by the constraints  $\,B^c_a \o_{cd} B^d_b =\o_{ab}\,$,
where $\o$ is the $2n{\times} 2n$ matrix of the symplectic form.
Similarly, the elements of  $\Sp_1(\bC) =\SL_2(\bC) \subset \GL_2(\bC)$   satisfy
\begin{equation}
 \det(u^i_\pm) = u_+^1 u_- ^2 - u_+^2 u_- ^1 =  1\,,
\end{equation}
which  is tantamount to  
$\varepsilon^{AB} \varepsilon_{ij} u^i_B u^j_C = \d^A_C\,$, where
the $2 {\times} 2$  skewsymmetric matrices 
$(\varepsilon_{ij})_{i,j \in \{1,2 \} }\,,\, (\varepsilon_{AB})_{A,B \in \{+,- \} }$
and their respective inverses 
$(\varepsilon^{ij}) \,{=}\, (\varepsilon_{\ell m})^{-1}$, 
$(\varepsilon^{AB}) \,{=}\, (\varepsilon_{CD})^{-1} $,
have nonzero elements 
$\varepsilon_{12} \,{=}\, - \varepsilon^{12} \,{=}\, 1$ and 
$\varepsilon_{+-} \,{=}\, - \varepsilon^{+-}  \,{=}\,  1 $. 
Then,   the inverse matrix $U^{-1} \in \Sp_1(\bC) $ takes the form
\begin{equation}
\label{inverseU}
U^{-1}  =  \left(- u^\pm_i  \right) = -( \varepsilon^{AB} \varepsilon_{j\ell}u^\ell_B  )
= \begin{pmatrix} u^2_- & -u^1_-\\ -u^2_+ & u^1_+ \end{pmatrix}.
\end{equation}
Here, we adopt the convention
$\  \a_{i}  := \varepsilon_{i \ell} \a^{\ell}\,$ and 
$\,\b^{j} := \varepsilon^{j\ell} \b_{\ell}\,$
for  raising and lowering $\Sp_1(\bC)$-indices.

We  denote by $(z^{1 a}, z^{2b})$,  $a,b =1,\dots 2n$,  the elements of 
$V = \bC^{4n} = \bC^2 \otimes \bC^{2n}$ and 
call  the standard system of coordinates on 
$ \left( \GL_2(\bC) {\times} \GL_{2n}(\bC)\right)\ltimes \bC^{4n}$
the {\it central coordinate system},
$$
((u^i_{\pm}), (B^a_b),  (z^{i a}))\colon 
\left(\GL_2(\bC) {\times} \GL_{2n}(\bC)\right)\ltimes \bC^{4n} 
\longrightarrow \bC^4 {\times} \bC^{4n^2} {\times} \bC^{4n}.
$$
The basis elements of $\gp$ in \eqref{flat_g} 
are  restrictions  to $\cP$  of left-invariant  vector fields of  the Lie group
$ \left( \GL_2(\bC) {\times} \GL_{2n}(\bC)\right)\ltimes \bC^{4n}$, 
which contains $\cP$  as a proper Lie subgroup.  
 In  central coordinates we have
\begin{align}  
H_{0}^o &= u^i_+ \frac{\partial}{\partial u^i_+}  
                                   - u^i_- \frac{\partial}{\partial u^i_-} \ , &
H_{++}^o &=  u^i_+ \frac{\partial}{\partial u^i_-} \ , &
H_{--}^o  =   u^i_- \frac{\partial}{\partial u^i_+}\ ,
\nonumber \\ \label{coord_vfs} 
 E^o_A &=   B^a_c(E^o_A)^c_b \frac{\partial}{\partial B^a_b}\ , &
e^o_{\pm a}  
 &= B^b_a u^j_{\pm} \frac{\partial}{\partial z^{j b}}\ . &
\end{align}
A  useful alternative coordinate system is the  {\it analytic coordinate system}, 
$$
((u^i_{\pm}), (B^a_b), (z^{\pm a}))\colon 
 \left(\GL_2(\bC) \times \GL_{2n}(\bC)\right)\ltimes \bC^{4n} 
 \longrightarrow \bC^4 \times \bC^{4n^2} \times \bC^{4n}\, ,
$$
determined by the coordinate transformation
$$
(u^i_{\pm}) \mapsto (u^i_{\pm}) \,,\quad 
(B^a_b)  \mapsto  (B^a_b) \,,\quad 
(z^{ia}) \mapsto  \left.(z^{\pm a})\right|_{(U , B, z)} 
:= U^{-1} \cdot \left.(z^{ia})\right|_{(U , B, z)}\,,
$$
where the last  mapping is equivalent to
$$
 \left.  \begin{pmatrix} z^{1a} \\ z^{2a}
                  \end{pmatrix}\right|_{(U = (u^i_{\pm}), B, z = z^{ia})} 
 = \begin{pmatrix} u^1_+ & u^1_-\\ u^2_+ & u^2_{-}\end{pmatrix}
\left. \begin{pmatrix} z^{+a} \\ z^{-a}
                  \end{pmatrix}\right|_{(U = (u^i_{\pm}), B, z = z^{ia})}.
$$
In these coordinates, the vector fields  $e^o_{\pm a}$ have the simple expression 
$e^o_{\pm a} = B_a^b \frac{\partial}{\partial z^{\pm b}}$. 
In terms of the elements of $U^{-1}$ given in \eqref{inverseU}, we may also write
$z^{\pm a} = - u^\pm_i z^{ia}$.

\subsection{U(1)-charge}
\label{charge}
\noindent 
Let  $F$  be a   $\bC^N$-valued    holomorphic function, defined  
on (an open subset of) $\Sp_1(\bC)$ or    $\Sp_1(\bC) \ltimes \bC^{4n}$. 
Our main results depend crucially on  certain properties of such functions and 
we discuss these in Section \ref{harmonics}. 
In analogy with the terminology for the elements of $\gp$, we  say that   
 $F$ has  {\bss charge}   $k$  
if it  is a solution of    the differential equation\footnote{We use the notation  
$X {\cdot} f := X(f)$ to denote the directional  derivative of a function $f$ 
on a manifold $M$ along a  vector field $X$.}
\beq 
H^o_0 {\cdot}F = k F\,, \quad  k \in \bZ\, .
\eeq
We shall write such functions with $|k|$ plus or minus signs in the subscript 
or,  when it is less cumbersome, as $F_{(\pm |k|)}$.  We also adopt the sign 
convention that a plus or minus sign in the {\it superscript}  denotes, 
respectively, a negative or positive  charge,  i.e. the opposite charge to that 
denoted by the same sign in the subscript. So,  for instance, 
the coordinates $z^{\pm a}$ defined above satisfy the condition
 $\ H^o_0 {\cdot}z^{\pm a} = \mp z^{\pm a}$.

\subsection{Real structures on $\cP=(\Sp_1(\bC)\times\Sp_n(\bC))\ltimes\bC^{4n} $}
\label{real_structures}
\noindent 
Consider  non-negative integers $p,q$   with  $p + q = n$, and organise
the matrix $I_{2p, 2q}$ of the flat metric of signature $(2p,2q)$ and the complex structures $\bJ$ and $J$
of $\bC^{2n}$ and $\bC^2$, respectively,   as follows: 
\begin{equation*}
I_{2p,2q} =  \begin{pmatrix}  \eta & 0  \\    0 & \eta         \end{pmatrix}  ,\quad 
(\bJ_a^b) = \begin{pmatrix}  0&- \eta  \\   \eta &0     \end{pmatrix},\quad 
(J_i^j)   =           \begin{pmatrix}  0& -1  \\    1 & 0   \end{pmatrix},
\end{equation*}
where $\eta :=     \begin{pmatrix}  I_p & 0\\ 0 & - I_q  \end{pmatrix}$.
Using these matrices we may define a   holomorphic map
\begin{equation*} 
\begin{split}
\psi \colon \left(\GL_2(\bC) {\times} \GL_{2n}(\bC)\right)\ltimes \bC^{4n}   
 &\rightarrow 
     \left( \GL_2(\bC) {\times} \GL_{2n}(\bC)\right)\ltimes \bC^{4n}\\
(U,  B,  (z^{ia}) ) &\mapsto (\psi(U), \psi(B) , (\j^{jb}(z)) )\,,
\end{split}
\end{equation*}
where 
\begin{equation*}
\psi(U) = (U^T)^{-1},\quad  \psi(B) = (I_{2p,2q} {B^T} I_{2p,2q})^{-1},\quad
\psi^{jb}(z) =  - J^{j}_i\  \bJ^{b}_a\ z^{ia}. 
\end{equation*}
 Under  $\psi$   the left-invariant vector fields of $\cP$ transform as:
 \beq \label{bistecca} 
 {\j}_*(H^o_0) = - H^o_0\ ,\quad   
 {\j}_*(H^o_{\pm\pm}) = - H^o_{\mp\mp}\ , \quad  
 {\j}_*(e^o_{\pm a}) := \pm  \wh \bJ_a^b  e^o_{\mp b}
 \ ,\eeq
 where   
 the $\GL_{2n}(\bC)$-valued function 
$\wh \bJ\colon \cP \rightarrow \GL_{2n}(\bC)$  is defined by 
\beq 
\wh  \bJ|_{(U, B, z)} :=   
- I_{2p, 2q} {\cdot} \left(B\, {\cdot}
                 \begin{pmatrix}  0&- I_n \\ I_n &0 \end{pmatrix} 
             {\cdot}\, B^T \right)^{-1} .
\eeq
The map $\psi$ determines, by conjugation, an anti-holomorphic map,
\beq \label{re_str_1}
\tau(U,  B,  z)  := \overline{\psi(U,  B,  z)}
= \left((\overline{U^T})^{-1}\,,\, (I_{2p,2q} \overline{B^T} I_{2p,2q})^{-1} \,,\, \overline{\j(z)}\right).
\eeq
The map $\psi$ and complex conjugation clearly commute.
The  push-forwards of the complex vector fields  $H^o_0$, $H^o_{\pm\pm}$,  
$e^o_{\pm a}$ under the anti-involution  $\t$ are  
\beq\label{pushforwards}  
\t_*(H_0^o) =   - H^o_{0}\ ,\quad 
\t_*(H_{\pm\pm} ^o) =  - H^o_{\mp \mp}\ ,\quad 
\t_*(e^o_{\pm a}) =  \pm  \wh \bJ_a^b e^o_{\mp b}\  .
\eeq
Now,  $\t(\cP) \subset \cP$ and the  $\t$-fixed point set is
$
\cP^\t = (\Sp_1 \times \Sp_{p,q})\ltimes (\bC^{4n})^\t \,, 
$
where
$$
(\bC^{4n})^\t  := \{\, (z^{1a}, \overline{z^{2a}},  \h^a_b z^{2b},  
                                      - \h^a_b \overline{z^{1 b}}),  \ a = 1,\ldots, n\,\}.
$$
We call   $\t|_{\cP}$ the  {\bss real structure  of  signature $\boldsymbol{(4p, 4q)}$ on $\boldsymbol{\cP}$ }. 
For simplicity,   we shall use $\t$ instead of  $\t|_\cP$ and we similarly 
denote  each  of the three  component parts, the anti-involutions  
on $\Sp_1(\bC)$, $\Sp_n(\bC)$ and $\bC^{4n}$  given by \eqref{re_str_1}. 
Which anti-involution is meant will  always  be clear from the context.

The space  $(\bC^{4n})^\t$ is   endowed with an  $(\Sp_1 {\times} \Sp_{p,q})$-invariant 
quaternionic structure  $\gJ\colon\bH^n \to \bH^n$, $\gJ^2 = -1$, 
$$
\gJ(z^{1a}, \overline{z^{2b}},  \h^c_a z^{2a}, - \h^d_a \overline{z^{1 a}})  
           := (  z^{2b}, - \overline{z^{1a}}, \h^c_a  z^{1a},  \h^d_a \overline{z^{2 a}})\ ,
$$
and is naturally identifiable with   
$\bH^n =   \{(z^{1a}, \overline{z^{2b}}) , \ z^{ja} \in \bC\} $, 
$n$-dimensional  quaternion  space.

\subsection{$\gsp_n(\bC)$-equivariance}
\noindent 
Let  $\rho\colon \gsp_n(\bC) \rightarrow \ggl(S)$ be  a linear representation 
of $\gsp_n(\bC)$ on  a complex vector space  $S$ and 
$f\colon \cU \subset \cP \rightarrow S$ a holomorphic map.
We   say that    $f$ is {\bss $\gsp_n(\bC)$-equivariant} if it 
satisfies  the differential equation
\beq 
E_A^o {\cdot}f = \rho(E_A^o)(f)\ .
\eeq   
For instance,  functions $(f^a)\,,\, (h_b)\,,\, (g^a_b)$ and $(\ell^a_{bc})$, 
taking values in  the spaces 
$\wt V:=\bC^{2n}\,,\, \wt V^*\,,\, \wt V \otimes \wt V^*$  and 
$\wt V \otimes \wt V^* \otimes \wt V^*$,  respectively, are   
$\gsp_n(\bC)$-equivariant  if they satisfy the differential equations
\begin{align*}
E_A^o {\cdot} f^a &=  - (E^o_A)^a_b f^b\ ,\quad 
E_A^o {\cdot} h_a  =   (E^o_A)^b_a h_b\ ,  \\
E_A^o {\cdot} g^a_b &=  - (E^o_A)^a_c g^c_b + (E^o_A)^c_b g^a_c\  ,\\
E_A^o {\cdot} \ell^a_{bc} 
  &=  - (E^o_A)^a_d \ell^d_{bc} + (E^o_A)^d_b \ell^a_{dc} + (E^o_A)^d_c \ell^a_{bd}\ .  
\end{align*}

\section{hk-frames}
\label{hk_frames}
\noindent 
Let $p, q$ be   non-negative integers   with  $p {+} q = n$  and $\t$  the  real 
structure \eqref{re_str_1}  of signature $(4p,4q)$ on  
$\cP =  G {\ltimes} V ,\,  V = \bC^{4n},\,  G = \Sp_1(\bC) {\times} \Sp_n(\bC)$.
Following \cite{gios2,gios_book} we introduce:

\begin{definition}  \label{appropriate} 
Let $\cV \subset V =  \bC^{4n}$ be a connected, simply connected neighbourhood of $0$, 
invariant under  the involution $\t$.  The {\bss harmonic space} of  $\cV$ is
the   set  $\cH|_{\cV} :=  \Sp_1(\bC) {\times}\{I_{2n}\} {\times} \cV$.  
When $\cV = \bC^{4n}$,  we  write simply $\cH$.
Further, an open subset $\cU \subset \cP$  is called {\bss appropriate} if it is
a $\t$-invariant simply connected neighbourhood of $e = (I_2, I_{2n}, 0)$,  
such that     $\cU \cap \cH= \cH|_\cV $ for some open subset  $\cV \subset  \bC^{4n}$.
\end{definition}

\noindent
Let  $\gX^{\text{hol}}(\cU)$ be the space of holomorphic vector fields on  
 an appropriate open subset $\cU \subset \cP$. 

\begin{definition}  \label{hkframe} 
A   collection $\cA = (H_0, H_{\pm\pm}, E_A, e_{\pm a})\ $ 
of  holomorphic vector fields  in $ \gX^{\text{hol}}(\cU)$,
$\bC$-linearly independent at all points, 
is an {\bss hk-frame\/} if:  

\begin{itemize}[itemsep=4pt, leftmargin=18pt]
\item[a)] 
$\cU$  carries a holomorphic right action 
$ \rho\colon \cU {\times} G{\ \rightarrow\ } \cU$  
of $G = \Sp_1(\bC){\times}\Sp_n(\bC)$ such that the associated  
group homomorphism $\wh \rho\colon G \rightarrow \Diff(\cU)$  satisfies
$$
\wh \rho_*(H^o_0) = H_0\,,\quad 
\wh \rho_*(H^o_{\pm\pm}) = H_{\pm\pm}\,, \quad 
\wh \rho_*(E^o_A) = E_A\,.
$$ 
Since the map $\ \wh \rho_*\colon \gg \simeq T_e G \rightarrow \gX^{\text{hol}}(\cU)$
is an injective Lie algebra homomorphism, the vector fields 
$(H_0, H_{\pm\pm}, E_A)$ satisfy the Lie bracket relations  (cf.\ \eqref{flat_g})
\beq
[H_0, H_{\pm\pm} ] \ =\  \pm 2 H_{\pm \pm}\ , \quad
[H_{++}, H_{--} ] =\,  H_0\ ,\quad 
[E_A, E_B] =  c_{AB}^C E_C\ .
\label{algebra}	
\eeq

\item[b)] 
The  Lie  brackets of the other fields of $\cA$ are given by
\begin{align} 
[H_0, e_{\pm a}]   &=    \pm e_{\pm a}\ ,& 
[H_{\pm\pm}, e_{\pm a}] &=   0\ ,&
[H_{\pm\pm}, e_{\mp a}]  &= e_{\pm a}\ , &
\nonumber  \\
[E_A, e_{\pm a}]  &=  (E^o_A)_a^b e_{\pm b} \ , &
[e_{\pm a} , e_{\pm b}]  &= 0\ ,&
  [e_{+a}, e_{-b}] &= R_{ab}^A E_A\ ,  
\label{curv_constr} 
\end{align}
where  $\ R_{ab}^A\colon \cU \rightarrow \bC$ are   holomorphic functions.

\item[c)]  
The orbit space $M = \cU/G$ is a  manifold and   
$\pi \colon \cU \rightarrow M = \cU/G$   is a principal $G$-bundle  over $M$.
\end{itemize}
\noindent
\noindent
A pair of hk-frames   $\cA\,, \cA'$ defined on appropriate open subsets 
$\cU\,, \cU' \subset \cP$,   respectively, are  
 {\bss locally equivalent\/} if there exists a  $G$-equivariant 
biholomorphism  $\varphi\colon \cU \rightarrow \cU'$ which maps  the fields of $\cA$ 
into the corresponding fields of  $\cA'$. We write:   $\cA' = \varphi_*(\cA)$.
\end{definition}

\noindent
A particularly important class of hk-frames is given by:

\begin{definition}  \label{canon_frame}
A {\bss canonical hk-frame} is an hk-frame  
$\cA \,{=}\, (H_0, H_{\pm\pm}, E_A, e_{\pm a})$  on an appropriate open subset   
$\cU \subset \cP$,  in which the vector fields take the form\footnote{
Here, the components $v^{\pm b}_{\pm \pm}$ and  $v_{-a}^{+b}$ 
 are complex functions on $\cU$ and
have charges in accordance 
with  the notation of Sect.\ \ref{charge}:  a plus (minus) sign in the superscript  denotes 
a negative (positive) charge and vice versa for subscripts. So, for instance, the components $v^{- b}_{++}$, 
which satisfy  $\ H^o_0 {\cdot} v^{- b}_{++} = 3 v^{- b}_{++}\,$ in virtue of 
eq.\ \eqref{algebra}, have  charge $+3$.} 
\begin{displaymath}
 \begin{split}
H_0 &=  H_0^o \\
H_{\pm\pm} &=  (H_{\pm\pm}^o + \d H_{\pm\pm}) \ , \quad
\d H_{\pm \pm} 
= v^{-b}_{\pm \pm} e^o_{-b} +  v^{+b}_{\pm \pm} e^o_{+b} + A_{++}^B E^o_B\\
E_A   &=  E^o_A \\ 
e_{+a} &= e^o_{+a} \\
e_{-a} &= (e^o_{-a} + \d e_{-a})\ ,\quad
\d e_{-a} = v_{-a}^{+ b} e_{+ b}^o + A^B_{-a} E^o_B\ ,
\end{split} 
\end{displaymath}
 with components $v_{++}^{+ b}$ identically vanishing  
 on the submanifold $\{z^{+a} = 0\} \subset \cU$. 
\end{definition}

\noindent
The $\bC^{2n}$-valued function
$\ 
(v^{-b}_{++})|_{\cH\cap \cU}\colon \cH \cap \cU\rightarrow \bC^{2n}$, 
appearing as the coefficient of $e^o_{-b}$ in $H_{++}|_{\cH \cap \cU}$, 
is called  the {\bss v-potential} of $\cA$.
We shall see that it effectively parametrises the equivalence classes of  hk-frames.

\noindent
{\bss Remark.} Canonical  hk-frames are called  {\it analytic frames} in the 
harmonic space literature \cite{gios_book}. 

\noindent
There is another important class of hk-frames:

\begin{definition}  \label{centralframe} 
A {\bss central hk-frame} is an  hk-frame  $\cA$   
on an appropriate open subset   
$\cU \subset \cP$, in which  the  vector fields take the form
\begin{align*}
 H_0 &= H_0^o\ ,\quad 
 H_{\pm\pm} = H^o_{\pm\pm}\ ,\quad 
 E_A  = E^o_A\ ,\\
 e_{\pm a} &= e^o_{\pm a} +  v^{+b}_{\pm a} e_{+b}^o + v^{-b}_{\pm a} e_{-b}^o 
                         + A_{\pm a}^B E_B^o\,,
\end{align*}
where the components $v^{+b}_{\pm a} $, $v^{-b}_{\pm a} $, $A_{\pm a}^B$ are  
holomorphic functions.
The  collection  $ \cA^o = (H^o_0, H^o_{\pm\pm}, E^o_A, e^o_{\pm a})$ 
of left-invariant vector fields  on $\cP$, forming  the standard basis of $\gp$, 
is called the {\bss flat hk-frame}.
\end{definition}

We shall see  that  if appropriate reality conditions are satisfied,  
every  hk-frame $\cA$,  defined on an appropriate open set $\cU \subset \cP$, determines 
a real analytic pseudo-Riemannian metric $g$ on $M = \cU/G$.  
In this case,   the functions  $R_{ab}^A$, appearing in \eqref{curv_constr},  
are components of the  curvature tensor of $(M, g)$. 
In particular, the flat hk-frame corresponds to a flat pseudo-Riemannian metric.

A   local biholomorphism 
$\varphi \colon \cU   \rightarrow \cU' $ between two appropriate open subsets of $\cP$, 
with  components in central coordinates $\f = (\varphi^i_{\pm}, \varphi^a_b, \varphi^{ia})$ such that 
$\varphi^i_{\pm}(U, B, z) = u^i_{\pm}$ and  mapping  a central hk-frame  $\cA$ into  a canonical  hk-frame $\cA' = \f_*(\cA)$, 
  is called  a {\bss bridge} between $\cA$ and $\cA'$.   From the definitions of central and canonical hk-frames, this means that 
$\f$ is a biholomorphism  satisfying
 \beq \label{bridge}
\begin{split} 
& \varphi^i_{\pm}(U, B, z) = u^i_{\pm}\, \\
 & \varphi_*( E_A^o) =  E^o_A\,,\quad  
\varphi_*( H^o_0) = H^o_0\,,\quad   
\varphi_*(e_{+a}) =e^o_{+ a}\,, \\
\end{split}
\eeq
with the property that  the vector fields $H'_{\pm\pm}:= \varphi_*(H^o_{\pm\pm})$, $e'_{-a}:= \varphi_*(e_{-a})$
have the form prescribed in Def. \ref{canon_frame}.

Canonically associated with an hk-frame $\cA= (H_0, H_{\pm\pm}, E_A, e_{\pm a})$, 
there exists an  {\bss absolute hk-parallelism},  a $\bC$-linear map  
$\ \a^\cA\colon \gp \rightarrow \gX^{\text{hol}}(\cU)\, $  from the (abstract) Lie 
algebra $\gp$ to the holomorphic vector fields on  $\cU\subset \cP$, defined by
\beq \label{corr1} 
\a^\cA(H_0^o) = H_0\,,\ \, \a^\cA(H^o_{\pm\pm}) = H_{\pm\pm}\,,\ \, 
\a^\cA(E^o_A) = E_A\,,\ \, \a^\cA(e^o_{\pm a}) = e_{\pm a}\,.
\eeq
This map satisfies the following conditions:
\begin{itemize} [itemsep=4pt, leftmargin=20pt]
\item[a)] 
 $\a^\cA$ is a holomorphic absolute parallelism,  i.e.\  it gives a linear 
 isomorphism between $\gp$ and $T^{10}_w \cU$   for every  $w \in \cU$. 

\item[b)]  The restriction 
$\a^\cA|_{\gg}\colon \gg \rightarrow \Span_\bC\{H_0, H_{\pm\pm}, E_A\}$ 
coincides with the  map
$\wh \rho_*\colon \gg \simeq T_e G \rightarrow \gX^{\text{hol}}(\cU)$ 
in Def.\,\ref{hkframe} 
corresponding to the right action  $\rho\colon \cU {\times} G \rightarrow \cU$.

\item[c)] $\a^\cA([X, v]) = [\a^\cA(X), \a^\cA(v)]\ $ for all  $\,X \in \gg\,$ and $\,v \in V$.

\item[d)] 
$[\a^\cA(v), \a^\cA(v')]_w \in \a^\cA(\gsp_n(\bC))\ $ for all   $\,v, v' \in V\,,\, w \in \cU\,$.
\end{itemize}
Conversely,  for a given right action $\rho$ of $G$ on an 
appropriate open subset $\cU \subset \cP$ and a $\bC$-linear map $\a^\cA$ 
satisfying a) -  d),  the  vector fields defined by \eqref{corr1} constitute  an hk-frame.

\section{The main theorems}
\label{classify}
\subsection{Canonical hk-pairs}
  
 \noindent 
For a  real  $\cC^\infty$-manifold $M$, we denote by $\gX(M)$  the space  of smooth vector fields 
on $M$. Given an hk-frame $\cA$ on an appropriate open subset $\cU \subset \cP$, we call the $\bR$-linear map
\begin{equation}
\a^\cA_{(\bR)}\colon \gp \rightarrow \gX(\cU)\ ,\qquad 
X \mapsto \a^\cA_{(\bR)} (X) := 2 \Re(\a^\cA(X))\ ,
\label{reality}
\end{equation}
the  {\bss real absolute hk-parallelism} associated with $\cA$. 
Notice that
$$
\a^\cA_{(\bR)}(i X) = J_o \a^\cA_{(\bR)}(X)
$$
 for all $X \in \gg$, 
where  $J_o$  is the {\it real} $(1,1)$-tensor field corresponding to  the standard  
complex structure of  $\cP$,   and  that $ \a^\cA(X) =(\a^\cA_{(\bR)}(X))^{10}$ .

Our  classification of    (local) isometry classes of pseudo-hyperk\"ahler metrics is 
based on the following:

\begin{definition}\label{hk-pair}
An {\bss hk-pair} of signature $(4p,4q)$ is a pair $(\cA, M)$, consisting of  
an hk-frame $\cA$ on  an appropriate open subset $\cU \subset \cP$  and  a totally 
real $4n$-dimensional  submanifold $M \subset \cU$  passing through $e = (I_2, I_{2n}, 0)$ and satisfying the 
following conditions:

\begin{itemize} [itemsep=4pt, leftmargin=20pt]
\item[i)] 
$M$ is transversal to the  $\gsp_{p,q}$-orbits, i.e. 
$T_x M \cap \a^\cA_{(\bR)}(\gsp_{p,q})\vert_x= \{0\} \ {\rm for\ all}\  x $

\item[ii)]  $T_x M \subset \a^\cA_{(\bR)}(V^\t + \gsp_{p,q})\vert_x$
%
%
 for all $x \in M$. 
\end{itemize}

\noindent
Two  hk-pairs   $(\cA, M), (\cA', M')$  are called  {\bss locally equivalent\/} 
if there exists a $G$-equivariant  biholomorphism  $\varphi\colon \cU \rightarrow \cU'$ 
mapping $\cA$ into  $\cA' $   and   $M$ into  $M'$.  

\noindent
The  {\bss flat hk-pair}  $(\cA^o, M^{o})$ of signature $(4p,4q)$ consists of
the flat hk-frame $\cA^o$ on $\cP$ together with the  real submanifold 
\beq \label{flat manifold} 
M^{o} =   \{I_2\} {\times} \{I_{2n} \} {\times} (\bC^{4n})^\t \subset \cP \,.
\eeq
\end{definition}

\begin{definition}\label{canon_hkpair}  
An  hk-pair $(\cA, M)$  on   $\cU \subset \cP$ is called {\bss  canonical\/} if 
\begin{itemize}  [itemsep=4pt, leftmargin=18pt]
\item[a)] the hk-frame $\cA$  is
canonical (Def.\ \eqref{canon_frame}), 

\item[b)]  there exists  a bridge 
$\varphi \colon \cU'   \rightarrow \cU $ that maps a central hk-frame  $\wh \cA$ on $\cU'$  to  
$\cA = \varphi_*(\wh \cA)$ and

\item[c)]  there exists a real submanifold  
$\wh M \subset \{I_2\}\times\{I_{2n}\} \times \bC^{4n}$ passing through $e = (I_2, I_{2n}, 0)$ 
such that the bridge $\varphi$ determines a local equivalence between 
$(\wh \cA, \wh M)$  and  $(\cA, M)$.
\end{itemize}

\end{definition}

\subsection{Correspondence between hk-pairs and pseudo-hyperk\"ahler metrics}
\label{hkpairs-metrics}
\noindent 
Pseudo-hyperk\"ahler metrics and  hk-pairs  are related as  follows.
Consider  an hk-pair  $(\cA, M)$  of signature $(4p, 4q)$, with associated  
holomorphic action  $\r\colon \cU {\times} G \rightarrow G$,  and let 
$$
\r^\t = \r|_{\cU \times G^\t}\colon   \cU \times G^\t \rightarrow \cU\ ,\qquad 
G^\t =  \Sp_1 \times \Sp_{p,q}\ ,
$$
be the induced right action of  $G^\t$. 
The infinitesimal  transformations of $\rho^\t$  are, by construction,  
the real vector fields in  $\  \gg^{\cA,\t} := \a^\cA_{(\bR)}(\gg^\t)$.
In accordance with Def.\ \ref{hk-pair},   the union of $\Sp_{p,q}$-orbits
\beq \label{USP}
 \cU^{(\Sp_{p,q})} = M {\cdot} \Sp_{p,q} :=\bigcup_{x \in M} x \cdot\Sp_{p,q}
\eeq
is a manifold,  $\Sp_{p,q}$-equivariantly diffeomorphic to  $M {\times} \Sp_{p,q}\,$, 
and the $4n$-vectors, 
\begin{align} 
e^{\t}_{I}\vert_{x} & = \a^\cA_{(\bR)}(e^{o\t}_{I})\vert_{x},
\quad  I = 1,\dots, 4n, \ \ x \in M\ ,
\label{tau_basis}
\end{align}
with $I$ labelling the ordered  index pairs $(+1,\dots,+2n, -1,\dots,-2n)$,
belong to   the  vector space
\beq \label{decomp}
T_x \cU^{(\Sp_{p,q})}  = T_x M + \a^\cA_{(\bR)}(\gsp_{p,q}) .
\eeq
Here   $e^{o\t}_I$ denote a choice of  basis vectors   for the
 $4n$-dimensional real $\t$-invariant subspace $V^\t \subset V = \bC^{4n}$. 
By \eqref{decomp} and  $\Sp_{p,q}$-equivariance,  the restrictions to  $\cU^{(\Sp_{p,q})}$ 
of the vector fields  $ e^\t_{I} = \a^\cA_{(\bR)}(e^{o\t}_{I})$ 
are   tangent to  $\cU^{(\Sp_{p,q})}$ at all its points.

Now, we choose a section $\s\colon M \rightarrow \cU^{(\Sp_{p,q})}$
of the trivial bundle
$\pi\colon \cU^{(\Sp_{p,q})} \simeq M \times \Sp_{p,q} \rightarrow M $ and we   
consider the vector fields  on $M$
$$
e^{(\s)}_{I} =  \left\{  e^{(\s)}_{I}|_x  := \pi_*( e^\t_{I}\vert_{\s(x)})\,, 
                                      \quad    x \in M \right\} \  .
$$
There clearly exists   a unique  real analytic, pseudo-Riemannian metric $g$  of 
signature $(4p,4q)$,  for which  the   $(e^{(\s)}_I|_x)$ are vielbeins, i.e.
\beq \label{metric} 
g(e^{(\s)}_{I}, e^{(\s)}_{J}) 
          = (I_4 \otimes \h)_{IJ} \ .
\eeq
The properties of the absolute hk-parallelism $\a^\cA$ imply that for 
any other section $\s'\colon M \to \cU^{(\Sp_{p,q})}$,  the corresponding frames  
$\big(e^{(\s')}_{I}|_x 
:=\pi_*(\left.e^\t_{I}\right|_{\s'(x)} )\big)$ 
are also vielbeins  for this metric, which thus
does not depend on the choice of  $\s$
and  is uniquely associated with the hk-pair  $(\cA, M)$ 
(see Lemma \Aref{sigma-indep}).

\noindent
Moreover, the following proposition holds by construction:
\begin{prop} 
\label{loc.equiv}
Let $(M, g)$ and $(M', g')$ be real analytic pseudo-hyperk\"ahler manifolds  
of signature $(4p,4q)$ corresponding, in the above-described fashion, 
to hk-pairs $(\cA, M)$ and  $(\cA', M')$, respectively. 
Then  $(M, g)$ and $(M', g')$  are locally isometric if and only if 
the  hk-pairs $(\cA, M)$ and $(\cA', M')$ are  locally equivalent.
\end{prop}

\noindent
Further (Theorem \Aref{loc.isom}):\\
 {\it In each local isometry class  of (germs of) real analytic  pseudo-hyperk\"ahler manifolds  of signature 
$(4p,4q)$, there is a  pseudo-hyperk\"ahler manifold $(M, g)$ 
which is determined by  an hk-pair $(\cA, M)$ of signature $(4p,4q)$ in the above-described fashion}.

\noindent
It follows immediately  that:

\begin{theo}\label{bi_hkpairs}
There is a natural one to one correspondence between the local 
isometry classes of (germs of) real analytic pseudo-hyperk\"ahler manifolds  and
the local equivalence classes of (germs of) hk-pairs.
\end{theo}

\subsection{Prepotentials of   pseudo-hyperk\"ahler metrics}  
\noindent 
According to the  above results,  the classification of local isometry classes of  
real analytic  pseudo-hyperk\"ahler metrics  corresponds  to  
the classification  of   local equivalence  classes of  hk-pairs.
The latter is achieved  by means of the following two fundamental results.
 
 \begin{theo} \label{vpot_thm} 
Every local equivalence class of (germs of) hk-pairs contains a canonical hk-pair. 
Moreover,  if  hk-pairs  $(\cA, M)$ and $(\cA', M')$ are both canonical and  
 have identical v-potentials, then $\cA = \cA'$ and 
$(\cA, M)$ and $(\cA', M')$  are locally equivalent.
 \end{theo}

\begin{theo} \label{prepot_thm}
There exists a one-to-one correspondence between canonical hk-pairs  
and holomorphic functions on   harmonic space 
$ \cL_{(+4)} \colon \cH\vert_\cV \rightarrow \bC$ satisfying
the  system of first order equations,
 \beq   \label{spippolo}  
 e^o_{+a}  {\cdot}\,   \cL_{(+4)} = 0\ ,
 \quad
 H_0^o {\cdot}\, \cL_{(+4)} = 4 \cL_{(+4)} \ ,\quad
 \cL_{(+4)}|_{\Sp_1(\bC) \times \{I_{2n}\} \times \{0\}}  = 0\ . 
 \eeq
More precisely, given such an $\cL_{(+4)}$,  there exists a canonical hk-pair 
$(\cA, M)$ on an appropriate subset $\cU\subset\cP$ with $\cU \cap \cH=\cH\vert_\cV$, 
whose $v$-potential is equal to
\beq  \label{spippolino} 
v^{-a}_{++}\vert_{\cU \cap \cH} = \o^{ab} \big(e^o_{-b} {\cdot}\, \cL_{(+4)}\big).
\eeq
Here $(\o^{ab})$ is the inverse  matrix of $(\o_{ab})$.
Conversely, given a canonical hk-pair $(\cA, M)$,  with v-potential 
$ v^{-a}_{++}|_{\cU \cap \cH}$,   there exists a unique holomorphic function 
$\cL_{(+4)}$  satisfying  \eqref{spippolo} and \eqref{spippolino}.
\end{theo}

The holomorphic function $\cL_{(+4)}$  is  the {\bss prepotential}\footnote{Our 
sign convention for the prepotential differs  from the customary one (e.g.\  \cite{gios_book, do1}).} 
of the canonical hk-pair $(\cA, M)$.  The space of prepotentials parametrises
the local equivalence classes of real analytic pseudo-hyper\-k\"ahler manifolds. 
Given an unconstrained prepotential $ \cL_{(+4)}$  satisfying \eqref{spippolo},   
all the vector fields of the associated canonical  hk-pair $(\cA, M)$ may be 
obtained explicitly by solving a  system of partial   differential equations on 
harmonic space $\cH|_\cV$. The corresponding pseudo-hyperk\"ahler manifold can then be 
determined according to the procedure of  Sect.\ \ref{hkpairs-metrics}.
Since the  equivalence classes of (germs of) hk-pairs are  in one to one 
correspondence with the (germs of) real analytic pseudo-hyperk\"ahler metrics 
(Theorem \ref{bi_hkpairs}) and each of them contains a canonical hk-pair 
(Theorem \ref{vpot_thm}),  the  parametrisation of pseudo-hyperk\"ahler metrics  
advertised in the Introduction is established.
 
In the next section we discuss some technical  properties of holomorphic functions on 
$\Sp_1(\bC)$, which are essential in  our discussion.  We  then  prove 
Theorem \ref{vpot_thm} in  Sect.\ \ref{pf_vpot_thm} and 
Theorem \ref{prepot_thm}  in    Sect.\ \ref{pf_prepot_thm}.  
In Sect.\ \ref{reconstruct},  we describe a five-step recipe  for the explicit 
construction of a pseudo-hyperk\"ahler metric from its prepotential. 

\section{Holomorphic functions on $\Sp_1(\bC)$}
\label{harmonics}
 \noindent 
 Consider the standard coordinates of $\GL_2(\bC)$,
 $$
 (u^1_{+}, u^2_{+}, u^1_-, u^2_{-}) \colon \GL_2(\bC) \longrightarrow \bC^4 ,
 $$
which associate with every matrix 
$U = \left(\smallmatrix u^1_+ & u^1_-\\ u^2_+ & u^2_- \endsmallmatrix \right) $ 
the values of its entries, and  the class of meromorphic functions  
$\ h\colon \GL_2(\bC) \rightarrow \bC$ of the form 
\beq \label{mero} 
h = \sum_{ p,q,r,s \in \bZ} c_{pqrs} 
                                            (u^1_+)^{p} (u^2_+)^{q} (u^1_-)^{r} (u^2_-)^{s} .
\eeq
Two such maps $h,h'$ are called  
{\bss Sp$_{\mathbf{1}}{\boldsymbol{(\bC)}}$-equivalent}  
if $\ h|_{\Sp_1(\bC)} = h'|_{\Sp_1(\bC)}$.  Since the elements  
$\  U=(u^i_\pm) \in \Sp_1(\bC)\subset \GL_2(\bC)$ are constrained by 
\beq \label{det1} 
\det U = u^1_+ u^2_- - u^2_+ u^1_- = 1\,,\eeq
any  one coordinate from  $\{u^1_{+}, u^2_{+}, u^1_-, u^2_{-}\}$  is 
$\Sp_1(\bC)$-equivalent to a rational function of the other three. It follows that  
every function \eqref{mero}  is $\Sp_1(\bC)$-equivalent to four others,  
obtained by expressing each of the four coordinates  in terms of  the  others 
in accordance with \eqref{det1}.  The meromorphic functions obtained in this way are said to be 
in {\bss reduced form}. Clearly,  each $\Sp_1(\bC)$-equivalence class of meromorphic 
functions \eqref{mero} contains at most four distinct functions in reduced form. 

The meromorphic functions of the form \eqref{mero} are   related to the
holomorphic functions of $\Sp_1(\bC)$. Indeed, we have the following:

\begin{lemma} \label{Sp1_restr} 
Any  holomorphic function $\,g\colon \Sp_1(\bC) \,{\rightarrow}\, \bC\,$  is 
a restriction of   some meromorphic function  $\,h\colon \GL_2(\bC) \,{\rightarrow}\, \bC\,$ 
of the form \eqref{mero},  $\,g = h|_{\Sp_1(\bC)}$.
Further, if the function $h$    thus  associated with $g$ is in reduced form, 
its coefficients $c_{pqrs}$  are uniquely  
determined by  the  expansion of  $\, g|_{\Sp_1}$ in  generalised spherical functions. 
\end{lemma}

\begin{proof} Since $\Sp_1$ is  a three dimensional,  totally real submanifold of 
the three dimensional complex manifold $\Sp_1(\bC)$, the  holomorphic map  
$\,g\colon \Sp_1(\bC) \rightarrow \bC\,$ is uniquely determined by its restriction 
$ g|_{\Sp_1}$, which is  of class $\cC^\infty$ and hence in $L^2(\Sp_1)$. 
This implies that  $ g|_{\Sp_1}$ admits a unique series expansion in terms of 
generalised spherical functions (see e.g.\  \cite{GMS} p. 94).
We recall that  these are the functions  
$$
T^\l_{mn}\colon \Sp_1 \rightarrow \bC\,,\quad \text{with}\ 
\l = \frac{\ell}{2}\ ,\ \ell \in \bN\ ,\ m , n = - \ell, -\ell + 1, \ldots, \ell - 1, \ell\,,
$$
which associate with every  $U  = (u^i_{\pm}) \in \Sp_1$  the $(m, n)$-element  
of the matrix  $T^\l(U)$  representing the action of  $U$ on   the  (unique, up to an isomorphism) 
irreducible $\Sp_1$-module of highest weight $\l$. Since every such irreducible 
$\Sp_1$-module is  a symmetric power  of  the standard module $\bC^2$, 
the entries  of  $T^\l(U)$ are   polynomials in  the entries of $U$. 
A  generalised spherical function $T^\l_{mn}$ is  therefore a  polynomial in these variables and has an expression of the  form
 \beq \label{spherical} 
 T^\l_{mn} =  \sum
     T^\l_{mn|pqrs} ( u^1_+)^{p} ( u^2_+)^{q} ( u^1_-)^{r} ( u^2_-)^{s} \big\vert_{\Sp_1}\ ,
\quad  T^\l_{mn|pqrs} \in \bC\ .
\eeq
Since the   restrictions $u^i_\pm|_{\Sp_1}$'s  are constrained by \eqref{det1}, the coefficients in the 
expansion \eqref{spherical}  are  in general not uniquely  determined by  
the spherical function $T^\l_{mn}$. 
However,  replacing one of the functions $u^i_\pm|_{\Sp_1}$  by a  rational expression of the others, 
one can always   reduce to an expression  
for  $T^\l_{mn}$ as 
a Laurent  series of the other three functions.  
Summing up,  a  spherical function $T^\l_{mn}\colon \Sp_1 \rightarrow \bC$
admits at most four  specific  expansions 
 \eqref{spherical}, 
each of them equal to   the  restriction $ \cT^\l_{mn}|_{\Sp_1}$ of  a
meromorphic function $\cT^\l_{mn}\colon \GL_2(\bC) \rightarrow \bC$  
as in  \eqref{mero} and 
 in reduced form. 
We therefore have  that    $ g|_{\Sp_1}$ can be expanded in a  series of the form
\beq \label{expansion} 
g\vert_{\Sp_1} = \sum_{\l, m, n} c_{\l}^{mn} \cT_{mn}^\l \big\vert_{\Sp_1} 
=  \left(\sum c_{pqrs}  
             (u^1_+)^{p} (u^2_+)^{q} ( u^1_-)^{r} (u^2_-)^{s}\right)\Big\vert_{\Sp_1}\ ,
\eeq
where the coefficients   $\, c_{pqrs}  = \sum c_{\l}^{mn} T_{mn|pqrs}^\l$  are completely 
determined by the coefficients  $c_{\l}^{mn}$ of the expansion of  $g|_{\Sp_1}$ in 
generalised spherical functions if  the  $T_{mn|pqrs}^\l$  are coefficients  of a 
reduced form of the maps  $\cT^\l_{mn}\colon \GL_2(\bC) \rightarrow \bC$.  

Consider now the  function $h =  \sum c_{pqrs}  
             (u^1_+)^{p} (u^2_+)^{q} ( u^1_-)^{r} (u^2_-)^{s}$ on $\GL_2(\bC)$. 
Being meromorphic, it is holomorphic  
on a dense open subset  $\cU\subset \Sp_1(\bC)$. Since  $h|_{\Sp_1} = g|_{\Sp_1}$, 
it  follows that  $h|_{\cU} = g|_{\cU} $, so that, by continuity,  $h|_{\Sp_1(\bC)} = g $, proving  the first claim of the lemma.
The second claim  follows from the construction of $h$ and the fact that $h$ is in reduced 
form if and only if  all  meromorphic  functions $\cT^\l_{mn}$ appearing 
in \eqref{expansion}  are taken in reduced form. 
\end{proof}

We now solve  certain equations for holomorphic  functions 
$f\colon \Sp_1(\bC) \rightarrow \bC$.

\begin{lemma} \label{genlem} 
 \begin{itemize} [itemsep=2pt, leftmargin=18pt]
\item[i)] 
Every solution of   
\beq \label{h++eq} 
H^o_{++} {\cdot} f = 0 
\eeq  
is a  restriction $f  = h|_{\Sp_1(\bC)}$  of a holomorphic  map  
$h\colon \GL_2(\bC) \rightarrow \bC$ of the form
\beq \label{gs1} 
h = \sum_{ n,m \ge 0}  c_{nm} (u^1_+)^{n} (u^2_+)^{m}.
\eeq
 
\item[ii)]  Every solution of    
\beq\label{h0eq} 
H^o_0 {\cdot} f = k f\  ,\ k \in \bZ\,,
\eeq 
is a  restriction $f  = h|_{\Sp_1(\bC)}$  of a meromorphic  map
$h\colon \GL_2(\bC) \rightarrow \bC$  of the form
\beq \label{gs2}  
h = \sum_{ \substack{ n,m,p,q  \,\in\, \bZ \\
                                      n+ m - p - q  \,=\, k}}
                  c_{nmpq} (u^1_+)^{n} (u^2_+)^{m} (u^1_-)^{p} (u^2_-)^{q}. 
\eeq
\end{itemize}
\end{lemma}

\proof
A holomorphic function $f$ on $\Sp_1(\bC)$ is of  the form  $f  = h|_{\Sp_1(\bC)}\,$,  
for some meromorphic  $h$  in reduced  form (Lemma \ref{Sp1_restr}). 
The  coordinate expression  for $H^o_{++}$  \eqref{coord_vfs} and holomorphicity 
imply  that $f$ satisfies $\ H^o_{++} {\cdot} f = 0\ $ if and only if $h$ has the 
form \eqref{gs1}, proving  i).   A  similar argument proves  ii). 
\qed

\noindent
Simultaneous solutions of \eqref{h++eq} and \eqref{h0eq} may now be  
 constructed.  More generally:
\begin{lemma} \label{lemma_split}
Let   $g\colon \Sp_1(\bC) \rightarrow \bC$ be holomorphic. The system of equations for a holomorphic 
function $f\colon \Sp_1(\bC) \rightarrow \bC$, 
\beq\begin{split}
H^o_0 {\cdot} f &=   k f \,,\quad   k \in \bZ\,,\\
 H^o_{++} {\cdot} f  &=  g\,,
 \end{split}
\label{c}\eeq  
admits solutions if and only if  
$g$ satisfies the equation,
\beq \label{conditionbis}
H^o_{0}{\cdot} g =   (k+2)g\,. 
\eeq
 If \eqref{conditionbis}  holds, the set of solutions to \eqref{c} is as follows. 
\begin{itemize} [itemsep=2pt, leftmargin=18pt]
\item[a)]  For $k < 0$ there exists exactly one holomorphic solution.
\item[b)]  For  $k \geq 0$, the solutions are precisely all  the functions   
$f = h|_{\Sp_1(\bC)}$ determined  by holomorphic maps $h\colon \GL_2(\bC) \rightarrow \bC$ of the  form  
\beq \label{gs3} 
h = h_o + \sum_{\substack{ m,n \ge 0 \\ m+n = k }}   
                          c_{mn} (u^1_+)^{m} (u^2_+)^{n},
\eeq
where  $f_o = h_o|_{\Sp_1(\bC)}$ is   some solution of  \eqref{c}. 
\end{itemize}
\end{lemma}

\begin{proof}  If a solution  $f$ to \eqref{c} exists,  then 
$$
H^o_{0}{\cdot} g \ =\ [H^o_0, H^o_{++}]{\,\cdot} f + H^o_{++}{\cdot\,}(H^o_0 {\cdot} f)
                              \ =\  2 H^o_{++}{\cdot} f +  k H^o_{++}{\cdot} f \ =\   (k+2)g\ ,
$$
so  \eqref{conditionbis} is a necessary  condition for the integrability of the system \eqref{c}. 
On the other hand, by Lemma \ref{genlem} ii)  we have that 
$$ 
g =\wt g|_{\Sp_1(\bC)}\ ,\quad
\wt g \ =  \sum_{ p + q  - r - s \, =\, k+2} 
              c_{pqrs} (u^1_+)^{p} (u^2_+)^{q} (u^1_-)^{r} (u^2_-)^{s}.
$$
Now,   by integration of $H^o_{++} {\cdot} h_o  =\wt g$  we obtain the series,
$$ 
h_o = \sum_{ p + q  - r - s \,=\, k + 2} c_{pqrs}\  
                (u^1_+)^{p - 1} (u^2_+)^{q-1 } (u^1_-)^{r } (u^2_-)^{s}  
          \left(  \frac{ u^2_+ u^1_-  }{r+1}  +   \frac{ u^1_+ u^2_- }{s+1} \right).
$$
This converges  uniformly  to  a holomorphic solution  of  \eqref{gs3}  on relatively  compact 
neighbourhoods of  the  points of  $\Sp_1(\bC) {\setminus} \mathcal Y$, where
$ \mathcal Y := \{ (u^i_\pm)\in \Sp_1(\bC)\, \vert  \,  u^1_+  u^2_+  \,{=}\,  0  \}$.
Moreover, since there is no element of $\Sp_1(\bC)$,  on which $u^1_+$ and $u^2_+$ 
are both zero,  for  any $U = (u^i_{\pm}) \in \mathcal Y$ we may replace  the meromorphic 
functions $\wt g$ and  $h_o$ by  equivalent   functions  $\wt g'\,,\, h_o'$  in reduced form, 
both  independent  of either  $u^1_+$ or  $u^2_+$ and  hence  with no singularity  
at the chosen $U \in \mathcal Y$.  This means that the  functions $\wt g|_{\Sp_1(\bC)}$ and $h_o|_{\Sp_1(\bC)}$ are holomorphic on $\Sp_1(\bC) \setminus \mathcal Y$ and 
extendable to  all points of $\mathcal Y$, i.e. $\mathcal Y$ is a  set of removable singularities for them. Thus,  $f_o = h_o|_{\Sp_1(\bC)}$  is  a  solution of  \eqref{gs3} which is  holomorphic 
everywhere on $\Sp_1(\bC)$.
To complete the proof it suffices to observe that if  both  
$f_o\,, f\colon  \Sp_1(\bC) \rightarrow \bC$  satisfy  \eqref{gs3}, their difference 
$\d f = f - f_o$ satisfies  \eqref{h++eq} and  \eqref{h0eq}. Therefore it is equal to
$\d f = \d h|_{\Sp_1(\bC)}$ for some  $\d h\colon \GL_2(\bC) \rightarrow \bC$ of the form  
$\d h =  \sum_{\substack{ m,n \ge 0 \\  m+n = k }}   
                          c_{mn} (u^1_+)^{m} (u^2_+)^{n}$.  From this,  (a) and (b) follow immediately.
\end{proof}
 
We now consider an initial value problem for an important  generalisation of the 
system \eqref{c} to  harmonic space 
$\cH|_{\cV} = \Sp_1(\bC) {\times} \{I_{2n}\} {\times} \cV$, 
$\cV \subset \bC^{4n}$. In what follows,   we represent  the elements  $(U, I_{2n}, z) \in \cH|_{\cV}$  
simply  as  $(U, z)$.  

\begin{lem} \label{charge2_lem} 
Let   $\,\cV \subset \bC^{4n} $ be a simply connected 
open neighbourhood of  $\ 0 \in \bC^{4n}$. 
The   system of differential equations on  
$\cH|_{\cV}  := \Sp_1(\bC) \times \cV$
for holomorphic maps  
$\, k = (k^{ic})\colon   \cH|_{\cV}  \rightarrow  \bC^{4n} $,
$ i=1,2,\ a, b, c=1,\dots,2n$,
\beq 
\label{frob1} 
H_0^o \cdot k = 0 \ ,\quad 
H^o_{++} \cdot k\vert_{(U, z)} = F(U,k(U,z)),
\eeq
where
$
F = (F^{ia})\colon  \Sp_1(\bC) \times \bC^{4n}  
                                                          \rightarrow  \bC^{4n}
$
is holomorphic and satisfies  the integrability condition
\beq \label{charge2} 
H_0^o {\cdot} F = 2 F\,,
\eeq
admits
\begin{itemize} [itemsep=6pt, leftmargin=18pt]
\item[A)]  {\em local} solutions around any point   $(u^i_{o\pm}, z^{ia}_o) \in \cH|_{\cV}\,$
 with  arbitrary  initial conditions  $  k^{jd} (u^i_{o\pm}, z^{ia}_o) = c^{jd} $ 
and 
\item[B)]  a {\rm global} solution $k$  on  $\cH|_{\cV} $ for any choice of   initial values 
$k|_{(I_2, z)}  = \wh k(z)$, $z \in \cV$, having   the property 
\beq 
k (U, z) =    k(\psi(U), z)\ ,\qquad \psi(U) := (U^T)^{-1}.
\label{reality_cond1}
\eeq 
\end{itemize}
 \end{lem}

\begin{rem} The latter property  is merely a tool in our  proof of the existence of  
global solutions. It is by no means  true  that  every global solution  on $\cH|_{\cV}$ has this property. 
 \end{rem}
 
\noindent 
\textit{Proof of A)}\,. 
 The existence  of  a solution  $k$ of \eqref{frob1}  is equivalent to the 
existence  of a certain special  submanifold $S^{(k)}$ corresponding  to  the graph of $k$ 
in  the cartesian product $\cN :=\cH|_{\cV} \times \bC^{4n} $.  
Denote the  coordinates  of $ \cN$ by $(u^i_\pm, z^{ia},  w^{jd})$ and  the
standard projections  onto its factors by   $\pi_1\colon  \cN  \rightarrow \cH|_{\cV}\,$
and $\pi_2\colon  \cN \rightarrow \bC^{4n}$.
Let  $\wh F\colon  \cN \ra  \bC^{4n}$  be the   map 
$\ \wh F(u^i_\pm, z^{ia},  w^{jd}) := F(u^i_{\pm},  w^{jd})$
and  $\wh H_0$ and $\wh H_{++}$ the  vector fields on $ \cN$  
uniquely determined   at $x \in  \cN$ by the conditions,
\begin{align*} 
\pi_{1*}(\wh H_0|_x)  &=   H^o_0\vert_{\pi_1(x)}\ , &
\pi_{1*}(\wh H_{++}|_x)  &=   H^o_{++}|_{\pi_1(x)}\ ,   \\[6pt]
 \pi_{2*}( \wh H_0\vert_x)  & =   0\ , &
\pi_{2*}(\wh H_{++}|_x)  
                          &= \wh F^{ia} (x)  \frac{\partial}{\partial w^{ia}}\Big|_{\pi_2(x)} .
\end{align*}
Further, let $\cD \subset T \cN$  be the complex distribution  generated by   
$\wh H_0$ and $\wh H_{++}\,$.  We immediately  see that  a map  
$k \colon \cW \subset \cH|_{\cV} \rightarrow   \bC^{4n} $
is a  solution of \eqref{frob1} if and only if the vector fields of $\cD$ are everywhere tangent 
to the graph  $S^{(k)} \subset  \cN$ given by
$$
S^{(k)}:= 
\{ (u^i_{\pm}, z^{ia},  w^{jd})\ \vert \   (u^i_{\pm}, z^{ia}) \in \cW\ , 
\  w^{jd}   = k^{jd}(u^i_{\pm}, z^{ia})  \}.
$$
We have:
$$  [\wh H_0, \wh H_{++}]  
= 2 \wh H_{++} +  (\wh H_{0} {\cdot} \wh F^{ia}   
 - 2  \wh F^{ia}) \frac{\partial}{\partial w^{ia}} \  .$$
From \eqref{charge2}  $ \wh H_{0} {\cdot} \wh F^{ia}   - 2  \wh F^{ia}\, {=}\, 0 $, so
 $\cD$ is involutive. 
Let $x_o \, {=}\, (u^i_{o\pm} , z^{ia}_o,  c^{jd}) \in  \cN$ and choose  
a disk $\D_\varepsilon \subset \bC$ of radius $\varepsilon$ and centre 0, 
a neighbourhood $\cV' \subset \cV$ of $(z^{ia}_o)$ 
and  a  holomorphic map (possibly constant) 
$\wh k\colon\D_\varepsilon \times \cV' \to \bC^n$ 
such that   $\wh k^{jd}(0, z^{ia}_o) = c^{jd}$. 
If $\varepsilon$ is sufficiently small, the set
$$
\cT :=  \big\{ 
           \big(\, u^i_{\pm}= \exp(\z H^o_{--}) {\cdot} u^i_{o\pm}\ ,\  z^{ia}\ ,\  
                         w^{jb}= \wh k^{jd} (\z, z^{ia})\, \big) \ 
                         \text{with} \     \z \in \D_\varepsilon\ , z^{ia} \in \cV'  
         \big\}
$$
is a  $(4n {+} 1)$-dimensional $\cD$-transversal complex submanifold of $\cN$.
By  the complex Frobenius Theorem,  there exists a  family of  two-dimensional 
integral leaves of $\cD$, each  passing through a distinct  point of $\cT$, 
which combine to form a complex manifold  of dimension $4n {+}3$ 
with the property that the vector fields in  $\cD$ are everywhere tangent to it. 
This  submanifold is   the graph $S^{(k)}$ of a map $k$
in a neighbourhood of  $(u^i_{o\pm} , z^{ia}_o)$  such that  
 $  k^{jd}(u^i_{o\pm}, z^{ia}_o) = c^{jd}$.  
This is one  of the required local solutions.
 \vskip 6pt plus 2pt minus 4pt

\noindent
\textit{Proof of B)}\,. 
 We now turn to the existence of global solutions. 
We recall that 
the standard transitive action of $\Sp_1(\bC)$ on $\bC P^1$ yields
a natural identification $\bC P^1 \simeq \Sp_1(\bC)/\cB$, where  $\cB$
is the Borel subgroup   formed by upper triangular matrices in $ \Sp_1(\bC)$,
$$
\cB := \left\{\ \begin{pmatrix}\l & \mu \\ 0 & \l^{-1}\end{pmatrix} \ ,\ 
                           \l \in \bC^*\ , \mu \in \bC\ \right\} \simeq \bC^* \times \bC .
$$
The affine subspaces of $\bC P^1$  
$$
\bC_{(0)} = \{ [1: \z] \,;\,  \z \in \bC \}\quad  \text{and}\quad 
\bC_{(\infty)} =  \{ [\z: 1]\,;\,  \z \in \bC \},
$$
can be identified with the    cosets in $\Sp_1(\bC)/\cB$ given by the points of 
\begin{align*}
\wt \bC_{(0)} 
&:= \left\{\ \begin{pmatrix}1 & 0\\ \z & 1\end{pmatrix}\ ,\ \z \in \bC\ \right\} 
     = \exp(\bC H^o_{--}) \\
\intertext{and}
 \wt \bC_{(\infty)} 
 &:= \left\{\ \begin{pmatrix} \z & -1\\ 1 & 0\end{pmatrix}\ ,\ \z \in \bC\ \right\} 
      = J_o {\cdot} \exp(\bC H^o_{--})\,,\quad  
                      J_o = \begin{pmatrix}0 & -1\\ 1 & 0 \end{pmatrix},
\end{align*}
respectively. This means  that $\cH|_{\cV} $ is the union of the two patches,
 \beq \label{coset} 
\cH|_{\cV} =  \Sp_1(\bC) \times \cV 
 = (\wt \bC_{(0)} \times \cV) {\cdot} \cB \cup (\wt \bC_{(\infty)} \times \cV) {\cdot} \cB\,,
\eeq
with their intersection (a tube over an annulus) having two equivalent descriptions, 
 \begin{align}
 (\wt \bC_{(0)} \times \cV) {\cdot} \cB \cap (\wt \bC_{(\infty)} \times \cV) {\cdot} \cB  
&=  (\exp(\bC^* H^o_{--}) \times \cV) {\cdot} \cB    
\nonumber \\
&=   (J_o {\cdot} \exp(\bC^* H^o_{--}) \times \cV) {\cdot} \cB\,.
 \label{intersection}  
 \end{align}
Now, every $\cB$-orbit $x {\cdot} \cB\,,\, x \in \cH|_\cV$, is  biholomorphic to 
$B \simeq \bC^* \times \bC$ and  the non-trivial elements of its fundamental group  
$\pi_1(x {\cdot} \cB)$  are given by  the  $H^o_0$-orbits in $x {\cdot} \cB$.  
A local solution of \eqref{frob1} is constant along any open subset of an   
$H^o_0$-orbit. 
 Using this and the existence of local solutions around every point, we  see that 
 if there exists a solution $k$  on  a given simply connected open subset  
 $\cS \subset \cH|_{\cV}$, it can always be extended to a solution  defined on  
 the  union of $\cB$-orbits 
$\cS {\cdot} \cB := \bigcup_{y \in \cS}  y {\cdot} B\,$.

Now consider a  simply connected  subset $\cZ \subset  \cH|_{\cV}$ transversal to the $B$-orbits.  
It may be covered by a collection of open sets  $\cW_x \subset  \cH|_{\cV}$, $x \in \cZ$, 
each admitting, by part A, a local solution  with arbitrary initial data on $\cW_x \cap \cZ$.  
The initial conditions can be chosen so that the solutions  agree on 
non-empty intersections $\cW_x \cap \cW_{x'}$, $x' \neq x$. 
By the simple connectedness of $\cZ$  these solutions  combine to give  a  solution  
on a neighbourhood of  $\cZ$ for any choice of initial data  $\wt k$ on $\cZ$.  
Such a solution uniquely extends to  $\cZ {\cdot} \cB$.

Since $\wt \bC_{(0)} \times \cV$  is simply connected and transversal to $\cB$-orbits   
it follows that   for any choice of  data on 
$ \wt \bC_{(0)} \times \cV$
$$ 
\wt k\colon \wt \bC_{(0)} \times \cV = \exp(\bC H^o_{--}) \times \cV\longrightarrow \bC^{4n} 
$$  
there is  a unique solution   $k$  on   the collection of $B$-orbits
$ (\wt \bC_{(0)} \times \cV) {\cdot} \cB$   
with   $k \vert_{\wt \bC_{(0)} \times \cV} =   \wt k$. 
We now prove  that  the  solution $k$ satisfies $k(U, z) = k(\psi(U),z)$ and 
$k(I_2,z) = \wh k(z)$, provided that   $\wt k$  is chosen appropriately. 

Let $K\colon B\times \cV \to \bC^{4n}$ be the unique holomorphic function satisfying   
\eqref{frob1} at the points of  the $B$-orbits  $(I_2, z) {\cdot} B$, $z \in \cV$,  with  initial data 
$K(I_2, z) = \wh k(z)$. Then set $\wt k$ to be the unique  holomorphic function 
at the points   $(\exp(\z H_{--}), z)$ such that $\wt k(I_2, z) = \wh k(z)$ and  
$$
H^o_{--}{\cdot}\wt k(\exp(\z H_{--}), z) 
         = - F(\psi(\exp(\z H^o_{--})), K(\psi(\exp(\z H^o_{--})), z))\,.
$$

Now consider the modified differential problem on maps $h\colon \cH|_\cV \,{\to}\,\bC^{4n}$
\beq
\begin{split}  
H^o_0 {\cdot} h|_{(U, z)} &= 0\\
 H^o_{++} {\cdot} h|_{(U, z)} &=  F(U\,,\,h(U,z))\\
 H^o_{--} {\cdot} h|_{(U, z)} &= - F( \psi(U)\,,\,h(\psi(U),z))\,.
 \end{split}
 \label{frob1-mod}
\eeq
Note that  \eqref{frob1-mod} is simply  \eqref{frob1}  with  the addition of a third equation,  which  is non-local; the right hand side 
depends on the value of $h$ at the shifted point $(\psi(U), z)$. 
However, every solution $h$ of \eqref{frob1} satisfying   $h(U, z) = h(\psi(U), z)$, 
necessarily  satisfies   the  third equation of \eqref{frob1-mod} also. Indeed, since
$\psi_*(H^o_0) = - H^o_0$,  $\psi_*(H^o_{\pm\pm}) =  -  H^o_{\mp\mp}$
(see  \eqref{bistecca}), we have:
\begin{align*} 
H^o_{--} {\cdot} h\big|_{(U, z)} 
&= H^o_{--} {\cdot} h(\psi(\cdot), \cdot))\big|_{(U, z)} 
=  (\psi_*(H^o_{--}) {\cdot} h)\big|_{(\psi(U), z)} \\
&  =  - (H^o_{++}{\cdot} h)\big|_{(\psi(U), z)} 
= - F(\psi(U)\,,\,h(\psi(U),z))\ .
 \end{align*}
The  solution  $k$  of \eqref{frob1},  which we constructed on  
$ (\wt \bC_{(0)} \times \cV) {\cdot} \cB$   
with initial data  $k|_{\wt \bC_{(0)} \times \cV} =   \wt k$, clearly   solves  the first two equations of \eqref{frob1-mod} and  by the above choice of $\wt k$, it also satisfies  the third equation   at  the points $(\wh U, \wh z) \in \wt \bC_{(0)} \times \cV$.  Further, on the points   $(U^{(\l)}, z) = (\wh U, \wh z) {\cdot} \exp(\l H^o_{++})$, $\l \in \bC$,   of their $H^o_{++}$-orbits  we have 
\begin{align*}
H_{--}^o  {\cdot} k|_{(U^{(\l)}, z)} \!\!
& = H_{--}^o  {\cdot} k|_{(\wh U, \wh z)}  
        + \int_0^\l H^o_{++}  {\cdot} H_{--}^o  {\cdot} k|_{(U^{(\mu)}, \wh z) }  \,d\mu \\
& = H_{--}^o  {\cdot} k|_{(\wh U, \wh z)}  
          + \int_0^\l \left(H^o_0 + H^o_{--}  {\cdot} H_{++}^o \right) {\cdot} k|_{(U^{(\mu)}, \wh z)}  \,d\mu 
\\
& =  - F(\psi(\wh U), k(\psi(\wh U), \wh z)) 
             \,{+} \int_0^\l H^o_{--}  {\cdot} F(\cdot, k(\cdot, \cdot)) |_{(U^{(\mu)}, \wh z) } \, d\mu 
\\
&  =  {-} F(\psi(\wh U), k(\psi(\wh U), \wh z)) 
         \,{-} \int_0^\l\!\! 
         H^o_{++}  {\cdot} F(\psi(\cdot), k(\psi(\cdot), \cdot)) |_{(\psi(U^{(\mu)}), \wh z) } d\mu 
\\
&  = -  F(\psi(U^{(\l)}), k(\psi(U^{(\l)}), z))\ . 
\end{align*}
Thus $k$ solves the third equation in \eqref{frob1-mod} at   points of 
$(\wt \bC_{(0)} \times \cV) {\cdot} \exp(\bC H^o_{++})$ as well. 
A  similar argument shows that it solves the third  equation also at the   points  of the  $H^o_0$-orbits in 
$$
((\wt \bC_{(0)} \times \cV) {\cdot} \exp(\bC H^o_{++})) {\cdot} \exp(\bC H^o_0) 
    = ( \wt \bC_{(0)} \times \cV)\cdot \cB\,.
$$
So, $k$ solves \eqref{frob1-mod}   at {\it all} points of its domain. 

Now, the  new map  $k'(U,z) := k(\psi(U), z)$ satisfies
\begin{align*}
& H_0^o {\cdot}  k' |_{(U, z)}  
= -   (H^o_0 {\cdot} k)|_{(\psi(U), z)} = 0 
\\
& H^o_{++}{\cdot} k'  |_{(U, z)}=  \psi_*(H^o_{++}){\cdot}  k|_{\psi(U, z)} 
    = - H^o_{--}{\cdot}   k|_{(\psi(U), z)} = F(U, k(U, z)) 
\\
& H^o_{--}  {\cdot}  k'|_{(U, z)} = \psi_*(H^o_{--}) {\cdot}  k|_{(\psi(U), z)}   
           = -  H^o_{++} {\cdot} k |_{ (\psi(U), z)} = - F (\psi(U)\,,\, k(\psi(U), z)) . 
\end{align*}
So,  $k$ and   $k'$ are both solutions of the system
\beq
\begin{split}  
\label{frob1-final}
H^o_0 {\cdot} h|_{(U, z)} &= 0\\
 H^o_{++} {\cdot} h|_{(U, z)} &=  F(U\,,\,k(U,z))\\
 H^o_{--} {\cdot} h|_{(U, z)} &= - F( \psi(U)\,,\,k(\psi(U),z))
 \end{split}
\eeq
 with identical initial data
$k'|_{\{I_2\} \times \cV} = \wh k = k|_{\{I_2\} \times \cV}$. We thus have  $k' \equiv k$ by  
the uniqueness of  local solutions  of \eqref{frob1-final} and  the connectedness of the domain 
$( \wt \bC_{(0)} \times \cV){\cdot} \cB$. This concludes the proof that  $k(U, z) = k(\psi(U), z)$.

We now show that the solution $k$ extends holomorphically to a solution defined 
on all of $\cH|_{\cV}$. Since  $ k(U,z) = k (\psi(U), z)$,   we have
$$
k|_{\left(\left(\smallmatrix \z & \z -1\\ 1 & 1\endsmallmatrix \right) , z\right)} 
 =  k|_{\left(\left(\smallmatrix 
                                              1  & - 1 \\  1-  \z & \z 
                                 \endsmallmatrix \right),   z\right)}
\qquad  \text{for}\  \z \in \bC^*\ . $$
Since  $k$ 
 is holomorphic  on  $(\wt \bC_{(0)} \times \cV) {\cdot} \cB$ 
 and
$\left(\left( \smallmatrix 1& - 1\\   1 -  \z & \z \endsmallmatrix \right),    z\right)
=  \left(\left(\smallmatrix 1 & 0\\ 1- \z & 1 \endsmallmatrix \right) 
        \left(\smallmatrix  1& -1\\ 0 &1 \endsmallmatrix \right),    z\right)$ 
belongs to  $(\wt \bC_{(0)} \times \cV) {\cdot} \cB$  for any  $\z \in \bC$ 
(including $\z = 0$),  it follows that   for every $z \in \cV$  the   map on $\bC^*$
$$
\z  \longmapsto  
k \vert_{\left(\left(\smallmatrix \z  & \z -1\\ 1 & 1\endsmallmatrix \right) , z\right)} 
$$
admits a holomorphic extension  to   $\z = 0$.  
Now, the $\cB$-orbits of the  points  
$\left(\smallmatrix \z & \z -1\\ 1 & 1\endsmallmatrix \right) 
= \left(\smallmatrix \z &  -1\\ 1 & 0\endsmallmatrix \right)  
                   \left(\smallmatrix 1 &  1\\ 0 & 1\endsmallmatrix \right)$ 
generate the entire set 
$$ 
(\wt \bC_{(\infty)} \times \cV) {\cdot}\cB 
= \bigcup_{\z \in \bC, z \in \cV}
          \left(\left(\smallmatrix \z &  -1\\ 1 & 0\endsmallmatrix \right) 
          \left(\smallmatrix 1 &  1\\ 0 & 1\endsmallmatrix \right), z\right) {\cdot} \cB\ .
$$
So we may take   solutions of \eqref{frob1} along  $\cB$-orbits  having  
$ k \big|_{\left(\left(\smallmatrix \z & \z -1\\ 1 & 1\endsmallmatrix \right) , z\right)}$ 
as initial values and combine them into a  holomorphic extension of    $k$  to 
$(\wt \bC_{(\infty)} {\times} \cV) {\cdot} \cB $. 
In virtue  of   \eqref{coset} and  \eqref{intersection},  $k$ extends  to  
$\cH|_{\cV} $ and, by continuity, it  satisfies 
 \eqref{reality_cond1}  everywhere. 
\qed

 \begin{rem}\label{remark_initial_data} 
 Given a global solution  $k = (k^{ic} )\colon\cH\vert_\cV \to \bC^{4n}$ of  \eqref{frob1}, we  set
 $$k^{\pm c}(U,z):=  - u^{\pm}_i k^{ic}(U,z)\ ,\qquad\ U = (u^i_\pm)\in \Sp_1(\bC)\ ,\ z \in \cV .$$
For any  $k^{\pm c}$,  the corresponding $k^{ia}$ are recovered using the inverse formula $k^{ia} = u^i_+k^{+ a}+ u^i_-k^{- a}$.
The lemma says that 
there exists  a  global solution to \eqref{frob1}, \eqref{reality_cond1}  
for any choice of initial values  $ \wh k^{\pm a} = k^{\pm a}|_{\cV\times\{I_2\}}\colon\cV\times\{I_2\} \ra \bC^{2n}$.
\end{rem}

\section{The existence of  canonical hk-pairs}  
\label{pf_vpot_thm}
\centerline{\smc Proof of Theorem\ \ref{vpot_thm}}
\setcounter{equation}{0}
\smallskip
\noindent
To begin, we need the following:

\begin{lem} \label{central} 
In every  local equivalence class  of  hk-pairs   
of signature $(4p, 4q)$  there exists 
 an hk-pair  $(\cA, M)$ with  $ \cA $ central 
(Def.\ \ref{centralframe})
and $M \subset   \{I_2\} \times \{I_{2n}\} \times \bC^{4n}$.
\end{lem}

\begin{proof} 
As shown in Sect. \ref{hkpairs-metrics}, every hk-pair
$(\cA, M)$ is associated with a pseudo-hyper\-k\"ahler metric $g$ on $M$. We may then use  a local system of coordinates to identify  $M$ with an open subset $M' \subset \bR^{4n}$,  $g$ with a pseudohyperk\"ahler metric $g'$ on $M'$ and the hk-frame $\cA$ with the  hk-frame $\cA'$ of  
holomorphic extensions of  vertical and horizontal  vector fields 
of the covering of  the  holonomy  bundle of $(M', g')$ with structure group $\Sp_1 \times \Sp_n$. 
This means that $(\cA, M)$ is locally equivalent to $(\cA', M')$ and the  explicit construction of  the holomorphic extensions that give the vector fields  in $\cA'$ (see Lemma \Aref{appr_lem})
shows that   that  $\cA'$ is a central hk-frame  and 
$M' \subset   \{I_2\} {\times} \{I_{2n}\} {\times} \bC^{4n}$. 
\end{proof}

It follows from this lemma that  in order to  prove that every equivalence class of 
hk-pairs includes a canonical one, it suffices to  show the following:  Given an  
 hk-pair $(\cA,  M)$, with $\cA = (H^o_0, H^o_{\pm\pm}, E^o_A, e_{\pm a})$ 
 central  and  $M$ contained in $\{I_2\} {\times} \{I_{2n}\} {\times} \bC^{4n}$ and passing 
 through $e = (I_2, I_{2n}, 0)$, 
there exists a local biholomorphism $\varphi\colon\cU \to \cU'$, with $\varphi(e) = e$,  
between two appropriate open sets mapping   $\cA$ into a canonical   hk-frame $\f_*(\cA)$
having central coordinate components  $\f = ( \f^a_b, \f^i_\pm, \f^{ia})$ with  
$\varphi^i_\pm(U, B,z) = u^i_\pm$. 
Indeed, if we are able to prove this, we immediately have that   $(\f_*(\cA),\f(M))$ is 
a canonical hk-pair in the local equivalence class of $(\cA, M)$,  as desired.

Let   $(\cA,  M)$ be an hk-pair on an appropriate open subset $\cU \subset \cP$ with 
$\cA = (H^o_0, H^o_{\pm\pm}, E^o_A, e_{\pm a})$  central and 
$M \subset \{I_2\}\times\{I_{2n}\} \times \bC^{4n}$. With no loss of generality, we  may assume 
that  the restriction  of $e_{-a}$ to  $\cV = \cU \cap \{I_2\} {\times} \{I_{2n}\} {\times} \bC^{4n}$
 has the form  $e_{-a}|_{\cV} 
 = \frac{\partial}{\partial z^{-a}}\big|_{\cV} + A_{-a}^B E^o_B|_{\cV}$.   
Indeed,  this form can always be attained by   applying a biholomorphism of the form
 $(U, B, z) \mapsto (U, B, \psi(z))$ to $(\cA, M)$, for some  
  appropriate local transformation $\psi$  of $\bC^{4n}$. Such  appropriate  transformation $\psi$ surely
 exists because the images of the vector fields $e_{-a}$ on $\bC^{4n}$ under  the standard projection 
 $\pi\colon\cP \to \bC^{4n}$ are  commuting vector fields.

 Let us now show the existence of a local biholomorphism  $\f$, with $\f^i_\pm \,{=}\, u^i_\pm$, 
 that  maps $\cA$ to a canonical hk-frame.  We denote the  components in central coordinates 
 of the required biholomorphism as 
 $\varphi = (\varphi^A) = (\varphi^i_{\pm} \,{=}\, u^i_{\pm}\,,\, \varphi^a_b\,,\, \varphi^{ia})\,$ 
 and those   in analytic coordinates  as  
$ (\varphi^i_{\pm} \,{=}\, u^i_{\pm}\,,\, \varphi^a_b\,,\, \varphi^{\pm a} \,{=}\, -u^{\pm}_i \varphi^{ia})$. 
The images of the vector fields in  $\cA$ under such a biholomorphism $\f$ have the form: 
\begin{align*}
\varphi_*(H_0^o) 
&=  H_0^o + H_0^o {\cdot} \varphi^a_b \frac{\partial}{\partial B^a_b} 
         + H_0^o {\cdot} \varphi^{i a} \frac{\partial}{\partial z^{ia}} \\
&=  H_0^o + H_0^o {\cdot} \varphi^a_b \frac{\partial}{\partial B^a_b} 
        + (H_0^o {\cdot} \varphi^{+ a} - \varphi^{+a})\frac{\partial}{\partial z^{+a}} 
        +(H_0^o {\cdot} \varphi^{- a} + \f^{-a})\frac{\partial}{\partial z^{-a}}  
\\[1pt]
\varphi_*(H_{\pm\pm}^o) 
&=  H_{\pm \pm}^o + H_{\pm\pm}^o {\cdot} \varphi^a_b \frac{\partial}{\partial B^a_b} 
         + H_{\pm \pm}^o {\cdot} \varphi^{i a} \frac{\partial}{\partial z^{ia}} 
\\[1pt]
\varphi_*(E_B^o ) 
&=  E_B^o {\cdot} \varphi^{ia} \frac{\partial}{\partial z^{ia}} 
    + E_B^o {\cdot} \varphi^a_b \frac{\partial}{\partial B^a_b}
\\[1pt]
\varphi_*(e_{\pm a}) 
&=   e_{\pm a}{\cdot} \varphi^c_b \frac{\partial}{\partial B^c_b} 
        + e_{\pm a}  {\cdot} \varphi^{i b} \frac{\partial}{\partial z^{ib}}\ .
\end{align*}
Hence the pushed-forward hk-frame $\varphi_*(\cA)$ is canonical if and only if $\f$ satisfies the following four conditions.
\begin{itemize} [itemsep=2pt, leftmargin=18pt]
\item[i)] $ \varphi_*(E^o_B) - E^o_B = 0$, which means that 
    $\varphi^{ia}$ does not depend on   $B^a_b$ and    $\varphi^a_b$   has  
 the form 
\beq 
\label{phi_trfn} 
\varphi^a_b\left((u^i_\pm), (B^e_f), (z^{ja})\right) 
                   = \varphi^a_c\left((u^i_{\pm}), I_2,  (z^{ja})\right)B^c_b 
\eeq
\item[ii)]  $ \varphi_*(H^o_0) - H^o_0 = 0$ and $\varphi_*(e_{+a}) - e^o_{+a} = 0$, 
which are equivalent to    
\begin{align}
H_0^o \cdot \varphi^a_b &=  0\,, & H^o_0 \cdot \varphi^{\pm a}&= \mp \varphi^{\pm a}, &
\label{phi0}\\
e_{+ a}\cdot \varphi^c_b &= 0\,, & e_{+a} \cdot \varphi^{- b} &= 0\,, 
& e_{+a} \cdot \varphi^{+ b} = \f^b_a
\label{phi+}
\end{align}
\item[iii)] the components $
v_{++}^{+b} $ 
of the vector field 
$$
H_{++} 
= \varphi_*(H^o_{++}) = H^o_{++} + v^{\pm b}_{++} e_{\pm b}^o + A_{++}^B E^o_B\ ,
$$
which are given by 
$v_{++}^{+b} = - u^+_i (H_{++}^o {\cdot} \varphi^{i b})  
 = H_{++}^o {\cdot} \varphi^{+ b} - \varphi^{-b}$, 
are such that   $v_{++}^{+b}|_{\{z^{+a} = 0\}} = 0$
\item[iv)]  the components $v^{-b}_{-a}$ of  
of the vector field     
$$\d e_{-a} 
= \varphi_*(e_{-a}) - e_{-a}^o =  v^{\pm b}_{-a} e^o_{\pm b} + A_{-a}^B E^o_B\ ,$$
which are given by 
$v^{-b}_{-a}  = e_{-a}{\cdot} \varphi^{-b} - \f^a_b$, 
are identically equal to $0$.
\end{itemize}
\noindent 
It remains  to prove that there exists a  
$ \f = (\varphi^i_{\pm} \,{=}\, u^i_{\pm}\,,\, \varphi^a_b\,,\, \varphi^{ia})\,$ 
satisfying  (i)-(iv), with $\varphi(I_2, I_{2n}, 0) \,{=}\,  (I_2, I_{2n}, 0)$. 
First  we define   $\wh \cV := \cV \,{\cap}\, \{z^{+a} \,{=}\, 0\}$, 
with $\cV := \cU \cap \{I_2\} \times \{I_{2n}\} \times \bC^{4n}$, and  consider a   holomorphic map 
 $$
 (g^{-b})\colon \cH|_{\wh \cV} \cdot \Sp_n(\bC)
 := \Sp_1(\bC) \times\Sp_n(\bC) \times  \wh \cV \to \bC^{2n}
 $$
 independent of $B$,  with charge $+1$,  such that 
  $g^{- b}(I_2, I_{2n}, 0) = 0$ and \linebreak $e_{-a}{\cdot}g^{-b}|_{(I_2, I_{2n}, 0)} = \d_a^b$. Second,  
we set 
\beq 
\label{g0equationbis}
g^b_a\colon \cH|_{\wh \cV} \cdot \Sp_n(\bC) \to \bC\ ,\qquad g^b_a:= e_{-a}{\cdot}g^{-b}\ .
\eeq
Third, using Lemma \ref{lemma_split}, we determine  functions
  $g^{+a}\colon \cH|_{\wh \cV} \cdot \Sp_n(\bC) \to \bC$ satisfying 
\beq
 \label{g0equation} 
  H^o_0\cdot g^{+b} =  - g^{+b},  \qquad H_{++}^o \cdot g^{+ b} =  g^{-b},
 \eeq
with  initial data chosen to be independent of $B$ and with  $g^{+a}(I_2, I_{2n}, 0) = 0$. 

We now extend  the functions $g^{\pm a}, g^c_d\colon\cH|_{\wh \cV} \cdot \Sp_n(\bC) \to \bC$ 
to holomorphic functions $\varphi^{\pm c}, \varphi^c_d$  on 
an appropriate open set  $\cU = \cH|_{\cV} \,{\cdot}\, \Sp_n(\bC) $   as follows.  
First,  we consider the   points $y \in \cP$ of the form
\beq 
y(x, t^1, \ldots, t^{2n}) := \Phi^{e_{+1}}_{t^1} \circ \ldots \circ  \Phi^{e_{+ 2n}}_{t^{2n}}(x) \ ,
\quad \ x \in \cP|_{\wh \cV}\ ,\  t^j \in \D_\varepsilon(0) \subset   \bC\ ,
\label{points}
\eeq
where $\Phi^{e_{+b}}_s$ is the  holomorphic flow of the vector field  $e_{+b}$ parametrised by $s$.  
Second, we set 
\begin{align*} 
&\varphi^{-b}(y(x, t^1, \ldots, t^{2n})) :=   g^{-b}(x)\ ,\\
 &\varphi^{+b}(y(x, t^1, \ldots, t^{2n})) := g^b_a(x) t^a + g^{+b}(x) ,\\
& \varphi^b_c(y(x, t^1, \ldots, t^{2n})) := g^b_c(x) .
 \end{align*}
 By construction, the map $\varphi = (\f^i_{\pm} = u^i_\pm, \f^a_b, \f^{\pm a})$ is such that  $\varphi(I_2, I_{2n}, 0) = (I_2, I_{2n}, 0)$ and it satisfies  (i).
Since  $g^a_b, g^{\pm a}$ satisfy  \eqref{g0equationbis} and \eqref{g0equation},   the  map $\varphi$    is a solution of    \eqref{phi0},   \eqref{phi+},  hence it satisfies (ii). Moreover, \eqref{g0equation} implies  $(H_{++}^o {\cdot} \varphi^{+ b} - \varphi^{-b} )|_{\{z^{+a} = 0\}} = 0$, so that    (iii) holds.  Finally, from  \eqref{g0equationbis} we see that also (iv) is satisfied, meaning that $\f$ is a bridge.

To conclude the proof of Theorem\ \ref{vpot_thm}, it remains to show that  canonical hk-pairs with  the same v-potential  have the same hk-frame and are locally equivalent.
Let $(\cA, M),\, (\cA', M')$ be canonical  hk-pairs   on  an appropriate open set 
$\cU \subset \cP$, with bridges $\varphi\,,\, \varphi'$ and  identical  v-potentials,
 $$
 v^{-a}_{++}|_{\cH|_\cV} = v'{}^{-a}_{++}|_{ \cH|_\cV}\ , \qquad \cH|_\cV := \cU \cap \cH\ .
 $$ 
We claim that all  components of $H_{++}$ (and,   similarly,  of $H'_{++}$)  are completely 
determined by the v-\-poten\-tials $v^{-a}_{++}|_{ \cH|_\cV} $ ($= v'{}^{-a}_{++}|_{ \cH|_\cV})$. 
For this, we first observe that   $[H^o_0, H_{++}] = 0$
 implies  that the components  $\ v_{++}^{-a}\,,\, v_{++}^{+ a}\,,\,A_{++}^B(E_B^o)_a^b\ $  
 of $\,H_{++}\,$  have charges 
$k = 3\,,\, 1$ and $2$, respectively. Further, the relation 
 $[E^o_A, H_{++}] = 0$ implies that these components   are $\gsp_n(\bC)$-equi\-variant and 
 hence  are uniquely determined by their restrictions  to $\cH|_\cV$. 
It therefore  suffices  to check that  
$v_{++}^{+ a}\vert_{\cH|_\cV}$, $A_{++}^B(E_B^o)_a^b\vert_{\cH|_\cV}$ 
are uniquely determined by the v-potential 
 $v^{-a}_{++}|_{ \cH|_\cV}$.

 We now  recall  that $e^o_{+a}  =  [H_{++}\,,\, e_{-a}] $ and $[H_{++}\, ,\, e^o_{+a}] = 0$. 
 Expanding  all vector fields in terms of the flat hk-frame, we get
\begin{align*}
e^o_{+a}  &=  [H_{++}\,,\, e_{-a}] 
= [H^o_{++} + v_{++}^{\pm b} e_{\pm b}^o + A_{++}^A E_A^o\,,\, e^o_{-a} 
   + v_{-a}^{+b} e_{+b}^o + A_{-a}^B E^o_B] 
\\[2pt]    
    &=  \d^b_a e^o_{+b} + H_{++}^o {\cdot} v^{+b}_{-a} e^o_{+b} + H_{++}^o {\cdot} A_{-a}^B E_B^o 
             - e^o_{-a}  {\cdot}   v_{++}^{+ b} e_{+ b}^o - e^o_{-a} {\cdot}   v_{++}^{- b} e_{- b}^o
\\  
&\quad  + v_{++}^{+ c} e_{+ c}^o {\cdot}  v^{+b}_{-a} e^o_{+b} + v_{++}^{- c} e_{- c}^o {\cdot}   v^{+b}_{-a} e^o_{+b} 
        - v^{+c}_{-a} e^o_{+c}{\cdot}  v_{++}^{+b} e_{+ b}^o - v^{+c}_{-a} e^o_{+c}{\cdot}  v_{++}^{- b} e_{- b}^o 
\\  
&\quad  +  v_{++}^{+ c} e_{+ c}^o {\cdot}  A_{-a}^B E_B^o + v_{++}^{- c} e_{- c}^o {\cdot}  A_{-a}^B E_B^o  
      + A_{++}^B (E_B^o)^b_a e^o_{-b} -  e^o_{-b} {\cdot}  A_{++}^B E_B^o 
\\   
&\quad  + A_{++}^A (E_A^o)^c_a v_{-c}^{+b} e^o_{+b} -  v_{-a}^{+c} e^o_{+c} {\cdot}  A_{++}^B E_B^o + A_{++}^A (E_A^o)^c_a A_{-c}^B E_B^o    
\\[2pt]   
&=   \left(-  e_{-a}^o {\cdot}  v_{++}^{- b} - v^{+c}_{-a} e^o_{+c}{\cdot}  v_{++}^{- b} + A_{++}^A (E_A^o)^b_a\right) e_{- b}^o 
           \mod\{ e^o_{+b}, E^o_B\} \ ,
\\[10pt]   
 0 &=  [H_{++}\, ,\, e^o_{+a}] = [H^o_{++} + v_{++}^{\pm b} e_{\pm b}^o + A_{++}^A E_A^o\,, \, e^o_{+a}] 
\\[2pt]   
&= \left(A_{++}^A (E_A^o)^b_a  - e^o_{+a} {\cdot}  v^{+b}_{++}  \right)  e_{+b}^o   - e^o_{+a} {\cdot}  v^{-b}_{++}  e_{-b}^o  - e^o_{+a} {\cdot} A^{A}_{++}  E_A^o\ .
\end{align*}
It follows that 
 \beq\label{potentials}
 \begin{split}
 e^o_{+a} \cdot v^{-b}_{++}  &=  0\ ,
\\[2pt]   
  A_{++}^A (E_A^o)^b_a &=   
       e_{-a}^o \cdot v_{++}^{- b}  +  v^{+c}_{-a} \, e^o_{+c}\cdot v_{++}^{- b} 
       = e_{-a}^o \cdot v_{++}^{- b}\ , 
\\[2pt]   
 e^o_{+a} \cdot v^{+b}_{++} &=  A_{++}^A (E_A^o)^b_a =   e_{-a}^o \cdot v_{++}^{- b}\ . 
 \end{split}
\eeq 
Since   $v^{+a}_{++}|_{\{z^{+a} = 0\} } = 0$, these equations show  that  $v_{++}^{+ a}$ and   $A_{++}^B(E_B^o)_a^b$ are uniquely determined by 
the (first derivatives of the) functions $v_{++}^{- b}$, as claimed. So $H_{++}$ (and $H'_{++}$)  is completely determined by the v-potential, as claimed.
Since  the two v-potentials are equal, it follows also that $H_{++} = H'_{++}$.

Now,   applying the  inverse of the bridge $\varphi$ to both   hk-frames $\cA$ and $\cA'$ we obtain the new  hk-frames
\begin{align*} \f^{-1}_*(\cA) &= (H^o_0, H^o_{++} , H^o_{--} , E^o_A, \wh e_{\pm a}:= \varphi^{-1}_*(e_{\pm a}))\ ,\\
\f^{-1}_*(\cA') & = (H^o_0, H^o_{++} , \wh H'_{--} := \varphi^{-1}_*(H'_{--}), E^o_A, \wh e'_{\pm a}:= \varphi^{-1}_*(e'_{\pm a}) )\ ,
\end{align*}
where we used the facts that  $H_{++} = H_{++}'$ and that  $\f$ is a bridge from a central hk-frame  to  the canonical hk-frame $\cA$.
Now, if we can  prove   that $\wh H_{--}' = H^o_{--}$, it would immediately follow that 
$H'_{--} = \f_*(\wh H_{--}')$ coincides with  $H_{--}$ and  that $e'_{-a} = [H'_{--}, e^o_{+a}] = [H_{--}, e^o_{+a}] = e_{-a}$, meaning that $\cA = \cA'$.

Since    $\varphi^i_{\pm} = u^i_{\pm}$, the vector field  $\wh H'_{--}$ has the form 
$$
\wh H'_{--} 
=   H^o_{--} + v_{--}^{+ a} e_{+ a}^o +  v_{--}^{- a} e_{- a}^o + A_{--}^B E_B\ .
$$
On the other hand,  
$$\ 
[H^o_0, \wh H'_{--}] = \f^{-1}_*([H^o_0, H'_{--}]) 
                                =   \f^{-1}_*(- 2 H'_{--})  = - 2 \wh H'_{--}\ .
$$
Thus the components $v_{--}^{+ a}\,,\,v_{--}^{-a}$ and $A_{--}^B$ have charges 
$-3$, $-1$ and $-2$, respectively. Further, 
\begin{align*}
 H^o_0 &=   \f^{-1}_*([H_{++}', H_{--}'])   = [H_{++}^o, \wh H'_{--}] \\
 &= [H_{++}^o\ ,\  H^o_{--} + v_{--}^{+ a} e_{+ a}^o +  v_{--}^{- a} e_{- a}^o + A_{--}^B E_B] 
 \end{align*}
implies that 
 $$
 H^o_{++} \cdot v^{+a}_{--} +  v_{--}^{-a}  = 0\,,\qquad  
 H^o_{++} \cdot v_{--}^{-a}  = 0 \,,\qquad 
 H^o_{++} \cdot A_{--}^B  = 0\,.
 $$
Since the  functions $v_{--}^{+ a}\,,\,v_{--}^{-a}$ and $A_{--}^B$ are negatively charged,  
they vanish by Lemma \ref{lemma_split}.  Thus  $\wh H'_{--} = H^o_{--}$ and $\cA = \cA'$, 
as required. 

We now observe that by definition of hk-pairs,  the   $\Sp_{p,q}$-orbits of the points of $M$ and $M'$  (namely, the submanifolds   $\cU^{(\Sp_{p,q})}$, $\cU'{}^{(\Sp_{p,q})} \subset \cU$ defined in \eqref{USP}) determine two integral submanifolds 
of the distribution  generated by the vector fields in $\a^{\cA}_{(\bR)}(V^\t + \gsp_{p,q})$.  Since $e = (I_2, I_{2n}, 0)$ belongs to both of them,  
$\cU^{(\Sp_{p,q})} = \cU'{}^{(\Sp_{p,q})} $ and 
  $M'$ can be identified with  a section of the (trivial) $\Sp_{p,q}$-bundle $\pi: \cU^{(\Sp_{p,q})} \simeq M \times \Sp_{p,q} \to M$. 
So,  if  we  construct  pseudo-hyperk\"ahler metrics $g$ and $g'$ on $M$ and $M'$, respectively,     as  in Sect.\ref{hkpairs-metrics}, 
we see that  the projection $\pi|_{M'}: M' \to M$  maps  the vielbeins of $(M',g')$  onto vielbeins of $(M,g)$ and is therefore an 
 isometry between $(M', g')$ and $(M,g)$.  Proposition \ref{loc.equiv} implies  that $(\cA, M)$ and $(\cA = \cA', M')$
 are locally equivalent. This concludes the proof of Theorem \ref{vpot_thm}.

 \section
 {Parameterisation  of  canonical hk-pairs} \label{pf_prepot_thm}
 \centerline{\smc  Proof of Theorem \ref{prepot_thm}}
\medskip
\noindent 
 The proof is divided into two steps. We first need to prove that   for every  
 prepotential $\cL_{(+4)}$ there exists  a canonical hk-pair  whose
 v-potential   is related to $\cL_{(+4)}$ by \eqref{spippolino}.  We then need to prove  the converse statement:  every  canonical hk-pair has  a uniquely associated prepotential  satisfying   \eqref{spippolino}.
 
\goodbreak
\noindent
{\bf Step 1:  Existence of a canonical  hk-pair for every  prepotential}  

\noindent 
Consider   a charge $k {=} 4$ holomorphic map 
$ \cL_{(+4)}\colon \cH\vert_\cV  \rightarrow \bC$  on an harmonic space $\cH\vert_\cV$
satisfying \eqref{spippolo}.  It follows from \eqref{potentials}  that if  there exists 
an hk-pair $(\cA, M)$  with v-potential satisfying \eqref{spippolino}, then 
the components   of the vector field 
\beq\label{defH++} 
H_{++}  =  H^o_{++} +  v^{+ c}_{++} e^o_{+ c} + v^{- c}_{++} e^o_{- c} + A^C_{++} E^o_C
\eeq
necessarily   have the following form  at any point 
 $ x {\cdot} B \in\cU :=  \cH\vert_\cV {\cdot}  \Sp_{n}(\bC)$,  $x \in \cH$, $B \in \Sp_n(\bC)$, 
\beq\label{occhio} 
\begin{split} 
v^{-a}_{++}|_{x {\cdot} B} 
&=   (B^{-1})^a_c\, \o^{cd}  \frac{\partial \cL_{(+4)}}{\partial z^{-d}}\Big|_{x}  
\\
A^B_{++} (E^o_B)^a_b\vert_{x {\cdot} B} 
&=  (B^{-1})^a_c\, B^d_b\, \o^{ce}  
              \frac{\partial^2 \cL_{(+4)}}{\partial z^{-d} \partial z^{-e}}\Big|_{x} 
\\
v^{+a}_{++}|_{x {\cdot} B} &= (B^{-1})_b^a \,\wt v^{+b}_{++}|_x \ .
\end{split}
\eeq
Here  $(\o^{ab}) = (\o_{ab})^{-1}$ and  the  functions $\wt v^{+a}_{++}\colon \cH\vert_\cV \rightarrow \bC^{2n}$
are  solutions to  the  differential problem 
$$ \frac{\partial \wt v^{+a}_{++} }{\partial z^{+b}} 
= \o^{ac}  \frac{\partial^2 \cL_{(+4)}}{\partial z^{-b} \partial z^{-c}}\ ,
\qquad      \wt v^{+a}_{++}|_{\{z^{+a} = 0\}} \equiv 0\ .
$$
Since $e^o_{+a} {\cdot} \cL_{(+4)} = \frac{\partial \cL_{(+4)}}{\partial z^{+a}} = 0$,     
this has a unique solution, linear in $z^{+b}$,
\beq \label{firstsystem-bis} 
\wt v^{+a}_{++} =  \o^{ac}  \frac{\partial^2 \cL_{(+4)}}{\partial z^{-b} \partial z^{-c}} z^{+b}\ . 
\eeq
Now, as an ansatz we take $H_{++}$ with components  thus determined by 
$\cL_{(+4)}$  and we search for a local biholomorphism  $\varphi$, with $\varphi(I_2, I_{2n}, 0) = (I_2, I_{2n}, 0)$,  whose components  
in central coordinates satisfy a  system  of differential equations which corresponds  to a special subset of  the conditions  that  characterise a bridge. 
In the next three lemmata,  we shall show that: a)
these  equations admit at least one global solution  $\varphi$ on $\cH\vert_\cV\,$
(Lemma \ref{ic_lem}), b)
in the class  of global solutions there exists one satisfying a special set of 
 initial conditions (Lemma \ref{strange}), c) using such  a special solution
we may construct  an hk-pair $(\cA, M)$ having $\cL_{(+4)}$ as prepotential and the map $\f$ as  a bridge  (Lemma \ref{hk_lem}). 
These lemmata will  conclude  the proof of Step\,1.

\begin{lem} \label{ic_lem} 
On an appropriate open set $\,\cU = \cH|_{\cV}\, {\cdot}\, \Sp_n(\bC) \subset \cP$, 
the differential equations  
\beq  
\varphi_*(E_A^o) = E^o_A\,,\quad  \varphi_*(H_0^o) = H^o_0\,,\quad  
\varphi_*(H_{++}^o) = H_{++} \ ,  \label{secondsystem} 
\eeq
admit at least one  global solution  $\varphi\colon \cU   \rightarrow \cU$
with $\ \varphi^i_{\pm} =  u^i_\pm\,$.
\end{lem}

\begin{proof}  
 In virtue of \eqref{coord_vfs},  the differential equations take the form
\begin{align}\label{prima-bis} 
(E_B^o \cdot \varphi^A) \frac{\partial}{\partial x^A} 
 &= \varphi^a_c (E_B^o)_b^c \frac{\partial}{\partial B^a_b}\,,
\\ \label{seconda-bis}
 (H_0^o\cdot \varphi^A) \frac{\partial}{\partial x^A} 
 &= \varphi^i_+ \frac{\partial}{\partial u^i_+}  \,-\,  \varphi^i_- \frac{\partial}{\partial u^i_-} \,,
\\ \label{terza}
(H_{++}^o \cdot  \varphi^A ) \frac{\partial}{\partial x^A} 
&= \varphi^i_+ \frac{\partial}{\partial u^i_-}  
      \,+\, \left.v^{\pm b}_{++}\right|_{\varphi} \varphi^a_b \varphi^j_{\pm} \frac{\partial}{\partial z^{ja}} 
       +  \left.A^B_{++}\right|_{\varphi}\varphi^a_c (E^o_B)^c_b \frac{\partial}{\partial B^a_b}\,. 
\end{align}
To prove the existence  of  solutions  with $\varphi^i_{\pm} = u^i_{\pm}$, we 
first  note that  solutions  of  \eqref{prima-bis} are   maps  such that 
a) the components   $\varphi^{ia}$ do not depend on $B$,  i.e.  
 $
 \varphi^{ia}= \varphi^{ia}((u^i_{\pm}), (z^{ja}))
 $,  and
b)  the components $\varphi^a_b$   satisfy eq. \eqref{phi_trfn},
with  $\varphi^a_b((u^i_{\pm}),  (z^{jc}))$  denoting the restriction  $ \varphi^a_b|_{\cH}$ . 
Thus  the problem reduces to looking for   holomorphic functions  
$\varphi^a_c\,,\, \varphi^{ia}$ on   $\cH|_{\cV}$  satisfying  \eqref{seconda-bis} 
and \eqref{terza}, with $\varphi^i_\pm = u^i_\pm$.  These equations   say    
that $ \varphi^a_c\,,\, \varphi^{ia}$ have  charge $0$ and using   
\eqref{occhio}-\eqref{firstsystem-bis} we obtain
\begin{align} 
\nonumber   H_{++}^o \cdot \varphi^{ia}   
& = v^{+ b}_{++}|_{\varphi} \varphi^a_b\, u^i_{+} 
        + v^{- b}_{++}|_{\varphi} \varphi^a_b\, u^i_{-}  \\
& = u^i_{+}   \o^{ac} 
        \frac{\partial^2 \cL_{(+4)}}{\partial z^{-b} \partial z^{-c}}\Big|_{(u^i_{\pm}, \varphi^{ja})} 
         u^+_j \varphi^{jb}  
+  u^i_{-}  \o^{ad}  
         \frac{\partial \cL_{(+4)}}{\partial z^{-d}}\Big|_{(u^i_{\pm}, \varphi^{ja})}
\label{sisteminobis}\\[5pt]
 H_{++}^o \cdot \varphi^a_b   
 & =  \varphi^a_c\,  A^B_{++}|_{\varphi} \, (E^o_B)^c_b 
 =  \varphi^c_b \, \o^{ad}  
         \frac{\partial^2 \cL_{(+4)}}{\partial z^{-c} \partial z^{-d}}\Big|_{(u^i_{\pm}, \varphi^{ja})}.
\label{sistemino}
\end{align}
Now, writing  $\varphi^a_b = (e^\psi)^a_b$ with 
$\, \psi\colon\cH|_\cV \to \bC^{2n}\otimes \bC^{2n}$,   equation  \eqref{sistemino} takes the form
\begin{equation}
H_{++}^o \cdot \psi^a_b =    
\o^{ac} \frac{\partial^2 \cL_{(+4)}}{\partial z^{-c} \partial z^{-b}}\Big|_{(u^i_{\pm}, (\varphi^{i a}))}.
\label{sisteminoquater}
\end{equation}
Equation \eqref{sisteminobis} decouples.  Its right hand side has charge $k = 2$, 
so by Lemma  \ref{charge2_lem},  it admits   a   global  solution  $\varphi^{ia}$ 
on $\cH|_{\cV}$. Inserting this  in \eqref{sisteminoquater}, we obtain a 
inhomogeneous linear  equation for $\psi^a_b$, which admits  a global solution  
by  Lemma \ref{lemma_split}. 
 \end{proof}
 \begin{rem} \label{remarkable} Since
 $H_{++}{\cdot} u^{+}_i = - u^{-}_i$, $H_{++} {\cdot} u^{-}_i = 0$, 
writing $\varphi^{\pm a} := - u^{\pm}_j \varphi^{ja}$,  
equation  \eqref{sisteminobis}  
allows the   convenient reformulation 
\begin{align}
 H_{++}^o \cdot \varphi^{-a}  
 &= \o^{ab}  \frac{\partial \cL_{(+4)}}{\partial z^{-b}}\Big|_{(u^i_{\pm}, \varphi^{- c})}\ ,
 \label{sisteminoquater1}\\
 H_{++}^o \cdot \varphi^{+a}    
 &=   \varphi^{+b} \o^{ac} 
            \frac{\partial^2 \cL_{(+4)}}{\partial z^{-b} \partial z^{-c}}\Big|_{(u^i_{\pm}, (\varphi^{- d}))} 
       + \varphi^{-a}  . 
\label{sisteminoquater2}
\end{align}
 Note that the first equation is a  nonlinear differential equation in  $\varphi^{-a}$ only, 
 while the second is linear and inhomogeneous in the remaining variable $\varphi^{+a}$. 
 \end{rem}
 
 \begin{lem} \label{strange} 
 There exists a global  solution 
$\varphi = (\varphi^i_{\pm} = u^i_{\pm}\,,\, \varphi^a_b\,,\, \varphi^{ia})$ 
to the system \eqref{secondsystem}  on an appropriate open set 
$\cU = \cH|_{\cV} {\cdot} \Sp_n(\bC)$  satisfying the additional condition 
 \beq \label{strange_constraint}
 \varphi^{+a}|_{\{I_2\} \times \cV} 
 = (- H_{--} \cdot \varphi^{-a}+ c^{+a})|_{\{I_2\} \times \cV}\, ,\quad 
 \f^a_b|_{\{I_2\} \times \cV} = \d^a_b\ ,
 \eeq
 where $\ \varphi^{\pm a} := - u^{\pm}_j \varphi^{ja}\ $ and 
 $\ c^{+a}:= H_{--} \cdot \varphi^{-a}|_{(I_2, 0)}$ .
 \end{lem}
 \begin{proof} 
 Let  $\wt \varphi = (\wt \varphi^i_\pm = u^i_\pm\,,\, \wt \varphi^a_b\,,\, \wt \varphi^{ia})$  
 be a   global solution  to \eqref{secondsystem} on an appropriate open set 
 $\cU = \cH|_{\cV} \cdot \Sp_n(\bC)$.  
 As shown in   the proof of  Lemma  \ref{ic_lem}, $\,\wt \varphi^{ia}\,$ and 
 $\,\wt \varphi^a_b\,$ are solutions to   \eqref{sisteminoquater1}, \eqref{sisteminoquater2} 
 and \eqref{sisteminoquater}.
We now consider the   {\it linear}  system for functions $\varphi'{}^{\pm a}$
\begin{align*}
& H_{++}^o \cdot \varphi'{}^{-a}  =0\ ,\\
&  H_{++}^o \cdot \varphi'{}^{+a}    
=     \varphi'{}^{+b} \o^{ac}
            \frac{\partial^2 \cL_{(+4)}}{\partial z^{-b} \partial z^{-c}}\Big|_{(u^i_{\pm}, (\wt \varphi^{- d}))} 
            +  \wt \varphi^{-a} 
\end{align*}
Writing  $\varphi'{}^{ia} := u^i_+ \varphi'{}^{+a} + u^i_- \varphi'{}^{-a}$, 
we obtain a system for functions $\varphi'{}^{ia}$   satisfying the hypotheses  
of Lemma \ref{charge2_lem}. Therefore (see Rem. \ref{remark_initial_data}) 
 there exists  a global  solution  $\varphi'{}^{ia}$ to  this system satisfying the 
 initial conditions 
$$
\varphi'{}^{+a}|_{\{I_2\} \times \cV} 
= \left(- H_{--}^o \cdot \wt \varphi^{-a} + c^{+a} - \wt \varphi^{+a}\right)\vert_{\{I_2\} \times \cV}  \ .
$$
Inserting
$$\varphi^{ia} := u^i_+( \wt \varphi^{+a} + \varphi'{}^{+a}) + u^i_- \wt \varphi^{-a} $$
in \eqref{sisteminoquater}, we choose a global solution    $\psi^a_b$ satisfying $\psi^a_b|_{\{I_2\} \times \cV} = 0$; it exists by Lemma \ref{lemma_split}.  A direct check shows that   $\varphi = (\varphi^i_\pm = u^i_\pm, \varphi^a_b = (e^\psi)^a_b, \varphi^{ia})$ is a global solution to \eqref{sisteminoquater1}, \eqref{sisteminoquater2} and \eqref{sisteminoquater}  and thus to 
 \eqref{secondsystem}. It  is
  defined on an appropriate open set 
and satisfies \eqref{strange_constraint}. 
\end{proof}

\begin{lem} \label{hk_lem}
 Let
 \begin{itemize}[itemsep=3pt, leftmargin=18pt]
 \item[i)] $H_{++} $ be the vector field defined by  \eqref{defH++}-\eqref{firstsystem-bis},
 \item[ii)]   $\varphi\colon \cU \subset \cP  \rightarrow \cU\subset \cP$,  with  $\varphi^i_{\pm} = u^i_\pm$,  be a   global solution  to    the differential equations \eqref{secondsystem}
on an appropriate open set $\cU = \cH|_{\cV} {\cdot} \Sp_n(\bC)$   satisfying  the condition   \eqref{strange_constraint} and 
\item[iii)]  
$H_{--} := \varphi_*(H^o_{--})\ ,\quad e_{+a} := e^o_{+a} \ ,\quad e_{-a} := [H_{--}, e^o_{+a}]\ .$
\end{itemize}
Then $\cA := \left(H^o_0, H_{\pm \pm},  E_A^o,  e_{\pm a} \right)$ is a canonical hk-frame with    v-potential  satisfying \eqref{spippolino}
and there exists a $4n$-dimensional real submanifold $M \subset \cU$ so that  $(\cA, M)$ 
 is  a canonical hk-pair   with bridge    $\varphi$. 
\end{lem}
\begin{proof} We first prove that $\cA$ is an hk-frame. Consider  the vector fields $\wh e_{\pm a} := \varphi^{-1}_*(e_{\pm a})$. By construction 
$\wh \cA = \left( H^o_0, H^o_{\pm \pm},  E_A^o,  \wh e_{\pm a} \right)$
is such that $\varphi_*(\wh \cA) = \cA$.  So   the proof that $\cA$ is an hk-frame     reduces    to  showing  that  $\wh \cA$ is an hk-frame. 

The  vector fields $\wh e_{\pm a}$ may be expressed in terms of the  flat hk-frame, 
$$
\wh e_{\pm a} = \wh v_{\pm a}^{+ b} e^o_{+ b} + \wh v_{\pm a}^{- b} e^o_{-b} 
                             +  \wh v^{++}_{\pm a} H^o_{++} +  \wh v^{0}_{\pm a} H^o_{0} 
                             +  \wh v^{--}_{\pm a} H^o_{--} + \wh A_{\pm a}^A E_A^o,
$$
where the components
$\wh v^{\pm b}_{\pm a}\,,\, \wh v^{\pm\pm}_{\pm a}\,,\, \wh v^0_{\pm a}\,,\, \wh A_{\pm a}^A$ are holomorphic functions. Since
$\varphi_*(E^o_A) = E^o_A$ and $\varphi_*(\wh e_{+a})   = e_{+a} = e^o_{+a}$, 
we have that      $d u^i_{\pm}(\varphi_*(\wh e_{+ a})) = 0$. 
Inserting the above expression  for $\wh e_{+a}$ we have
\begin{align*}
0  &=   d u^i_{\pm}\big( \wh v_{+ a}^{+ b} \varphi_*( e^o_{+ b})
                                     + \wh v_{+ a}^{- b} \varphi_*(e^o_{-b})   
                                      +     \wh v^{++}_{+ a} H^o_{++}  
                                       +    \wh v^{0}_{+ a}    H^o_{0} 
                                                + \wh v^{--}_{+ a} H^o_{--}  \big)   
 \\[4pt]
  &=    d u^i_{\pm}\big(  \wh v^{++}_{+ a} H^o_{++} 
                                    +  \wh v^{0}_{+ a}    H^o_{0} 
                                    + \wh v^{--}_{+ a}  H^o_{--}  \big)\ ,
 \quad \mbox{since}\  \varphi^i_\pm = u^i_\pm\ .
\end{align*} 
Since  $H^o_0\,,\, H^o_{++}$ and $H^o_{--}$ are linearly independent at each point,  
it follows that 
\beq \label{H_li} 
\wh v^{++}_{+ a} =  \wh v^0_{+ a}  =  \wh v^{--}_{+ a} = 0\ .
\eeq
Further, since  $\wh e_{-a} = \varphi^{-1}_*([H_{--}, e^o_{+a}]) = [H^o_{--}, \wh e_{+a}]$, we find   that  $\wh e_{-a}$  has no component along  $H^o_0 $ and $H^o_{\pm \pm}\,$, i.e.
\beq \label{e-hor} 
\wh v^{++}_{- a}  =  \wh v^0_{- a} =  \wh v^{--}_{- a} =0\ .
\eeq
We  now check that the Lie brackets between  $\wh e_{\pm a}$ and  the other fields in $\wh \cA$ have the  required form. By direct computation:
\begin{align*}
[ E^o_A, \wh e_{+ a}]  &= \varphi^{-1}_*([E^o_A, e^o_{+a}]) 
     = (E_A^o)^b_a \, \varphi^{-1}_*(  e^o_{+a}) =  (E_A^o)^b_a \,\wh e_{+b} 
\\[4pt]
[ E^o_A, \wh e_{- a}]  &= \varphi^{-1}_*([E^o_A, [H_{--}, e^o_{+a}]]) 
       =   (E_A^o)^b_a\,\varphi^{-1}_*([H_{--},  e^o_{+b}]) =  (E_A^o)^b_a\, \wh e_{-b}
\\[4pt]
[H^o_0, \wh e_{+a}] &= \varphi^{-1}_*([H^o_0,  e^o_{+a}])  = \wh e_{+a}
\\[4pt]
[H^o_0, \wh e_{-a}] &= \varphi^{-1}_*([H^o_0, [H_{--}, e^o_{+a}]]) 
     \\    &    = - 2 \varphi^{-1}_*( [H_{--}, e^o_{+a}]) + \varphi^{-1}_*( [H_{--}, e^o_{+a}]) 
     = - \wh e_{-a} 
 \\[4pt]
[H^o_{++}, \wh e_{+a}] &=  \varphi^{-1}_*([H_{++},  e^o_{+a}])  
      = - \varphi^{-1}_*((e^o_{+a} {\cdot} v^{\pm b}_{++})\, e^o_{\pm b} + (e^o_{+a} {\cdot} A_{++}^B)\, E^o_B) = 0\,.
\end{align*}
The last equality follows from \eqref{occhio} 
and  $e_{+a}^o {\cdot} \cL_{(+4)} = 0$. Further, 
$$
[H^o_{++}, \wh e_{-a}] = \varphi^{-1}_*([H_{++}, [H_{--}, e^o_{+a}]])  
               = [H^o_{++},  [H^o_{--}, \wh e_{+a}]]  = [H^o_0, \wh e_{+a}]  = \wh e_{+a}\ 
$$
and by construction $   [H^o_{--}, \wh e_{+a}]  = \wh e_{-a}$.
It remains to verify that $ X_{---a} := [H^o_{--}, \wh e_{-a}] =0$  and that 
$\wh Y_{ab} := [\wh e_{+a}, \wh e_{-b}]$  has terms only in the directions of the  $E^o_A$. 
Expanding   $X_{---a}$ in the vector fields of the flat hk-frame $\cA^o$, 
$$ 
X_{---a} = X_{---a}^{\pm b} e^o_{\pm b} 
                  + X_{---a}^0 H^o_0 + X_{---a}^{\pm\pm} H^o_{\pm\pm} 
                  + X^A_{--- a} E^o_A\ ,
$$
we see that since
$\ \ [H_0^o, X_{---a}] = - 2  [H^o_{--}, \wh e_{-a}]  - [H^o_{--}, \wh e_{-a}]  = - 3 X_{---a}\,$,  
each  component of $X_{---a}$ has a negative charge. 
Further, from the expansion in the flat hk-frame of the equality
\beq \label{H++commutator} 
[H_{++}^o, X_{---a}] = [H^o_{0}, \wh e_{-a}] +  [H^o_{--}, \wh e_{+a}] 
                                     = - \wh e_{-a} + \wh e_{-a} = 0\ ,
\eeq
we find  that 
$
H^o_{++} {\cdot} X_{---a}^{-b}  = H^o_{++} {\cdot} X_{---a}^{--}  
                                                         =  H^o_{++} {\cdot} X_{---a}^{A} = 0\,.
$
It follows from Lemma \ref{lemma_split} a)  that 
$X_{---a}^{-b} = X_{---a}^{--} = X_{---a}^{A} = 0$. 
Expanding  once again  \eqref{H++commutator} in the flat hk-frame and using the vanishing of these components, we get that 
$
H^o_{++} \cdot X_{---a}^{+b}  = H^o_{++} \cdot X_{---a}^{0}    = 0\,.
$
Lemma \ref{lemma_split} a)  then implies $X_{---a}^{+b} = X_{---a}^{0} = 0$ and we get that
the remaining component in the expansion of \eqref{H++commutator} gives  $X_{---a}^{++} = 0$. 
It follows that $X_{---a} =  [H^o_{--}, \wh e_{-a}]  = 0$, as required.

Now, in the image $\varphi(\cH\vert_\cV)\subset \cP$ we have:
\begin{align}
 H_{--}  &= \varphi_*(H_{--}^o) 
           = H^o_{--} + v_{--}^{+ a} e^o_{+ a} + v_{--}^{- a} e^o_{- a} + A_{--}^A E_A^o
\\
 \label{e-}
e_{-b}\ &=\ \varphi_*(\wh e_{-b})\  =  [H_{--},  e^o_{+b}] 
          =   e^o_{-b}  + v_{-b}^{-c} e^o_{-c} + v_{-b}^{+c} e^o_{+c} + A_{-b}^A E_A^o\,.
\end{align}
The components of these vector fields are:  
\begin{align*}  
v_{--}^{\pm a}|_{(U,  B,  z)} 
    &= -   (u^{\pm}_i (e^{-\psi})^a_c H_{--}^o \cdot \varphi^{ic})|_{\Phi(U,B, z)} 
\\   
A_{--}^B (E_B^o)^a_b|_{(U,B,  z)} 
 & = ( (e^{-\psi})^a_cH^o_{--} \cdot \varphi^c_b )|_{\Phi(U,B, z)}
 \\
v_{-b}^{+c}|_{(U,  B, z)}  
&=  (A_{--}^B (E_B^o)^c_b -e^o_{+b} \cdot v_{--}^{+ c})|_{(U,B,  z)} 
 \\ 
v_{-b}^{-c}|_{(U, B,z)}  &=  - e^o_{+b} \cdot v_{--}^{- c}|_{(U,  B, z)} 
\\ 
A_{-b}^B (E_B)^c_a|_{(U, B,z)} 
&=  (-  e^o_{+b} \cdot (A_{--}^B (E_B)^c_a))|_{(U,  B,z)}\,,
\end{align*}  
where we denote the inverse map of $\f$ by  
$\Phi = \f^{-1}$ and  write $(\f^a_b) = (e^\psi)^a_b$.
From this and \eqref{prima-bis} we see that $v^{-c}_{-b}$ is entirely determined 
by the map $\varphi$ as follows:
\beq
\label{v-c-b}
\begin{split}
 v^{-c}_{-b}|_{(U, B, z)} 
 &=   e^o_{+b}\vert_{(U,B,z)} {\cdot} (u^{\pm}_i (e^{-\psi})^a_c H_{--}^o \cdot \varphi^{ic}\vert_{\Phi(U,B,z)}) 
\\[3pt]
& = -\frac{\p ((e^{-\psi})^a_c(H_{--}^o  {\cdot} \varphi^{-c}
          + \varphi^{+c}))}{\p Y^M}\Big\vert_{\Phi(U,B,z)} \frac{\p \Phi^M}{\p z^{+b}}\Big|_{(U,B,z)} 
\\[3pt]
  & =  \Big((e^{-\psi})^a_c \frac{\p \psi^c_f}{\p B^d_e} (H_{--}^o  {\cdot} \varphi^{-f}
          + \varphi^{+f})\Big)\Big\vert_{\Phi(U,B,z)}  \frac{\p \Phi^d_e}{\p z^{+b}}\Big|_{(U,B,z)}  
\\[3pt]
&  \quad +\Big((e^{-\psi})^a_c   \frac{\p \psi^c_f}{\p z^{\pm d}} \big(H_{--}^o  {\cdot} \varphi^{-f}
          + \varphi^{+f}\big)\Big)\Big\vert_{\Phi(U,B,z)} \frac{\p \Phi^{\pm d}}{\p z^{+b}}\Big|_{(U,B,z)}
\\[3pt]
&  \quad - \Big( (e^{-\psi})^a_c \frac{\p\big(H_{--}^o  {\cdot} \varphi^{-c}
          + \varphi^{+c}\big)}{\p z^{\pm d}}\Big)\Big\vert_{\Phi(U,B,z)} \frac{\p \Phi^{\pm d}}{\p z^{+b}}\Big|_{(U,B,z)}  
\end{split}
\eeq
where   $(Y^M) = (u^i_\pm, B^a_b, z^{\pm})$.  
To proceed we need  the following technical lemma.

\begin{lem} \label{newlemma}  If $\varphi$ satisfies \eqref{strange_constraint},  
the components   $v_{-b}^{-c} $ of the vector field $e_{-b}$ are identically equal to $0$.
\end{lem}
\begin{proof}
Expanding the relation $\,e_{+b} =   [H_{++}, e_{-b}]\,$  in the flat basis, where $H_{++}$  is the vector field  in   \eqref{defH++}-\eqref{firstsystem-bis},  and using the $\gsp_n(\bC)$-equivariance
for the components of $e_{-b}$ implied by $\,[E^o_A, e_{-b}] = (E^o_A)^d_b e_{-d}\,$,
we have: 
\begin{align*}
e^o_{+b} 
&=    \left[  H_{++} \ ,\  e_{-b} \right] 
\\[2pt]
&  = \left[ H_{++} \ ,\  (\d^c_b + v^{-c}_{-b} )e^o_{-b} + v^{+ c}_{+b}\ e^o_{+ c}  
                                                                                 +A^B_{-b}\  E^o_B \right]  
\\[2pt]
& = (\d^c_b + v^{-c}_{-b} )  \left[ H_{++} \ ,\  e^o_{-c} \right]    
       + H_{++} {\cdot} v^{\pm c}_{-b}\  e^o_{\pm c} 
       + H_{++} {\cdot}A^B_{-b}\, E^o_B 
 \\[2pt]
& = (\d^c_b + v^{-c}_{-b} ) 
       \left[H^o_{++} + v_{++}^{\pm d} e^o_{\pm d} + A_{++}^C E^o_C \ ,\ e^o_{-c}\right]    
                \mod  \langle e^o_{+c} , E^o_B \rangle
\\[2pt]
&=  
\left(H_{++} \cdot v^{- c}_{-b} + (\d^d_b + v^{-d}_{-b} ) (A^{C}_{++}(E^o_C)^c_d - e^o_{-d} \cdot v^{-c}_{++}) \right) e^o_{- c}  
              \mod  \langle e^o_{+c} , E^o_B \rangle .
\end{align*}
Now, from \eqref{occhio}, we have that 
$$ 
A^{C}_{++}(E^o_C)^c_d = e^o_{-d} \cdot v^{-c}_{++}\,,
$$
so the remaining components in the $e^o_{- c}$-directions imply that
$(H_{++} {\cdot} v^{- c}_{-b})\circ \varphi  =  H^o_{++} {\cdot} (v^{-c}_{-b} \circ \varphi) = 0$. Since $\,v^{- c}_{-b}$  has charge zero, we have  
$(H^o_0 {\cdot} v^{- c}_{-b})\circ \varphi = H^o_0 {\cdot} (v^{- c}_{-b}\circ \varphi) = 0$
and  the conditions for the applicability of Lemma \ref{lemma_split} 
hold for $v^{-c}_{-b} \circ \varphi$. We deduce that  $v^{-c}_{-b}$ is constant 
along orbits  of $H^o_0, H_{++}$  and $H_{--}\,$,  the images  under the map $\varphi$
of the orbits of $\Sp_1(\bC)$ in harmonic space. 
Thus, $v^{-c}_{-b} = 0$ everywhere on $\cU$ if and only if   
 $\ v_{-b}^{-c}|_{\f(\{I_2\} \times \{I_{2n}\} \times \cV)} = 0$. 
 From \eqref{v-c-b},    this follows  if and only if, for any $z \in \cV$, 
 \begin{multline*}
0 =  \Big( \frac{\p \psi^c_f}{\p B^d_e} (H_{--}^o  {\cdot} \varphi^{-f} + \varphi^{+f})
                                                    \Big)\Big \vert_{(I_2,I_{2n},z)}  
                                       \frac{\p \Phi^d_e}{\p z^{+b}}\Big|_{\f(I_2,I_{2n},z)} 
 \\    
\quad + \Big(   \frac{\p \psi^c_f}{\p z^{\pm d}} (H_{--}^o  {\cdot} \varphi^{-f}+ \varphi^{+f})
             \Big)\Big\vert_{(I_2,I_{2n},z)} 
          \frac{\p \Phi^{\pm d}}{\p z^{+b}}\Big|_{\f(I_2,I_{2n},z)}
\\    
\quad -\Big(  \frac{\p (H_{--}^o  {\cdot} \varphi^{-c} + \varphi^{+c} )}
                                              {\p z^{\pm d}}\Big)\Big \vert_{(I_2, I_{2n}, z)} 
       \frac{\p \Phi^{\pm d}}{\p z^{+b}}\Big|_{\f(I_2,I_{2n},z)}  
 \end{multline*}
This holds since  the initial data satisfy \eqref{strange_constraint}.  
\end{proof}
Since the functions  $v^{-c}_{-b}$ are identically vanishing, the vector fields $e_{-b}$ 
have the form $e_{-b} = e^o_{-b} + v_{-b}^{+c} e^o_{+c} + A^B_{-b} E^o_B$. 
It follows that 
$$Y_{ab} := \varphi_*(\wh Y_{ab}) =[e^o_{+a} ,  e_{-b}] 
=T^{+ c}_{ab} e^o_{+ c} + R_{ab}^B E^o_B\,,
$$   
with 
$T^{+c}_{ab}=  e^o_{+a}{\cdot} v_{-b}^{+c} - A^B_{-b} (E^o_B)^c_a$ and 
$R_{ab}^B = e^o_{+a}{\cdot} A^B_{-b}$.
From  the equations $[H^o_0, [e_{+a}, e_{-b}] ] = 0\ $ and 
$\ [H^o_{++}, [e_{+a}, e_{-b}]] = 0 $, 
we  see that  $T^{+ c}_{ab}$  has charge $-1$ and  that
$\ H_{++} {\cdot} T^{+c}_{ab} = 0$.
By  Lemma \ref{lemma_split}a), applied to the functions $\wt T^{+c}_{ab}:= \varphi^*(T^{+c}_{ab})$,  
it follows that $T^{+c}_{ab} = 0$ and that 
$\wh Y_{ab} = [\wh e_{+a}, \wh e_{-b}] = \varphi^{-1}_*([e_{+a}, e_{-b}])$   
has terms only in the directions of the  $E^o_A$, as required. 
This concludes the proof that  $\cA$ is an hk-frame.

We now observe that the above construction, together with \eqref{e-} and Lemma \ref{newlemma}, 
shows that  $\cA$  is in fact canonical and that $\f$ is a bridge from the central hk-frame $\wh \cA = \f^{-1}_*(\cA)$ to $\cA$. It therefore remains to show that there exists a $4n$-dimensional 
real submanifold $M \subset \cU$ such that  $(\cA, M)$ is a canonical hk-pair.

Let us consider the distribution $\cD \subset T \cU$,  generated by  
$$
\a^{\wh \cA}_{(\bR)}(v)_x\ ,\ v \in V^\t \ ,\  x \in \cU\,,
\quad\text{and}  \quad \a^{\wh \cA}_{(\bR)}(E)\ ,\ E \in \gsp_n = \gsp_n(\bC)^\t , 
$$
where  $\a^{\wh \cA}_{(\bR)}$ is  the real absolute parallelism \eqref{reality} associated with ${\wh \cA}$.
Its  image   $\cD' = \pi_*(\cD) \subset T \cV$ under the natural projection  
$\pi\colon\cU \to \cV$ $  \simeq \{I_2\} \times \{I_{2n}\} \times \cV$ is  a totally real, $4n$-dimensional  distribution and it is 
is involutive by virtue of the   Lie brackets of  vector fields in  $\cA$.  
By  Frobenius' Theorem, $\cD'$ admits integral submanifolds.   
Let $M' \subset  \{I_2\} \times \{I_{2n}\} \times \cV$  be an integral submanifold  through $e =(I_2, I_2, 0)$. 
We now   show  that $M:= \varphi(M')$ is totally real and satisfies  the conditions of Def.\ \ref{hk-pair}. 

$M$ is totally real  because it is the image under a biholomorphism of a totally real submanifold. 
Condition (i) of Def.\ \ref{hk-pair} holds because $\varphi$ is $\gsp_n(\bC)$-equivariant and  $M'$ 
is transversal to the $\gsp_n(\bC)$-orbits. Finally, from \eqref{H_li} and \eqref{e-hor}, 
the real absolute parallelism  $\a^{\cA}_{(\bR)}$   associated with $ \cA = \varphi_*(\wh \cA)$ 
is such that   for any $x = \varphi(y) \in M = \varphi(M')$ and $v \in V^\t$
\begin{align*}
\a^{\cA}_{(\bR)}(v)_x   = \varphi_*( \a^{\wh \cA}_{(\bR)}(v)_y) 
 \in\ & \varphi_*\big( T_y M'  +  \a^{\wh\cA}_{(\bR)}(\gsp_{p,q})\vert _y\big) \\
& =  T_x M +   \varphi_*\big(\a^{\wh\cA}_{(\bR)}(\gsp_{p,q})|_y\big) \\
& = T_x M +  \a^{\cA}_{(\bR)}(\gsp_{p,q})|_x\ . 
\end{align*}
Hence condition (ii)  of Def.\ \ref{hk-pair}  holds as well.  This  concludes the proof that $(\cA, M)$ is a canonical  hk-pair.
\end{proof}

\begin{rem} The role of condition  \eqref{strange_constraint}  in this proof is merely 
to simplify the proof of the existence of a submanifold on which the functions $v^{-a}_{-b}$ vanish. 
Indeed, the argument at the end of Sect.\ \ref{pf_vpot_thm} shows that for every pair
of solutions $\varphi, \varphi'$ of \eqref{secondsystem}, the vector fields 
$H_{--} = \varphi_*(H^o_{--})$ and $H'_{--} = \varphi'_*(H^o_{--})$ necessarily coincide.  
This   means  that {\it any} solution $\varphi$ of \eqref{secondsystem}, 
not necessarily satisfying \eqref{strange_constraint}, can be used to construct 
the (unique) vector fields $H_{--}$ and $e_{-a} = [H_{--}, e^o_{+a}]$  required to complete
the  vector fields $H^o_0, H_{++}, E^o_A, e^o_{+a}$ to a canonical hk-frame.  
\end{rem}

\goodbreak
\noindent
{\bf Step 2: Existence of a prepotential for  any canonical hk-pair}

\noindent
Let $\cA= (H^o_0, H_{\pm\pm},  E^o_A, e_{\pm a})$ be a canonical  hk-frame 
defined  on an appropriate open subset  $\cU \subset \cP$. 
It follows from \eqref{algebra} and \eqref{curv_constr}  that 
the   maps $(v^{-a}_{++}) , (v^{+a}_{\pm\pm}), (v^{+b}_{- a})$ 
and $(A^B_{-a} (E^o)^b_{c})$, 
considered as components  taking values in   $V = \bC^{2n} ,\, \ggl_{2n}(V)$ and 
$V {\otimes} V^* {\otimes} V^*$, respectively,    are $\gsp_n(\bC)$-equivariant.
Further,  we have
\begin{align*}
[H_{++} , e_{+a}]  
&=  [H^o_{++} + v^{\pm b}_{++} e^o_{\pm b}  + A_{++}^B E^o_B, e^o_{+a}] = 0 \\
 [H_{++} , e_{-a}] 
 &=   [H^o_{++} + v^{\pm b}_{++} e^o_{\pm b}  +  A_{++}^B E^o_B, e^o_{-a}  
              + v_{-a}^{+ c} e_{+ c}^o + A^B_{-a} E^o_B]  = e^o_{+a}\ .
\end{align*}
Comparing   components along $e^o_{-b}$ on both sides of  
these equations, we see  that 
$\ e^o_{+a} \cdot v^{- b}_{++}  = 0\ $ and 
$\  e^o_{-a} \cdot v^{-b}_{++}  
= A_{++}^A (E^o_A)^b_a \in \gsp_n(\bC)$, 
or equivalently,  
 \beq \label{symplectic1}
 \o_{cb}\, e^o_{-a}\cdot v^{-c}_{++} \ -\o_{c a}\, e^o_{-b}\cdot v^{-c}_{++}  = 0\,. 
\eeq
This means  that  on  $\cH|_{\cV} = \cU \cap \cH$, we have 
\beq \frac{\partial (\o_{cb}v^{- b}_{++}) }{\partial z^{+a}} = 0\ ,\qquad \frac{\partial (\o_{cb}v^{- b}_{++}) }{\partial z^{-a}} = \frac{\partial (\o_{ab}v^{- b}_{++}) }{\partial z^{-c}}\ , \eeq
so that   there exists a holomorphic  prepotential $\cL_{(+4)}$ of charge $k = 4$ and  independent of  $z^{+a}$,  such that 
\beq 
v^{-c}_{++} =\o^{bc} \frac{\partial \cL_{(+4)}}{\partial z^{-b}}\, ,\quad \ 
(\o^{ab}) = (\o_{ab})^{-1}\ .
\eeq
This prepotential is determined up to an arbitrary function depending only on $u^i_\pm$, 
which is fixed by  imposing the initial value 
$\,\cL_{(+4)}|_{\Sp_1(\bC) \times \{I_{2n}\} \times \{0\}} = 0$.

\section{Construction of a pseudo-hyperk\"ahler metric from its prepotential}
\label{reconstruct}
\setcounter{equation}{0}
\noindent 
In this section we summarise the correspondence between prepotentials and metrics,  giving a recipe
 to construct a   real analytic pseudo-hyperk\"ahler metric from a specified prepotential   
$\cL_{(+4)}\colon \cH|_{\cV} \rightarrow \bC$.
 
 \smallskip
 \noindent{\it {\bf Step 1.}  Construct  the vector field $H_{++}|_{\cH}$}. 
 
\smallskip 
\noindent The vector field $H_{++}$ of the canonical hk-pair, corresponding to $\cL_{(+4)}$ is 
of the form
$ H_{++} = H^o_{++} + v^{-b}_{++} e^o_{-b} +  v^{+b}_{++} e^o_{+b} + A_{++}^B E^o_B$.
The components of its restriction $H_{++}|_{\cH|_\cV} $  are given by 
 $$
 \left.v^{-b}_{++}\right|_{\cH} =\o^{bc}  \frac{\partial \cL_{(+4)}}{\partial z^{-c}},\  \ 
  v^{+a}_{++} = \o^{ac}  \frac{\partial^2 \cL_{(+4)}}{\partial z^{-b} \partial z^{-c}} z^{+b},\  \ 
   \left. A^B_{++} (E^o_B)^a_b\right|_{\cH} = 
     \o^{ac}  \frac{\partial^2 \cL_{(+4)}}{\partial z^{-b} \partial z^{-c}}\,.
$$
The components at  other points of $\cP|_{\cV} = \cH|_{\cV} \cdot \Sp_n(\bC)$ are  determined  using $\Sp_n(\bC)$-equivariance.

\goodbreak\vfil
 \noindent{\it {\bf Step 2.}   
 Construct a   bridge  $\varphi$}. 
 
\noindent 
Determine holomorphic functions 
$\varphi^a_b$ and $\varphi^{\pm a}$ on   $\cH|_\cV$ by  solving the system of equations
\begin{align*}
\nonumber H_{++}^o \cdot \varphi^{-a}  &=
\o^{ab} \left. \frac{\partial \cL_{(+4)}}{\partial z^{-b}}\right|_{(u^i_{\pm}, \varphi^{- c})}
\\[3pt]
 H_{++}^o \cdot \varphi^{+a}    &=  
   \varphi^{+b} \o^{ac}\left.\frac{\partial^2 \cL_{(+4)}}{\partial z^{-b} \partial z^{-c}}\right|_{(u^i_{\pm}, (\varphi^{- d}))} 
 + \varphi^{-a}  
\\[3pt]
H_{++}^o \cdot \varphi^a_b &=    
\o^{ac}   \varphi^d_b \left.\frac{\partial^2 \cL_{(+4)}}{\partial z^{-c} \partial z^{-d}}\right|_{(u^i_{\pm}, (\varphi^{- a}))},
\end{align*}
with $\varphi^a_b(I_2, 0) = \d^a_b$, $\varphi^{\pm a}(I_2, 0) = 0$.
Then extend  $\varphi^{i a} = - u^i_+ \varphi^{+a} - u^i_- \varphi^{-a} $, as constant functions along $\Sp_n(\bC)$ orbits,  to   the appropriate open subset 
$\cU = \cH|_{\cV} \cdot \Sp_n(\bC)\subset \cP$ and extend the $\varphi^a_b|_{\cH|_\cV}$ to $\cU$ using  \eqref{phi_trfn}.  Now set  $\f = (\f^i_\pm = u^i_\pm, \f^a_b, \f^{ia})$.

\smallskip 
\noindent{\it {\bf Step 3.} Construct the hk-frame $\cA$}.

\noindent 
Set: 
 $H_0 = H^o_0\ ,\ \ H_{--} = \varphi_*(H^o_{--})\ ,\ \ 
 E_A = A^o_A\ ,\ \ e_{+a} = e^o_{+a}\ ,\ \ 
 e_{-a} = [H_{--}, e^o_{+a}]$. 
 
 \smallskip
 \noindent{\it {\bf Step 4.}  Determine the manifold $M$ for the hk-pair $(\cA, M)$}.
 
\noindent 
Consider the  integrable distribution $\cD'$ on $ \cV \subset \bC^{4n}$ spanned 
by the real and imaginary parts of the vectors 
\[
\begin{split}
\wh e^{(U)}_{a}|_z  
&:= (\pi \circ \varphi^{-1})_*(e^o_{+ a} |_{(U,z)}  +  \wh \bJ_c^d e_{- d}|_{(U,z)})\\ 
 \wh e^{(U)}_{a + 2n }|_z  
&:= (\pi \circ \varphi^{-1})_*(e_{- a}|_{(U,z)}  -  \wh \bJ_c^d e^o_{+ d})|_{_{(U,z)}})\ , 
\end{split}
\]
where  $(\wh \bJ_c^d) :=   
- I_{2p, 2q} {\cdot} \bigl( \begin{smallmatrix}  0&- I_n \\ I_n &0 \end{smallmatrix}\bigr) $, 
 $\ U \in \Sp_1(\bC)$ and $\pi\colon\cH|_{\cV} \to  \cV$ is the standard projection. 
We need to find  an integral submanifold $M' $  of   $\cD'$ through   $0 \in \cV$.
 This can be done, for instance, by choosing vector fields which locally   generate $\cD'$ 
 and considering an orbit of $0$ under the flows of these vector fields. 
 Then, set   $M = \varphi(M')$, where $M'$ is considered as a submanifold of 
 $\{I_2\} \times \{I_{2n}\} \times \cV \simeq \cV$. 

\smallskip  
 \noindent{\it {\bf Step 5.}    Construct  the pseudo-hyperk\"ahler metric}.

 \noindent
Find   the dual coframe field   $\cA^* = (H^0, H^{\pm\pm}, E^A, e^{\pm a})$ of 
the hk-frame $\cA = (H_0, H_{\pm\pm}, E_A, e_{\pm a})$. 
A   pseudo-hyperk\"ahler metric on $M$ in the isometry  class   associated with $\cL_{(+4)}$ is: 
$$
g = \sum_{a = 1}^{2n}\left.\left( e^{+a} \vee  e^{-a}\right)\right|_{T  M \times T M}\ .
$$

\medskip
To conclude this section, we summarise  the inverse construction of a prepotential 
from a given pseudo-hyper\-k\"ahler metric $g$. 
In order to determine this (non-unique) prepotential, we first  construct 
an hk-pair $(\wh \cA, M')$ associated with $g$, 
following the procedure in the proof of Lemma \Aref{appr_lem}.  
Here $\wh \cA {=} (H^o_0, H^o_{\pm\pm}, E^o_A, e_{\pm a})$ is a central  hk-frame 
on  $\cU = \cH|_{\cV}\, {\cdot}\,\Sp_n(\bC) \subset \cP$ and 
$M' \subset \{I_2\} {\times} \{ I_{2n}\} {\times} \cV$ with $e = (I_2, I_{2n}, 0) \in M'$.   Applying  a local 
biholomorphism of $\bC^{4n}$ if required, we may choose  $\wh \cA$ so that  
$e_{-a}|_{\cV} =\frac{\partial}{\partial z^{-a}}\big|_{\cV} {+} A_{-a}^B E^o_B|_{\cV}$. 
Then, we construct a  bridge  
$\varphi = (\varphi^i_\pm = u^i_\pm, \varphi^a_b , \varphi^{\pm a} 
= - u^{\pm}_i \varphi^{ia})$ with $\varphi(e) = e$ by solving  the differential problem
\begin{align*}
H^o_0 \cdot \varphi^{- a} &= - \varphi^{- a},   
&e_{-a} \cdot \varphi^{-a} &= \f^a_b \,, 
& e_{+a} \cdot \varphi^{- b} &= 0\ ,
\\
H^o_0 \cdot \varphi^{+ a} &=  \varphi^{+a},   
&e_{+a} \cdot \varphi^{+ b} &= \f^a_b\,, 
&H_0^o \cdot \varphi^a_b &=  0\, , 
\end{align*}
together with the conditions:
 \begin{itemize}[itemsep=6pt, leftmargin=24pt]
 \item[i)] $\varphi^{ia}$ is independent  of  $B^a_b$ 
 \item[ii)] $ \varphi^a_b\left(U, B, z\right) = \varphi^a_c\left(U,I_{2n}, z\right)B^c_b$ for any $(U, B, z) \in \cU$
 \item[iii)] $H_{++}^o {\cdot} \varphi^{+ b}|_{\{z^{+a} = 0\}}= \varphi^{-b}|_{\{z^{+a} = 0\}}$.
 \end{itemize}
 Computing $H_{++}:= \varphi_*(H^o_{++})$, the restriction to $\cH|_\cV$  
 of the components $v^{-b}_{++}$ in the expansion
 $H_{++} = H^o_{++} + v^{\pm b}_{++} e^o_{\pm b} + A_{++}^B E^o_B$ 
 gives the v-potential of the metric.
 Finally,   the $z^{+a}$-independent potential $\cL_{(+4)}$  for the exact 1-form 
 $\a := \o_{ab} v^{-b}_{++} dz^{-a}$ on $\cH|_{\cV}$, with  
 $\cL_{(+4)}|_{\Sp_1(\bC) \times \{I_{2n}\} \times \{0\}} = 0$,  
 is the required prepotential.

\appendix

\makeatletter
\def\subsection{\@startsection{subsection}{3}%
  \z@{.5\linespacing\@plus.7\linespacing}{.1\linespacing}%
  {\normalfont\bfseries}}
\makeatother

\makeatletter
\def\subsubsection{\@startsection{subsubsection}{3}%
  \z@{.5\linespacing\@plus.7\linespacing}{.1\linespacing}%
  {\normalfont\bfseries}}
\makeatother

\renewcommand*\thesubsection{\Alph{section}\arabic{subsection}}
\renewcommand*{\theequation}{\thesection\arabic{equation}}
\renewcommand*{\thedefinition}{\thesection\arabic{definition}}

\section{$G$-structures and pseudo-hyperk\"ahler manifolds}
\setcounter{equation}{0}
 \noindent 
In this appendix  we introduce real and complex $\gg$-structures, local 
reformulations of $G$-structures  in terms of  vector  fields.
They provide a useful tool for the investigation of local properties of 
manifolds with  real analytic $G$-structures.

In Sect.\ \ref{G-str} we show that there exists a  natural  one-to-one correspondence 
between  local equivalence classes of  (a) $G$-structures with connections and 
(b) complete $\gg$-structures.
This  correspondence allows the formulation of questions on local equivalences 
of $G$-structures in terms of local equivalence problems among
sets of vector fields.
 
We then discuss (Sect.\ \ref{complex}) 
{\it complexifications of real  $G$-structures} and  {\it real forms 
of complex $G$-structures},  with a view to expressing problems  of  
equivalence among real analytic $G$-structures in terms of 
holomorphic vector fields. 
In  Sect.\ \ref{appendixHK} we discuss the particular case of   $G$-structures 
corresponding to real analytic pseudo-hyperk\"ahler metrics.  
As our main result  in Sect.\ \ref{bijection}, we prove the bijection  between 
local isometry classes  of real analytic pseudo-hyperk\"ahler metrics and 
local equivalence classes of hk-pairs, which was advertised  in 
Sect.\ \ref{hkpairs-metrics}.

\subsection {$G$-structures and associated  $\gg$-structures} 
\label{G-str}
\noindent 
We start with a slight generalisation of the classical notion  of  a $G$-structure.

\begin{Adef} \label{Gstr} 
Let $G$ be a real Lie group  admitting an almost exact linear representation  
$\,\r\colon G \,{\to}\, \GL(W)\,,\,W = \bR^n$, i.e.\  $\ker \r$ is a discrete normal subgroup. 
A \hbox{\bss $G$-structure} $(P,\q)$ on an $n$-dimensional  manifold $M$ is 
a principal $G$-bundle $\pi\colon P \,{\rightarrow}\, M$ together with 
a {\bss soldering form} $\q \colon TP \,{\rightarrow}\, W$, a $G$-equivariant  $W$-valued 
1-form which is strictly horizontal, namely   the vertical distribution $T^vP$ of $P$ is such that  
$T^v_u P=\ker \q_u$ for any $u \in P$.
\end{Adef}

\begin{Arem} \label{G-remark} 
This  definition  can be thought of as a  a minor generalisation of the classical notion  
of  a $G$-structure  as  a $G$-reduction of the linear frame bundle $L(M)$ of  $M$  
(see e.g.\,\cite{KN, St}).  Various examples  motivate this generalisation.
In particular, the  $\Spin_n$-bundle of  a Riemannian manifold $(M,g)$
is not   a $G$-structure in the classical sense, but is indeed a  
$\Spin_n$-structure in the sense of Def.\ \Aref{Gstr}.
The relation between the two definitions may be understood as follows.
Let $(e^o_i)$ be a fixed basis  of $W$ and choose a point $u \in P$.  
Then find  $n$ vectors $\wh e_i \in T_u P$ satisfying the equations  
$\q_u(\wh e_i) = e^o_i\,$. These vectors are determined up to elements in 
$\ker \q = T^v P$, so that  their projections   $e_i := \pi_*(\wh e_i) \in  T_{\pi(u)} M$,
 are uniquely associated  with  the  point $u \in P$. 
Thus, there exists  a well-defined map  
$$
\wt p\colon  P \rightarrow L(M)  \ ,\quad  u \mapsto  \wt p(u) := (e_i) \subset T_{\pi(u)} M\,.
$$
In virtue of  the $G$-equivariance of $\q$, we may check that  $\wt p$ is  $G$-equivariant, 
namely that $\wt p(u {\cdot} g) = \wt p(u) {\cdot} \rho(g)$, for $u \in P$ and  $g\in G$. 
This property,  together with the assumption that $\r\colon G \rightarrow  \GL(W)$  is 
almost exact, implies that $P' := \wt p(P) \subset L(M)$ is a $\r(G)$-reduction of $L(M)$,  
that $p\colon P \to P',\,  p(u) := \wt p(u)$,  is a covering map and that $\q =  p^*(\q')$, where $\q'$ is the 
soldering form of $P'$.

Summing up, a   {\em $G$-structure}, as defined in Def.\,\Aref{Gstr}, 
{\em always admits   a covering  map $p\colon P \to P'$ onto  a 
$\r(G)$-reduction $P'$ of  the linear frame bundle $L(M)$  such that
$\q$ is  the pullback, $\q = p^*(\q')$, of the soldering form $\q'$ of $P' \,{\subset}\, L(M)$}. 
\end{Arem}

We recall that a connection on a principal $G$-bundle  $ \pi \colon P \rightarrow M = P/G\ $
is a \hbox{$G$-equivariant} $\mathfrak{g}$-valued 1-form 
$\omega \colon TP \rightarrow \mathfrak{g} = \Lie(G)$, 
for which the restriction  $\omega|_{T^v_u P}$  to any vertical subspace $T^v_uP$  
coincides with the inverse of the canonical identification 
$\nu \colon \mathfrak{g} \rightarrow T^{v}_u P$ between  $\mathfrak{g}$ and  $T^{v}_uP$.
A connection $\omega \colon TP \rightarrow \gg$ on a  $G$-structure 
$\,(\pi \colon P \rightarrow M\,,\, \q)\,$ yields  a Cartan connection,  
\beq \label{cartan}
\kappa :=  \omega + \q\ ,\quad
\kappa  \colon TP \rightarrow \gq := \gg + W\ ,
\eeq
namely a $\gq$-valued $1$-form which  a)  has  trivial kernel,  
b) extends   the natural isomorphism  
$ \nu^{-1}\colon T^v_u P \rightarrow \gg$ at every $u\in P$ and
c)  satisfies  $(R_g^* \kappa)(v) = \Ad_{g^{-1}} ( \kappa(v))$ 
for  $g \in G$ and $v \in TP$,  where  $R_g\colon P \to P$ is 
the right action of $g$ on $P$. 
The notion of a Cartan connection is related to the following:

\begin{Adef}  
An {\bss absolute parallelism} on an $n$-dimensional manifold  $N$ is  
an  $\bR$-linear map from a fixed $n$-dimensional real vector space, 
say $\bR^n$,  into the space of smooth vector fields $\gX(N)$, 
$$
\b\colon   \bR^n \rightarrow \gX(N) \,,
$$
with the property that  the induced map 
$\b_x\colon \bR^n \to T_x N\ ,\  \b_x(v) := \b(v)|_x$, 
is an isomorphism of vector spaces for each $x \in N$.
\end{Adef}

\noindent
The existence of an absolute parallelism $\b$ on $N$  is equivalent to  
the existence of  a set  $\{X_1, \ldots, X_n\}$ of $n$ vector fields 
$X_i \in \gX(N)$,  such that for every $x \in N$ the vectors 
$X_1|_x, \ldots, X_n|_x$  form a basis of $T_xN$.  In fact, for a given $\b$,  such  vector fields   are  images
$X_i = \b(e_i^o)$ of the elements of    some  basis $(e_i^o)$ of $\bR^n$.

It follows  immediately  that  if $\kappa$ is the $\gq$-valued one-form \eqref{cartan},
the   $\bR$-linear  map
\beq \label{alpha}
\a \colon \gq =  \gg +W\rightarrow  \gX(P)\,,\quad 
\a(X)|_u := \k_u^{-1}(X)\quad  \mbox{for}\  u\in P \  , \  X\in \gq\,, 
\eeq
 is an absolute parallelism on $P$. 
So, if  $(e^o_1, \dots, e^o_n)$ and $(E^o_1, \dots, E^o_N)$ are two  fixed bases 
for  $W$  and $\gg$, respectively,  the absolute parallelism $\alpha$ and 
consequently $\k$ (which is the inverse of $\a$ in the sense of \eqref{alpha}) 
are uniquely determined by  the corresponding set of vector fields  
$ \cA = ( e_i =  \a(e^o_i),  E_A =  \a(E^o_A))$, 
which provides a field of linear  frames for the tangent spaces of  $P$.
The absolute parallelism $ \a$,  constructed   from a connection on a 
$G$-structure $(P,\q)$, is an example  of special class of absolute parallelisms, 
which we call  {\it $\gg$-structures}.
Let $\gg \subset \ggl(W)$ be a {\it real} linear Lie algebra and
$\gq := \gg + W$ the associated nonhomogeneous Lie algebra,  with  
$[W, W] = 0$ and $[A\,,\, v] = A {\cdot} v$ for  $A \in \gg$ and $v \in W$.  

\begin{Adef} \label{gg-structure} 
Let $P$ be a  manifold  with $\dim P = \dim \gq$.
A {\bss $\gg$-structure} is an  absolute parallelism
$\,  \alpha \colon  \gq \rightarrow \gX(P)\,,\,  \gq:= \gg + W$, satisfying  the 
following Lie bracket relations:
\beq \label{gg-str} 
 [ \alpha(A),  \alpha(B)] =  \alpha([A,B])\  , \qquad  A \in \gg\ ,\  B\in \gq\,.
\eeq
Two $\gg$-structures  $  \alpha,  \alpha'$ on  $P$ are  {\bss equivalent} if 
there exists a diffeomorphism $\varphi$ of $P$ such that $ \a' = \varphi_* \circ  \a$.
The  {\bss vertical} and  {\bss horizontal distributions} of a $\gg$-structure are the 
distributions  $\gV$ and $\gH $ in $TP$  generated by  $\a(\gg)$  and $\a(W)$, respectively.  
\end{Adef}

The  brackets not included in \eqref{gg-str}, between a pair of  horizontal 
vector fields  $\a(v)$ and  $\a(v')\,,\, v, v' \in W $,   take the form
\beq 
[ \alpha(v),  \alpha(v')] = T(v,v') + R(v,v')\ ,\qquad v, v' \in W\ ,
\eeq
where we denote by $T(v,v')$  and $R(v,v')$  the  components  of 
$[ \alpha(v),  \alpha(v')] $ along  $\gH$ and $\gV$ respectively. 
The  maps 
\begin{align*}
T_u \in \Hom(\L^2 W, W) \, ,\quad  T_u(v, v') &:= \a_u^{-1}(T(v, v')|_u)
\\
R_u \in \Hom(\L^2W, \gg) \,,\ \quad  R_u(v, v') &:= \a_u^{-1}(R(v, v')|_u)\,,
\end{align*}
are  respectively called  {\bss torsion\/}  and {\bss  curvature\/}  of  the
$\gg$-structure $\a$ at $u\in P$.
These  generalise the classical notions of torsion and curvature of a connection. 
As we shall see, if  $\a$ satisfies certain additional conditions,  
there exists a right $G$-action on $P$, a $W$-valued 1-form $\q$ 
and a connection 1-form $\o$ on $P$ such that  $(P, \q, \o)$ is a $G$-structure  
with  a connection  having the property that  $\,\o + \q = \a^{-1}$,  
in the sense of  \eqref{alpha}. 
Then, given bases $(E^o_A)$ of $\gg$ and $(e^o_i)$ of $W$, 
the components $ T^k_{ij}$ and $R^A_{ij}$ of 
$T_u = T^k_{ij} e^o_k \otimes e^{oi} \otimes e^{oj}$ and 
$R_u = R^A_{ij} E^o_A \otimes e^{oi} \otimes e^{oj}$
are precisely   the  components of the torsion  and of the curvature of the connection $\o$ 
in the linear frame  $(e_i) = p(u) \in P' \subset L(M)$ (see Remark \Aref{G-remark}).

\begin{Arem}\label{gg_tant}
The conditions \eqref{gg-str} are tantamount to the following: 
\begin{itemize} [itemsep=2pt, leftmargin=18pt]
\item[a)] The map $\a|_{\gg}:\gg \rightarrow  \alpha(\gg)$ is  a faithful representation  
of $\gg$ in the Lie  algebra of vector fields.
\item[b)] The adjoint representation of $\a(\gg)$ in $\alpha(W)$ is equivalent to the  
linear representation of $\gg\subset \ggl(W)$ on $W$.
\end{itemize}
\end{Arem}

\begin{Adef} \label{completeness}
A $\gg$-structure  $ \alpha$ on $P$  is called  {\bss complete} if 
\begin{itemize} [itemsep=2pt, leftmargin=18pt]
\item[a)]  the Lie algebra $ \alpha(\gg)$  of vector fields  of $P$  defines a free 
right action of  a corresponding connected Lie group $G$   on $P$ and
\item[b)] the orbit space $M = P/G$ is a smooth manifold and the projection 
$\pi \colon P \rightarrow M = P/G$   is a locally trivial fibration.
\end{itemize}
\end{Adef}
\noindent
The absolute parallelism $\a$ on a $G$-structure $(P, \q)$ with 
a connection $\o$ given in \eqref{alpha} is a complete $\gg$-structure. 
The following proposition shows that this correspondence is in fact invertible. 
Thus complete $\gg$-structures  are in bijection with  
$G$-structures  endowed with  a connection.

\begin{Aprop}  \label{g-G-str} 
Let $\a$ be a complete $\gg$-structure   on a manifold $P$,
which has a free right $G$-action $\ \rho\colon P \times G \to P$.
Then there exists
\begin{itemize} [itemsep=2pt, leftmargin=18pt]
\item[i)]   a $W\!$-valued 1-form  $\vartheta\colon TP \to W$,
such that $(P, \q)$ is a $G$-structure and 
\item[ii)] a connection $\o$ on  the $G$-structure $(P, \q)$,
\end{itemize}
with the property that  the Cartan connection $ \k = \o + \q$
 is the inverse of the map $\a$ in the sense of eq.\ \eqref{alpha}.
\end{Aprop}

\begin{proof}  
By  completeness,  $P$ is a principal $G$-bundle over $M = P/G$, 
where $G$ is  the connected group  generated  by the Lie algebra of vector fields 
$ \alpha(\gg)\subset \gX(P)$. 
By  Remark \Aref{gg_tant}b),  $G$  has a linear representation on the vector space 
$  \alpha(W) \simeq W $.  
This  defines an almost exact  representation of $G$ in $W$. 

Now, for $X\in T_uP\,,\, u\in P$, consider the  natural projections 
$(X)^{\gH}$ and  $(X)^{\gV}$  onto the horizontal and vertical subspaces 
$\gH_u, \gV_u \subset T_uP$. 
Since $\a(W)$ and $\a(\gg) $ generate $\gH$ and $\gV$, respectively,
there exist unique elements  $v\in W\,,\, E\in \gg$ such that  
$\a(v)|_u = (X)^{\gH}$ and  $\a(E)|_u = (X)^{\gV}$.  
Thus, the horizontal and vertical projections provide the mappings  
\begin{align*}
\q_u \colon\ T_uP \ni X &\longmapsto v \in W \\
\o_u \colon\ T_uP \ni X &\longmapsto E \in \gg\,.
\end{align*}
By  construction,   $\q$ is $G$-equivariant,  $(P, \q)$ is a $G$-structure,  
$\o$ is a connection on the $G$-bundle $\pi\colon P \rightarrow M = P/G$ and $ \o + \q = \a^{-1}$.
\end{proof}

\subsection{Complex $G$-structures and their real forms}
\label{complex}
\noindent 
Given a complex manifold $(N, J)$,  its {\it holomorphic} and 
{\it anti-holomorphic tangent bundles}, $T^{10} N$ and  $T^{01} N = \overline{T^{10} N} $ 
are the subbundles of  $T^\bC N$ given by  the $+i$ and $-i$ eigenspaces, respectively,
of  the $\bC$-linear map $J_x\colon T^\bC_x N \rightarrow T^\bC_x N$.

We recall that  {\it holomorphic vector fields}  of $(N, J)$ coincide with  
complex vector fields  of $N$ of the form $X  =  Y - i J Y$ for  $Y \in \gX(N)$ 
satisfying  $\cL_Y J = 0$.  
This condition is  equivalent to say  that in any system of holomorphic complex coordinates 
$\xi = (\z^1, \dots, \z^m)\colon \cU \subset N \rightarrow \bC^m,\ m=\dim_\bC N$,  
the complex vector field $X = Y - i J Y$ has the form  $X = X^i \frac{\partial}{\partial \z^i}$  
with   $X^i = X^i(\z^1, \ldots \z^m)$  holomorphic in  the coordinates $\z^i$.

\subsubsection{Complex $G$-structures and  complex $\gg$-structures}
\label{cx.G-str}
\noindent 
In this section  $ G \subset \GL(V)\,,\,  V= \bC^n$, is 
a connected {\it complex} linear group with Lie algebra $ \Lie(G)=\colon\gg \subset \ggl(V)$
and $\gq := \gg + V$ the associated nonhomogeneous Lie algebra,  with  
$[V, V] = 0$ and $[A, v] = A \cdot v$ for  $A \in \gg$ and $v \in V$. 
We shall treat $\gg$ and $V$ as  $\bR$-vector spaces endowed with the standard
complex structures $J_o \colon \gg \to \gg$ and  $ J_o \colon  V\to V$.   
Denote by  $(E^o_A) $ and  $(e^o_a)$  fixed choices of  complex bases  for $\gg$ 
and  $V$, respectively. Further, let $\gg^{10}$ and $\gg^{01}= \ol{\gg^{10}}$, respectively, 
be the  $+i$ and $-i$ eigenspaces of $J_o$  in the complexification   
$\gg^\bC = \gg^{10} + \gg^{01}$ of $\gg$. 
Recall that $\gg^{10}$ is naturally isomorphic to $\gg$ as  complex Lie algebra and 
that  each holomorphic element   $X \in \gg^{10}$ has the form 
$X = Y - i J_o Y$  for some $Y \in \gg$. 
We may generalise Def.\ \Aref{Gstr}  to the case of a 
complex Lie group $G$ as follows:

\begin{Adef}
A {\bss complex $ G$-structure} $(P,\q)$ on an $n$-dimensional (real) manifold  $M$ 
is a principal $G$-bundle $\,\pi\colon  P \rightarrow M$  equipped with 
a {\bss complex soldering form} $\, \vartheta\colon T^\bC  P \rightarrow V$, 
a $V$-valued $\bC$-linear 1-form, which is 
\begin{itemize} [itemsep=2pt, leftmargin=18pt]
\item[i)] 
$G$-equivariant (i.e.\ $R_g{}_*  \vartheta = g^{-1} \cdot  \vartheta$ 
for all $g \in G$) and

\item[ii)] strictly horizontal (i.e.\ $\ker  \vartheta_u= T^v{}^\bC_u  P$ 
for all $u \in  P$, where  $T^v{}^\bC_u  P$ is the complexification of 
the vertical subspace $T^v_u  P \subset T_u  P$).
\end{itemize}
\end{Adef}

The main motivation for considering  {\it complex} $ G$-structures
comes from the following relation to the {\it real} ones.
Let $W = \bR^n$ and let  $H \subset \GL(W)$ be a real form of 
$ G \subset \GL(V)$, with $V = W^\bC = \bC^n$. 
A (real) $H$-structure $(\wt P, \wt \q)$, 
in the classical sense with  $\wt P \subset  L(M)$,
can be considered as a reduction of the $\GL(V)$-bundle $L^\bC(M)$  of 
complex linear frames of $T^\bC M$.   
Consider the unique $\GL(V)$-equivariant and  strictly horizontal  1-form  
$$ 
\wh \vartheta\colon T^\bC (L^\bC(M)) \longrightarrow W^\bC = V\qquad 
\text{with}\ \ \wh \q|_{T \wt P} = \wt \q
$$
and the  $G$-reduction  $P = \wt P \cdot G  \subset L^\bC(M)$. 
The pair    $(P, \q = \wh \q|_{T^\bC P})$ is  a complex $G$-structure, 
which we call the  {\bss complexification\/} of the $H$-structure $(\wt P, \wt \q)$.
We may therefore think of the class of complex $G$-structures as  
a natural generalisation of the  principal bundles,  which arise via the above   
complexification procedure from real $G$-structures of linear frames.

In the more general case, where $(\wt P, \wt \q)$ is a (possibly non-trivial) covering of an 
 $H$-structure  $(\wt Q, \wt \q)$ of linear frames $\wt Q \subset L(M)$, 
 a  {\bss complexification}  of $(\wt P, \wt \q )$
 is   a complex $ G$-structure $(P, \vartheta)$, with  $\wt P \subset P$, 
 which is a covering of the complexification of $(\wt Q, \wt \q)$.  
 
The  complexification procedure of a real $G$-structure allows reversal. 
Let $\tau\colon V \rightarrow V$ be a  $\bC$-antilinear involution and consider the 
induced  involutions on $ \ggl(V)$ and  $ \GL(V)$:
$$
\tau(A) :=  \tau \circ A \circ \tau\,,\quad
 \tau(g) :=  \tau \circ g \circ \tau \,,\ {\rm for\ all}\ A \in  \ggl(V)\ ,\ g \in  \GL(V) \,.
 $$
When $\gg \subset \ggl(V)$ and $G \subset \GL(V)$ are preserved by $\t$ 
we say that  $\tau$ is {\bss $\gg$-admissible} 
and we denote by $V^\tau ,  G^\tau $ and $ \gg^\tau$  
the $\tau$-fixed point sets in $V\,,\,  G$ and $ \gg$, respectively. 
Note  that in this case $ \gg^\tau$  is a real form of 
$ \gg$ (i.e.\ $ (\gg^\tau)^\bC$ is naturally isomorphic to $\gg$). 

\begin{Adef}  
Let $\tau\colon V \rightarrow V$ be  a $\gg$-admissible $\bC$-antilinear involution 
and  $(P, \vartheta)$ a complex $ G$-structure over $M$. 
A {\bss real form}  $(P^\tau, \vartheta^\tau)$  of $(P, \vartheta)$ 
is  a  $ G^\tau$-reduction  $P^\tau \subset P$ with soldering form  
$ \vartheta^\tau = \vartheta|_{TP^\t}$  taking values in  $V^\tau \simeq W = \bR^n$.
\end{Adef}
\noindent
Clearly,  if   $(P^\tau, \vartheta^\tau)$ is a real form of $(P, \vartheta)$,  
then  $(P, \vartheta)$ is the complexification of $(P^\tau, \vartheta^\t)$. 

Let $(P, \vartheta)$ be a complex $ G$-structure and  
$\omega\colon T P \rightarrow \gg$ a connection form on the $G$-bundle $P$. 
The $\bC$-linear extension  of the 1-form $\omega_u$ 
on the complexified tangent space  $T^\bC_u P\,,\, u \in P$,  
determines a $ \gg^\bC$-valued 1-form $\o$ on $T^\bC P$.
We call 
{\bss complex  Cartan connection} associated with  $\q$ and $\o$ the map
\be \label{cxCartan} 
\kappa \colon T^\bC P \longrightarrow  \gg^\bC  +  V\,,\quad
\kappa := \omega + \vartheta\, .
\ee
We note that the restriction 
$\k_u = \k|_{T^\bC_u P}\colon T^\bC_u P \to \gg^\bC + V$ 
is a  $\bC$-linear  isomorphism for every $u \in P$. 
Now, since $\k_u$ is  $\bC$-linear and $\gg \subset \gg^\bC$ 
a real form of $\gg^\bC$,  the inverse isomorphism 
$\k_u^{-1} \colon \gg^\bC + V \to  T^\bC_u P$ 
is uniquely determined by the map 
$$\a_u := \k_u^{-1}|_{\gg + V}\colon \gg + V \longrightarrow T^\bC_u P\ .$$
The family of linear maps $\a_u\,\,,\, u \in P$, combine into the single map
\beq\label{cx.parallel} 
\a\colon   \gg + V   \longrightarrow \gX^\bC(P)\,,\qquad \a(X)|_u := \a_u(X)\ ,
\eeq
which we call the {\bss complex (absolute) parallelism} associated with  $\k$. 
By definition, $\k$ is completely determined by $\a$. 
Further,  $\a$ has, by construction, the following properties: 
 
\begin{itemize} [itemsep=2pt, leftmargin=18pt]
\item[i)]  
$\a(X) = \overline{\a(X)}$ for all $X \in  \gg$ and 
$\a(J_o v) = i \a(v)$ for all $v \in V$.

\item[ii)] 
The vector fields $\a(X) \in \a( \gg)$  generate  the vertical distribution 
$\cD := T^v P\subset TP$, on   which  the complex structure $J_o$ on $\gg$
induces the family of  complex structures $J = \{J_u\}$  defined by
$$
J_u\colon \cD_u \rightarrow \cD_u\ ,\quad 
J_u\  \alpha(X)|_u := \a(J_o X)|_u\ ; \ \ X \in \gg \,,\, u\in P\,.
$$
The pair $(\cD, J)$ is a  CR structure.
We denote by $Y^{10} := \frac{1}{2}(Y {-} i J Y)  \in \cD^\bC$, for $Y \in \cD$,  
the unique complex vector  field  satisfying
$J Y^{10} = i Y^{10}$ and  $Y = Y^{10} + \overline{Y^{10}}$.  

\item[iii)]   
For  $X  \in  \gg$ and $v \in V$,
$$  [\a(X),\a(v)] = \a(X \cdot v)\ ,\ \qquad   [\a(J_o X), \a(v)]  =   i \a(X \cdot v)\ .$$

\item[iv)]  
The complex parallelism $\a$  uniquely determines  the following pair of objects: 
\begin{itemize}  [itemsep=2pt, leftmargin=18pt]
\item[a)]  the CR structure  $(\cD, J)$ and
\item[b)]  the collection of vector fields  in $T^\bC P$,
$$\cA^{(\a)} = 
\left( e_i = \a(e^o_i)  \ ,\  E_A  = \fr{1}{2}\left(\a( E^o_A) - i \a(J_o E^o_A)\right) \right),
$$
 which, together with the   vector fields $\overline{E_A}$,  form  a  field of 
 complex linear frames for  $T^\bC P$,  with   the fields $E_A\,,\, \overline{E_A}$ 
 taking values in  $\cD^\bC$. 
  \end{itemize}
Conversely,  given a CR structure  $(\cD, J)$ and a collection of complex vector fields 
$\cA^{(\a)}$,  the complex parallelism $\a$ allows explicit determination.
In fact, given $(\cD, J)$ and the fields $(e_i, E_A)$, the  family of $\bC$-antilinear 
involutions $\overline{(\cdot)}\colon\cD^\bC_u \rightarrow \cD^\bC_u\,,\, u \in P$, defined  by
$$ 
 \overline{(A + i J A)} := A - i J_u A\ ,\quad  \overline{(A - i J A)} := A + i J_u A\ , \qquad 
 \text{for}\ A \in \cD^\bC_u\,,
$$
affords the construction of  $\a$ as the unique  $\bC$-linear  map 
$\a\colon\gg + V \to \gX^\bC(P)$ such that 
$$
\a(E_A^o)|_u = E_A|_u + \overline{E_A}|_u\ ,\qquad \alpha(e_i^o)|_u = e_i|_u\ .
$$
Note  that this means that   $\ E_A  = \a(E^o_A)^{10}$.
\end{itemize}
 
\noindent
Analogously to the absolute parallelisms of real $G$-structures with a connection,   
the map \eqref{cx.parallel} is an example of  
a special class of maps, called  {\it complex $ \gg$-structures}.

\begin{Adef} 
Let  $ \gg \subset \ggl(V),\  V = \bC^n\,,$ be a linear complex Lie algebra,
 $ \gp =  \gg + V$  the associated nonhomogeneous Lie algebra, with  
the standard complex structure
 $ J_o\colon V  \rightarrow V\ ,\ J_o\colon \gg \rightarrow  \gg$. 
Further, let  $P$ be a real manifold of  dimension 
$  \dim_{\bR} \gg + \frac{1}{2} \dim_{\bR} V$.
 A {\bss complex $ \gg$-structure\/}   is an $\bR$-linear map
 $\a\colon  \gp \rightarrow \gX^{\bC}(P)$
 such that, for   $X\in \gg\,,\, Y \in  \gp$ and $v \in V$,
\[
\a(X) = \overline{\a(X)}\,,    \qquad  \a(J_o v) = i \a(v) 
\]
and 
\beq
 \label{hgg-structure1}
 \begin{split}   
 [ \alpha(X)\,,\, \alpha(Y)]  &=   \alpha([X,Y])
\\   
[\alpha(J_o X), \alpha(v)]  &=   i \alpha(X \cdot v) \,,
\end{split}
\eeq
and the fields in $ \cA^{(\a)} = ( e_i := \a(e^o_i)   \,,\,    E_A :=\a(E^o_A)^{10})$
are $\bC$-linearly independent  at each $u \in P$.
The  {\bss CR structure of $\a$} is the pair $(\cD, J)$, consisting of
the distribution  $\cD_u = \Span_\bR\{\ \a(X)_u\ ,\ X \in  \gg\ \} $ and
the family of complex structures $J_u\colon \cD_u \rightarrow \cD_u$
defined by
$$
J_u \alpha(X)_u := \a(J_o X)_u\,.
$$
Two complex $ \gg$-structures  $ \alpha, \alpha'$ on  $P$ are called
{\bss equivalent}  if  there exists a diffeomorphism $\varphi$ of $P$ such that
$\a' = \varphi_* \circ \a$, where
 $\varphi_*$ is  extended to $T^\bC P$ by $\bC$-linearity.
\end{Adef}
\noindent
Conditions  \eqref{hgg-structure1} 
may be reformulated as follows: 
\begin{itemize} [itemsep=4pt, leftmargin=18pt]
\item[a)]
$\a|_{ \gg}\colon  \gg \rightarrow \alpha( \gg)$ is  an exact representation  of 
(the real Lie algebra underlying) $ \gg$ on the Lie  algebra of real vector fields in $TP$.

\item[b)]  
The adjoint representation of $\a( \gg)$ on $\alpha(V)$ is equivalent to 
the  linear representation of $ \gg\subset \ggl(V)$ on $V$.
\end{itemize}

The definitions of  vertical and  horizontal subbundles  $\gV, \gH \subset T^\bC P$,
determined by a complex $\gg$-structure $\a$, as well as  
the  torsion and  curvature of $\a$, are analogous to those for a real $\gg$-structure.
The  vector fields of a  complex $\gg$-structure 
in the vertical subbundle $\gV \subset T^\bC P$  generate,  as in the real case,  
a local action of the complex Lie group $G$, whose orbit space, locally identifiable
with the base of the  principal $G$-bundle $P\ra M$,  has real dimension
$\dim P - \dim_\bR \gg = \frac{1}{2} \dim_\bR V = n\,$.
However, there is a crucial  difference  between complex and real $\gg$-structures
in the geometric interpretation of  the horizontal subbundle  $\gH \subset T^\bC P$
generated by  vector fields in $\a(V) \subset \gX^\bC(P)$.  
For any $\ u \in P$, we have $\dim_\bR \gH_u = \dim_\bR V = 2n$.
So the  subbundle  $\gH  \subset T^\bC P$ is isomorphic to $T^\bC M$ 
and admits no natural interpretation as a real horizontal distribution in  $TP$.
The curvature and torsion of a complex $\gg$-structure, $R(v,v')$ and $T(v,v')$,
thus have arguments $v,v' $ in $V \simeq T^\bC_u M$ rather than in $T_u M$. 
 
\subsubsection{Real forms of complex $\gg$-structures}
\noindent 
Let  $\tau\colon V \rightarrow V\,,\, V = \bC^n$,  be  a $\gg$-admissible 
$\bC$-antilinear involution.

\begin{Adef}\label{tau_compat} 
A  torsionless complex $ \gg$-structure $\a\colon  \gg + V\rightarrow \gX^\bC(P)$  is 
\linebreak  {\bss $\t$-compatible} around  $x_o \in P$  
if  there exists an n-dimensional submanifold $M \subset P$ containing $x_o$,  
such that, at every point  $y \in M$, the following hold:

\begin{itemize} [itemsep=4pt, leftmargin=18pt]
\item[i)] $T_y M$ is  transversal to $\a(\gg^\t)|_y \subset T_y P$
\item[ii)] $\a(v)_y = \overline{\a(v)_y}$ for all $v \in V^\t$
\item[iii)] $\a(V^\t)_{y} \subset T_y M + \a(\gg^\t)|_y\, $. 
\end{itemize}
\end{Adef}

\begin{Aprop} \label{realizations} 
Let   $\a\colon \gg +  V\rightarrow \gX^\bC(P)$  be a complete, torsionless 
complex $\gg$-structure and let  
$\k =  \o + \vartheta\colon T^\bC P \rightarrow \gg^\bC + V$  be such that 
$\a(X)|_u = \k^{-1}_u(X)$ for $u \in P\,,\, X \in \gg + V$.  
The following conditions are equivalent:
 \begin{itemize} [itemsep=4pt, leftmargin=18pt]
\item[a)] 
The complex $\gg$-structure $\a$ is $\tau$-compatible around $x_o  \in P$.

\item[b)] 
There exists a local diffeomorphism  
$\imath\colon \cU\rightarrow \cU'  \subset P'$  between a neighbourhood 
$\cU\subset P$ of $x_o$ and an open subset $\cU'$ of the complexification $P'$ 
of a complete (real) $G^\t$-structure $(\wt P, \wt \vartheta)$, such  that 
$\imath_*(\o) = \o'$ is a torsionless connection  on $\cU' \subset P'$,  
which is the complexification of  a   connection on  the real form $ \wt P\cap \cU'$. 
\end{itemize}
\end{Aprop} 

\begin{proof} 
We first  check that (a) implies (b). 
Consider a submanifold $M \subset P$ containing $x_o$ and satisfying 
conditions (i)-(iii) of Def.\ \Aref{tau_compat}. 
By (i), $M$ is transversal to the orbits of  the  action of the  real Lie group $G^\t$ 
determined by the flows of the fields in $\a(\gg^\tau)$, where $\gg^\t = \Lie(G^\t)$. 
By (ii), (iii) and the properties of the Lie brackets of  the fields in $\a(\gg^\t {+} V^\t)$,  
the  union of the $G^\t$-orbits of the points of $M$, denoted by 
$\cU^\t = M \cdot G^\t$,  is (locally) a smooth submanifold of $P$ 
passing through $x_o$ and  tangent to the fields of $\a(\gg^\t  {+} V^\t)$.  
In particular,  the restriction to $\cU^\t$ of the vector fields in $\a( \gg^\t {+} V^\t)$ 
determines a $\gg^\t$-structure. 
By  Prop.\  \Aref{g-G-str},  $\cU^\t$ is  locally diffeomorphic to an open set in a 
 $G^\t$-structure $\wt P$, which we may  assume, with no loss of generality,  to be 
 a Cartesian product $\wt P = M \times G^\t$.
 By construction, the restriction of $\o$ 
 to the tangent space $T \cU^\t$ is mapped into a torsionless connection on $\wt P$. 
Imposing equivariance under local action of  $G = (G^\t)^\bC$,  
the local diffeomorphism between $\cU^\t$ and $\wt P$ extends to a local diffeomorphism 
between an open set of the form $\cU = \cU^\t \cdot G$ and an open set of the 
complexification $P' \simeq M \times G$ of $\wt P$. This proves  (b).
The proof that (b) implies (a) follows  directly from  the definitions. 
\end{proof}

\subsection{Complex $G$-structures, pseudo-hyperk\"ahler metrics and hk-pairs}
\label{appendixHK}
\subsubsection{Pseudo-hyperk\"ahler metrics as $\Sp_{p,q}$-structures}  
\label{ps_hk}
\noindent 
A {\bss hypercomplex structure\/} on a $4n$-dimensional 
real vector space $W$ is a triple $(J_1, J_2, J_3)$ of endomorphisms of $W$ 
satisfying the multiplication relations of the imaginary quaternions,  
$J^2_\a = - \Id_W\,,\,J_\a J_\b =   J_\g$, 
for all  cyclic permutations  $(\a,\b,\g)$  of $(1,2,3)$. 
An  inner product $g$ on $W$ is called 
{\bss hermitian} with respect to  the  hypercomplex structure $(J_1, J_2, J_3)$ 
if every $J_\a$ is skew-symmetric with respect to $g$, i.e.\  
$g(J_\a w, w') + g(w,J_\a w') = 0$ for all $w,w' \in W$.

\begin{Adef}
{\rm A $4n$-dimensional pseudo-Riemannian manifold $(M, g)$ 
of signature  $(4p, 4q)$,  with  $p+q = n$,  is called {\bss pseudo-hyperk\"ahler\/}  
if it is endowed with a triple $(J_1, J_2, J_3)$ of global sections of $\End(TM)$ such that 
\begin{itemize} [itemsep=2pt, leftmargin=18pt]
\item[i)]    
$(J_1, J_2, J_3)_x$ is a hypercomplex structure on $T_xM$  for every $x \in M$
and $g_x$  is  hermitian with respect to it, and
\item[ii)]   
$\nabla J_\a = 0$ for $\a = 1,2,3$, where $\n$ is the Levi-Civita connection of the metric $g$.
\end{itemize}
Equivalently,  a   $4n$-dimensional pseudo-Riemannian manifold $(M, g)$ of signature 
$(4p, 4q)$  is pseudo-hyperk\"ahler if and only if its holonomy algebra $\mathfrak{hol}(M,g)$ 
is a subalgebra of $\gsp_{p,q}$\/}. 
This is equivalent  to requiring that the Levi-Civita connection on $M$  preserves   
an  $\Sp_{p,q}$-reduction  $Q \subset  \operatorname{O}_g(M)$ 
of the orthonormal frame bundle.
When $g$ is Riemannian, $(M, g)$ is called {\bss hyperk\"ahler\/}.
\end{Adef}
 \vskip 1pt plus 2pt minus 12pt
Let $(M, g)$ be a pseudo-hyperK\"ahler manifold and  $\vartheta$ the canonical 
soldering form of the orthonormal frame bundle $\operatorname{O}_g(M)$.  
The pair $(\pi\colon Q {\rightarrow} M\,,\,  \vartheta|_{TQ})$  is 
an $\Sp_{p,q}$-structure with a  unique  Levi-Civita  (torsionless)  
connection  $\o\colon TQ \rightarrow \gsp_{p,q}$.
The $\Sp_{p,q}$-structure with connection $\o$  is  uniquely associated with 
the manifold $(M, g)$, modulo principal bundles equivalences.  

Conversely (see Remark \Aref{G-remark}), every $\Sp_{p,q}$-structure 
$(\pi\colon Q {\rightarrow} M, \vartheta)$ 
with a torsionless connection $\o$ determines  a pseudo-hyperk\"ahler metric on $M$.
$Q$ can be identified with  an $\Sp_{p,q}$-reduction of $L(M)$. 
Further,  every (local) section  $\s\colon M \rightarrow Q$ determines a field of frames 
$(e_i)$ on $M$, together with a pseudo-Riemannian metric $g$ of signature $(4p, 4q)$, 
with respect to which the frames $(e_i)$ are orthonormal.  
Since $Q \subset L(M)$ is an $\Sp_{p,q}$-bundle, it follows  that
 $g$ is independent of the choice of section $\s$ in $Q$,
 it is pointwise hermitian with respect to a family of hypercomplex structures,
 $(J'_i|_x)_{i=1,2,3}\,,\,x\in M$,
and that the restriction of the Levi-Civita connection of $\operatorname{O}_g(M)$ to $Q$ 
is the torsionless connection $\o$.
Thus $g$ is pseudo-hyperk\"ahler and  $Q$ is a holonomy reduction of 
$\operatorname{O}_g(M)$.
We therefore have: 

\begin{Aprop}  \label{firstprop} 
There exists  a natural one-to-one correspondence between \linebreak 
pseudo-hyperk\"ahler metrics $g$ of signature $(4p, 4q)$ on a manifold $M$, 
up to isometries, and $\Sp_{p,q}$-structures $(\pi\colon Q{\ra} M \,,\, \vartheta)$
possessing a torsionless connection, up to principal bundle equivalences.
\end{Aprop}

An $\Sp_{p,q}$-structure $(\pi\colon Q{\ra} M\, ,\, \vartheta )$ with torsionless connection 
can be regarded locally as a real form of  a complex $\Sp_n(\bC)$-structure. 
It can also  be considered naturally as an $\Sp_{p,q}$-reduction of 
a real form of  a complex  $(\Sp_1(\bC) {\times} \Sp_n(\bC))$-structure
$(\pi\colon P{\ra} M\, ,\,  \vartheta )$.   The reason is the following:
Since   $(Q , \vartheta)$ is  locally   a  bundle  of  orthonormal frames of 
a pseudo-hyperk\"ahler manifold $(M, g)$, it  can  also  be
considered as an $\Sp_{p,q}$-reduction  of the 
$\Sp_1 {\times} \Sp_{p,q}$-subbundle of $ \Spin_{4p,4q}(M,g)$. 
Since this subbundle is  a  real form of its  complexification,  
the bundle $Q$ is in turn  naturally identifiable with an $ \Sp_{p,q}$-reduction of 
the complex $(\Sp_1(\bC) {\times} \Sp_n(\bC))$-structure $P$.  

The latter has the following geometrical interpretation.
Recall that  an  $\Sp_1 {\cdot}\Sp_{p,q}$-reduction of the linear frame bundle $L(M)$ 
is uniquely  associated with a (local)  isomorphism  $T^\bC M \simeq H \otimes_M E$ 
between $T^\bC M$ and  the tensor product of  two complex vector bundles 
$\,\pi^H\colon H {\rightarrow} M\,$ and  $\,\pi^E\colon E{\rightarrow} M$, with  fibres  given  
by  standard complex representations of $\Sp_1(\bC)$ and $\Sp_n(\bC)$, respectively  
(see e.g.\  \cite{Sa}). This (local) identification  allows us to consider  complex frames for  
$T^\bC_x M$  of  the form  $(h_i {\otimes} e_a)_{i = 1,2;  1 \leq a \leq 2n}$,  
where $(h_i)$ and $(e_a)$  are complex frames for  $H_x$ and $E_x$, respectively,  
adapted to the   standard symplectic forms of $H_x$ and $E_x$. 
The collection of all such complex frames  is an $\Sp_1(\bC) {\cdot} \Sp_n(\bC)$-reduction  
of the complex linear frame bundle $L^\bC(M)$, whose double cover is 
the complex  $(\Sp_1(\bC) {\times} \Sp_n(\bC))$-structure $P$. 

Now, by construction, the Levi-Civita connection of $(M,g)$ uniquely corresponds to
torsionless  connections on  $(Q, \vartheta)$ as well as on $(P , \vartheta )$. 
Since the latter connection  is an $\Sp_1(\bC) \times \Sp_n(\bC)$-equivariant 
extension of  the former it follows  that  its curvature 2-form 
necessarily takes  values only in $\gsp_n(\bC)$.
We may now  directly obtain the following:

\begin{Athm} \label{equivalences}
Let  $\t\colon\gg  \to \gg$  and  $\t\colon V \rightarrow V$ be the  anti-involutions of 
$\gg = \gsp_1(\bC) + \gsp_n(\bC)$ and $V= \bC^{4n}$ defined in Sect.\ref{real_structures}.
There is  a one-to-one correspondence between the following two sets of data, 
up to, respectively, local isometry and local equivalence:

\begin{itemize} [itemsep=4pt, leftmargin=18pt]
\item[i)]  Pseudo-hyperk\"ahler metrics  of signature $(4p,4q)$ over open subsets 
of  $\ W = \bR^{4n}$ 
\item[ii)] Torsionless complex $\gg$-structures $\a\colon \gg +  V\rightarrow \gX^\bC(\cU)$
on neighbourhoods   $\cU$ of the identity 
$e = (I_2, I_{2n}, 0)$ in   $P= \Sp_1(\bC) {\times} \Sp_n(\bC) {\ltimes} W$, $W = V^\t$,     
so that 
 \begin{itemize} [itemsep=4pt, leftmargin=18pt]
\item[a)]  
the  curvature   $R_u \in  \Hom(\L^2 V, \gg)\, ,\, u \in \cU$,  
takes values only in the $\gsp_n(\bC)$ part of $\gg$, the $\gsp_1(\bC)$ part being trivial and 
\item[b)]  there exists a submanifold   $\wh \cU \subset \cU$  containing  $e \in \cU$, 
tangent  to   the  distribution defined by 
\beq\label{dist}
\cD_u = \a(\gsp_n(\bC))_u + \Span_\bR\{\, \Re(X)_u \, ,\, X \in \a(V^\t )\, \}\,,
\eeq
of dimension $\dim\, \wh \ccU = \rank \cD$,  such that  the  map
\beq \label {beta} 
\b\colon  \gsp_n(\bC) + V \rightarrow \gX^\bC(\wh \cU)\ ,\quad 
\b(X) := \a(X)|_{\wh \ccU}\ ,
\eeq
is a complex $\gsp_n(\bC)$-structure,    $\t$-compatible  around $e$. 
\end{itemize}
\end{itemize}
\end{Athm}

\proof
By Prop.\ \Aref{firstprop},  a pseudo-hyperk\"ahler metric $g$ is naturally associated, 
up to local equivalences,  with a unique $\Sp_{p,q}$-structure with a torsionless 
connection.   The latter  is (locally) a reduction of a  real form of an 
$\Sp_1(\bC) {\times} \Sp_n(\bC)$-structure with a torsionless connection.
This real form corresponds to an associated  $\t$-compatible   
torsionless complex $\gg$-structure $\a$ (see Prop.\,\Aref{realizations}). 
This $\gg$-structure  satisfies the conditions  a) and b) by construction. 
Conversely,   if  a) and b) hold,  then $\a$ is associated, up to local equivalences, 
with an $\Sp_1(\bC) {\times} \Sp_n(C)$-structure  with a  real form admitting an 
$\Sp_{p,q}$-reduction corresponding to a pseudo-hyperk\"ahler metric $g$. 
\qed

\noindent
Complex $\gg$-structures  corresponding to pseudo-hyperk\"ahler metrics,
are said to be  {\bss reducible} to   $\gsp_{p,q}$-structures.

\subsubsection{$(\gsp_1(\bC){+}\gsp_n(\bC))$-structures and  hk-pairs} 
\label{hk_app} 

Consider a torsionless, complex $\gsp_1(\bC) {+} \gsp_n(\bC)$-structure 
$$
\alpha\colon  (\gsp_1(\bC) + \gsp_n(\bC)) + V\longrightarrow \gX^\bC(\cU') \,,\   V= W^\bC = \bC^{4n},
$$  
on some open neighbourhood    $\cU'$ of  
$e=(I_2, I_{2n}, 0) \in P  = \Sp_1(\bC) {\times} \Sp_n(\bC){\ltimes} W$,  $W\,{=}\,V^\t$, 
which is reducible  to an $\gsp_{p,q}$-structure. 
Let $ \cA_o = (H^o_0, H^o_{\pm\pm}, E^o_A, e^o_{\pm a} )$  be the  standard  basis of 
$\gp = \gsp_1(\bC) + \gsp_n(\bC) + V$ (see  Sect.\ \ref{basic}).
As discussed in Sect.\ \ref{cx.G-str}, $\a$ is completely determined by the 
set of complex vector fields
\begin{equation*}
\cA^{(\a)} = \left(e_{\pm a} = \a(e^o_{\pm a}) , 
                           H_0 = \a(H^o_0)^{10}, H_{\pm \pm} = \a(H^o_{\pm\pm})^{10},
                          E_A = \a(E^o_A)^{10}   \right)\,,
\end{equation*}
 which we call the {\bss  frame associated with $\a$}. 
The Lie brackets of the fields in $\cA^{(\a)}$ are  of the form 
\eqref{algebra}-\eqref{curv_constr} and it is therefore tempting to claim that
$\cA^{(\a)}$ is an hk-frame. 
Alas, this is not so (see Def.\ \ref{hkframe}), since the vector fields in $\cA^{(\a)}$ 
are not defined on an appropriate open neighbourhood $\cU \subset \cP$ of $e$,
but rather on  an open subset $\cU'$ of $P$,
a  codimension $n$ real  submanifold of $\cP = \Sp_1(\bC) \times \Sp_n(\bC)\ltimes V$. 
However, we   have:  

\begin{Alem} \label{appr_lem}
When  the data are real analytic and $\cU'$ is sufficiently small, there exists  
an  appropriate neighbourhood   $\ \cU \subset\cP\ $ of $\,e\,$, which contains $\cU'$, 
on which the vector fields in $\,\cA^{(\a)}$ admit  unique holomorphic extensions.\\ 
The set  
$\cA = (H_0, H_{\pm\pm}, E_A, e_{\pm a})$ of such holomorphic extensions on $\cU$ 
is a central hk-frame,    uniquely associated with $\a$ up to local equivalence,  such that  
 \begin{itemize} [itemsep=2pt, leftmargin=18pt]
\item[a)] 
the intersection $\left( \left(\{I_2\} \times \Sp_n(\bC)\right) \ltimes W \right)\cap \cU$  
is equal to a submanifold $\wh \cU$ of $\cU'$ as in Theorem \Aref{equivalences} b), 
i.e.  tangent at all points  the distribution  $\cD$  in \eqref{dist}   
\item[b)]  the pair $(\cA, M)$,  with  $M := \cU \cap \{I_2\}{\times}\{I_{2n}\}{\times} W  $, is an hk-pair. 
\end{itemize}
\end{Alem} 

\begin{Arem} 
The distribution $\cD$ can be  also described  as the  restriction 
 $\cD = \wt \cD|_{\cU'}$ of  the  real distribution  $\wt \cD \subset T\cU$,   generated by   
 an appropriate set of real and imaginary parts of  vector fields in  $\cA$, 
 namely by the real and imaginary parts of $E_A$, together with    
 the  vector fields  $\Re(\a(w))\,,\, w \in V^\t$.
 \end{Arem}

Before  proving the lemma, it is convenient to review  the notion of 
{\em (local) holomorphic extensions} of real analytic complex vector fields.

Let  $(N, J)$ be a complex manifold,  $\dim_\bC N = m$ and 
$T^\bC N = T^{10} N \oplus T^{01} N$ the decomposition in
holomorphic and anti-holomorphic tangent bundles.  
For any chart of holomorphic complex coordinates 
$\xi  = (\z^1, \dots, \z^m)\colon \cU \subset N \rightarrow \bC^m$
 and  $x \in \cU$, we have: 
\[ 
T^{10}_x N = \Span_\bC\left \{\ \left.\frac{\partial}{\partial \z^i}\right|_x\ \right\}
\ ,\quad 
T^{01}_x N = \Span_\bC \left\{\ \left.\frac{\partial}{\partial\overline{\z^i}}\right|_x\ \right\}\ .
\]
Given an open subset $\cV \subset N$ and standard coordinates $(z^i, w^j)$ of $\bC^{2m}$,  
a   {\bss complexification} of  $\cV$  
is a  real analytic embedding $\imath\colon \cV \subset N\rightarrow   \bC^{2m}$,  
satisfying

\begin{itemize} [itemsep=8pt, leftmargin=18pt]
\item[1)] 
$ \imath(\cV) \subset \{\ (z^i, w^i)\ \colon w^i - \overline{z^i} = 0\ \}\subset \bC^{2m}\ $ and
\item[2)] 
$ \imath_*\left(T^{10}_x N|_{\cV}\right) 
   = \Span_\bC\!\!\left\{\left.\frac{\partial}{\partial z^i}\right|_{\imath(z)} 
   +\left. \frac{\partial}{\partial \overline{w^i}}\right|_{\imath(z)},\ 1 \leq i \leq m\right\}$  
for all $x \in \cV$.
\end{itemize}
Complexifications are easily constructed  if  holomorphic complex coordinates  
$\xi = (\z^i)$  on $\cV$ exist. Namely, it suffices to consider  the  embedding   
$$
\imath\colon \cV \subset N\longrightarrow \bC^{2m}\ ,\qquad 
\imath(x) = (\z^1(x) , \dots, \z^m(x), \overline{\z^1}(x), \dots, \overline{\z^m}(x))\ .
$$
Consider now  a real analytic complex vector field  on $\cV$
$$
X = X^i(\zeta^k, \bar \zeta^\ell) \frac{\partial}{\partial \zeta^i} 
       + \cX^{j}(\zeta^k, \bar \zeta^\ell) \frac{\partial}{\partial \overline{\zeta^j}}
$$
and   identify  $X$ with  the field $\imath_*(X) \in T ^\bC \bC^{2m}|_{\imath(\cV)}$ on 
$\imath(\cV)$.  This vector field  immediately  extends   to a holomorphic vector field 
$ X^{hol}$, defined on an  open  neighbourhood  $\cW \subset \bC^{2m}$ of 
$\imath(\cV)$  by setting 
\beq 
\left.X^{hol}\right|_{(z^k,w^\ell)} 
                  := X^i(z^k, w^\ell) \frac{\partial}{\partial z^i} 
                     + \cX^{j}(z^k, w^\ell) \frac{\partial}{\partial w^j}\ ,
\eeq
i.e.\ replacing  the  (dependent) complex coordinates $\z^k$ and $\bar \z^\ell$  of $\cV$
by the independent variables $z^k$ and $w^\ell$ of $\bC^{2m}$. 
The resulting holomorphic vector fields are  called  {\bss (local) holomorphic extensions} 
of real analytic vector fields.  

We now proceed to the missing proof:

\noindent
{\em Proof of Lemma \Aref{appr_lem}.}
Consider the distribution   $\cD \subset T\cU'$ and the  submanifold $\wh \cU\subset \cU'$,
tangent to the distribution $\cD$, described  in  Theorem \Aref{equivalences}.   
From  \eqref {beta} and the hypotheses on the Lie brackets,  it follows that $\cU'$ 
is foliated by submanifolds of the form $\wh \cU {\cdot} g$  determined  by images of   
$\wh \cU$ under   the local action of the elements $g  \in \Sp_1(\bC) \times \Sp_n(\bC)$. 
By construction,  all such   submanifolds are  integral leaves of $\cD$, 
proving that the distribution  $\cD$  is indeed integrable. 
Consider  the complex distribution $\cD^{10}$ defined  by 
 $\cD^{10}_x = \Span_\bC\{\ E_A|_x, e_{+ \a}|_x\ \}$ for  $x \in \cU'$.
Recall  that $\ \cD^\bC_x = \cD^{10}_x  +  \cD^{01}_x$ where  
$\cD^{01}_x:=\overline{ \cD^{10}_x}\,$.
 
For  $y \in  \cU'$,   we  denote by $\gF_y$ the integral leaf  of  $\cD$ passing through  $y$. 
Since the fields $(E_A + \overline{E_A})|_{\gF_y}$ generate a Lie algebra of 
real vector fields  isomorphic to $\gsp_n(\bC)$,  they determine  a local right action of 
$\Sp_n(\bC)$ on $\gF_y$.
Moreover, the  complementary subbundles  $ \cD^{10}|_{\gF_y}\,,\, \cD^{01}|_{\gF_y}$ 
of $T^\bC \gF_y$ are  involutive and $\Sp_n(\bC)$-invariant.  Therefore, there exists 
a unique $ \Sp_n(\bC)$-invariant integrable complex structure $J_y$  on  $\gF_y $, 
which has the subbundles $ \cD^{10}|_{\gF_y} \,,\,  \cD^{01}|_{\gF_y}$ as associated 
holomorphic and anti-holomorphic distributions.  By $\Sp_n(\bC)$-invariance, such  a 
complex structure $J_y$ naturally projects onto a complex structure $\wt J_y$ 
on the  quotient $\wt \gF_y = \gF_y/\Sp_n(\bC)$. 
We may identify the pair $(\wt \gF_y, \wt J_y)$, without loss of generality,  with 
an open neighbourhood  $\wt \gF_y$ of $0$ in  $ (\bC^{2n})^\t \simeq \bR^{4n}$, 
endowed with an appropriate   complex structure $\wt J_y$.

If $\cU' \subset P$  is sufficiently small 
we may always assume that the following conditions are satisfied: 

\begin{enumerate} [itemsep=2pt, leftmargin=18pt]
\item[a)] 
All   integral leaves   of $\cD$ are  transversal to  the orbits of the local right action of 
$\Sp_1(\bC)$ generated by real vector fields   in  
$$
\Span_\bR\{\Re(H_0), \Im(H_0), \Re(H_{\pm\pm}), \Im(H_{\pm\pm})\} \simeq \gsp_1(\bC)\,.
$$

\item[b)] 
The quotients $\wt \gF_y  =  \gF_y/\Sp_n(\bC) \subset \bC^{2n} $ are all diffeomorphic 
to  a fixed suitable open subset  $\wt \gF$ of $(\bC^{4n})^\t \simeq \bR^{4n}$.  
Thus, $\wt \gF \subset \bR^{4n}$  is  equipped with a family integrable complex structures 
$\{\wt J_y\ ,\ y \in \cU'\ \}$, these being the push-forwards
of the complex structures of the leaves $\wt \gF_y\,,\, y\in \cU'$. 

\item[c)]  Any    leaf  $\gF_y$  admits a holomorphic trivialisation  
$$
\varphi_y\colon (\gF_y\,,\, J_y)\ \rightarrow\  (\Sp_n(\bC) \times \wt \gF\,,\,  J_{o} +\wt J_y )\ ,
$$ 
where   $J_{o}$ is the standard complex structure of $\Sp_n(\bC)$.
 
\item[d)]  
For all $y \in \cU'$, the complex manifold  
$(\wt \gF, \wt J_y)\subset (\bR^{4n}, \wt J_y) \simeq \bC^{2n}$ admits 
a complexification  $\ \imath_y\colon \wt \gF   \rightarrow V = \bC^{4n}$, 
which together with the trivialisation  $\varphi_y$,  determines  a real analytic 
$\Sp_n(\bC)$-equivariant embedding 
$$
\imath_y\colon \gF_y\ \rightarrow\   (\{I_2\} \times \Sp_n(\bC)) \ltimes  V\,  ,\
\ y \mapsto (I_2, I_{2n}, z^i(y)) \in \cP \,.
$$
Using  this embedding,  the vector fields    $(E_A, e_{\pm a})|_{\gF_y}$   
extend  holomorphically to an open neighbourhood  
$\cU_y  \subset (\{I_2\} {\times} \Sp_n(\bC)) {\ltimes} V$ 
of  $\imath(\gF_y) $.   We may choose 
the map $\imath_y$  so that $\imath_* (E_A) = E^o_A$.

\item[e)]   
Given a submanifold $M \subset \wh \cU$  satisfying the conditions of 
Def.\,\Aref{tau_compat},  the maps $\imath_y\,,\, y \in M$,  combine  to  determine  
an  $\Sp_1(\bC) {\times}  \Sp_n(\bC)$-equivariant real analytic embedding of $\cU'$ into  an 
appropriate open neighbourhood  $\cU$ of $e  \in  \cP$, 
\beq
  \imath\colon \cU' 
  = \!\!\!\! \bigcup_{y \in M {\cdot}  \Sp_1(\bC)} \!\!\!\!\!  \gF_y\  
                  \longrightarrow\ \cU \subset \cP =  (\Sp_1(\bC) \times  \Sp_n(\bC))\ltimes V\,.
\eeq
 This embedding  can be constructed to map the points  $y\in M$ into points of 
 $M^{o} = \{I_2\}{\times}\{I_{2n}\}{\times}(\bC^{4n})^\t$ and the complex vector fields  
 $H_0\,, H_{\pm\pm}\,, E_A\,, e_{\pm a}$ of $\cU'$  into  complex vector fields of 
 $\imath(\cU') \subset \cU$, which extend holomorphically   to all of  $\cU$.
The equivariant  embedding $\imath$ can also be  constructed so that the holomorphic 
extensions of  $\imath_*(H_0) , \imath_*(H_{\pm\pm}), \imath_*(E_A)$ are 
$H_0^o, H^o_{\pm\pm}, E_A^o$, respectively. 
\end{enumerate}
By construction, the pair $(\cA, M = M^{o})$,  formed by  the collection $\cA$ of the above holomorphic extensions of the vector fields in $\cA^{(\a)}$  and  the manifold $M^{o}= \imath(M)$  is a central hk-pair and it  is uniquely determined by $\a$ up to local equivalences.  
\qed

By Theorem \Aref{equivalences} and Lemma \Aref{appr_lem},  we may associate  
an hk-pair $(\cA, M)$ with  every real analytic pseudo-hyperk\"ahler manifold $(M, g)$. 
In the next section, we show that, up to local  equivalences,  this correspondence is invertible, providing a  bijection  between  local 
isometry classes of  real analytic pseudo-hyperk\"ahler manifolds  and
 local equivalence classes of hk-pairs.

\subsection{Inverse map  between hk-pairs and pseudo-hyperk\"ahler metrics}
\label{bijection}
\noindent 
Consider   the pseudo-hyperk\"ahler metric $g$ determined  by 
an  hk-pair $(\cA, M)$ and a section $\s\colon M \rightarrow \cU^{(\Sp_{p,q})}$ 
(see Sect.\ \ref{hkpairs-metrics}).  We now  prove
 that $g$ is  uniquely associated with $(\cA, M)$.

\begin{Alem}\label{sigma-indep}
The   metric \eqref{metric}, constructed from an hk-pair   $(\cA, M)$  of signature $(4p,4q)$  on  an appropriate open subset
$\cU \subset \cP$,  is independent of the choice of  section  
$\s\colon M \rightarrow \cU^{(\Sp_{p,q})}$ and  is a  real analytic
pseudo-hyperk\"ahler metric  of signature $(4p, 4q)$. 
\end{Alem}

\begin{proof}
Let  $\a\colon \gg + V \rightarrow \gX^{\bC}(\cU \cap P)$ be the $\bR$-linear map
defined by 
$$
\a(X) = \left.\Re(\a^\cA(X))\right|_{\cU \cap P}  \,,\quad 
\a(v)  = \left.\a^{\cA}(v)\right|_{\cU \cap P}
$$
for  $X \in \gg$ and $v \in V = \bC^{4n}$, 
where $\a^\cA$ is the absolute hk-parallelism associated with $\cA$ (see Sect.\ \ref{hk_frames}). 
By construction and the assumptions on $M$, the map $\a$ is a $\t$-compatible, complex 
$\Sp_1(\bC) \times \Sp_n(\bC)$-structure on $\cU \cap P$.
From  the proof of Prop.\ \Aref{realizations}, it follows that 
$M \cdot (\Sp_1(\bC) \times \Sp_{2n}(\bC)) \subset \cU$
is  an open subset of the complexification of an $\Sp_1 \times \Sp_{p,q}$\!-structure $\wt P$ 
over (an open subset of) $M$ and that  the set $\cU^{(\Sp_{p,q})} = M \cdot \Sp_{p,q}$, defined in \eqref{USP},  is an $\Sp_{p,q}$-reduction of an $\Sp_1 \times \Sp_{p,q}$-invariant 
open subset of such $\Sp_1 \times \Sp_{p,q}$-structure. 

The conditions on the curvature imply that $\wt P$   can be identified with a double covering of an 
$\Sp_1 \cdot \Sp_{p,q}$-reduction of $ L(M)$,  
admitting a further reduction to an $\Sp_{p,q}$-bundle $Q  \subset L(M)$.  
Thus the set $\cU^{\Sp_{p,q}}$ is identifiable with an $\Sp_{p,q}$-invariant open subset of $Q$. 
As explained in Sect.\ \ref{ps_hk},  we therefore have that: 
\begin{itemize}[itemsep=2pt, leftmargin=18pt]
\item[i)]  The bundles $\pi^Q\colon Q\to M$ and  $\pi^{\wt P}\colon \wt P\to M$  are  formed by 
linear frames which are orthonormal with respect to a pseudo-Riemannian metric $g$
and pointwise hermitian with respect to a family of hypercomplex structures 
$(J_i|_x)_{i=1,2,3}\,, x\in M$.
 
\item[ii)]  
The fields $ e^\t_{I}  = \a^\cA_{(\bR)}(e^{o\t}_{I})$  are identifiable with (local) vector fields on $Q$,
horizontal with respect to the Levi-Civita connection of $g$ and  
$\vartheta(e^\t_{I}) =  e^{o\t}_{I}$.
\end{itemize}
 It follows that $g$ is pseudo-hyperk\"ahler and that the frame fields  
 $e^{(\s)}_{I} = \pi_*(e^\t_{I})$  are orthonormal with respect to $g$
regardless of the choice of the local section 
 $\s\colon  M \rightarrow \cU^{(\Sp_{p,q})} $.
\end{proof}

We are finally in a position to prove the result quoted in Sect.\ \ref{hkpairs-metrics}.
\begin{Athm} \label{loc.isom}
 Every real analytic pseudo-hyperk\"ahler manifold of signature $(4p,4q)$ is 
locally isometric to a pseudo-hyperk\"ahler manifold $(M, g)$ 
corresponding to an hk-pair of signature $(4p,4q)$, with metric given by
eq.\ \eqref{metric}.  
\end{Athm}

\begin{proof}
By  the proof of Theorem \Aref{equivalences}, every pseudo-hyperk\"ahler metric $g$ of 
signature $(4p, 4q)$ on an open subset  $M \subset \bR^{4n} $ determines a bundle 
$\pi\colon P \rightarrow M$,  the complexification of  (a double covering of) an 
$\Sp_1 {\cdot} \Sp_{p,q}$-bundle $(\wt P,  \wt \q)$ of orthonormal frames of $(M, g)$. 
The Levi-Civita connection and the soldering form of $\wt P$ determine a complex 
$\gsp_1(\bC) {\times} \gsp_n(\bC)$-structure  on $P$, which is reducible to an $\gsp_{p,q}$-structure. 
This $\gsp_1(\bC) {\times} \gsp_n(\bC)$-structure is uniquely associated (see Sect.\ 4.3) 
with an  hk-pair $(\cA, M')$, where $\cA$ is  a central hk-frame on an 
appropriate open subset $\cU \subset \cP$  and  $M' = \cU \cap M^{o}$.  

The claim is proved if we can show that  the pseudo-hyperk\"ahler metric $g$ on 
$M \subset \bR^{4n}$  coincides (modulo identifications) with the metric on $M' \simeq M$, 
associated with the hk-pair $(\cA, M')$, i.e.\ the metric defined in eq.\ \eqref{metric}.
For this, it suffices to  observe  that, by construction,   the real submanifold  
$\cU^{(\Sp_{p,q})} = M' {\cdot} \Sp_{p,q} $ of $\cU$  considered in  \eqref{USP}  coincides 
with  the bundle of orthonormal frames $\wt P$ over $M (\simeq M')$,  
so that  the vector fields 
$e^\t_{I}|_{\cU^{(\Sp_{p,q})}} = \a^\cA_{(\bR)}(e^{o\t}_{I})|_{\cU^{(\Sp_{p,q})}}$  
are horizontal  with respect to the Levi-Civita connection 
and   satisfy the equation $\vartheta_u(e^\t_a) = e^{o\t}_a \in V^\t$ for any $u \in \wt P$. 
Hence, the projections   of the vectors $\vartheta_x(e^\t_a)|_u$  onto the points 
$$x = \pi(u) \in M' = 
\cU^\t/\Sp_1 \times \Sp_{p,q} \simeq \wt P/\Sp_1 {\cdot} \Sp_{p,q} = M$$
constitute  $g$-orthonormal frames and the metric \eqref{metric}  is  necessarily 
equal to the metric $g$. 
\end{proof}


\vskip 1.5truecm
\hbox{\parindent=0pt\parskip=0pt
\vbox{\baselineskip 9.5 pt \hsize=3.1truein
\obeylines
{\smallsmc
Chandrashekar Devchand
Max-Planck-Institut für Gravitationsphysik 
(Albert-Einstein-Institut)
Am Mühlenberg 1 
D-14476 Potsdam 
Germany
}\medskip
{\smallit E-mail}\/: {\smalltt devchand@math.uni-potsdam.de
}
}
\hskip  7truemm
\vbox{\baselineskip 9.5 pt \hsize=3.7truein
\obeylines
{\smallsmc
Andrea Spiro
Scuola di Scienze e Tecnologie
Universit\`a di Camerino
Via Madonna delle Carceri
I-62032 Camerino (Macerata)
Italy
}\medskip
{\smallit E-mail}\/: {\smalltt andrea.spiro@unicam.it}
}
}
\end{document}